\documentclass[onefignum,onetabnum]{siamart171218}

\usepackage{braket,amsfonts}
\usepackage{mathtools}

\definecolor{shadethmcolor}{rgb}{1,0.871,0.890}
\definecolor{shaderulecolor}{rgb}{0.651,0.074,0.090}
\usepackage{array}
\usepackage{dsfont}
\usepackage{hhline}
\usepackage[mathscr]{euscript}
\usepackage{textcomp}
\usepackage{amssymb} 
\usepackage[toc,page]{appendix} 
\usepackage[a4paper]{geometry}
\usepackage[caption=false]{subfig}
\usepackage{pgfplots}

\usepackage{tikz}
\usepackage{pgfplots}
\usepackage{subfig}
\usetikzlibrary{matrix,external}
\usepackage{mwe}

\newsiamthm{claim}{Claim}
\newsiamthm{remark}{Remark}
\newsiamthm{prop}{Proposition}
\newsiamthm{assumption}{Assumption}
\newsiamremark{hypothesis}{Hypothesis}
\crefname{hypothesis}{Hypothesis}{Hypotheses}

\usepackage{algorithmic}

\usepackage{graphicx,epstopdf}

\Crefname{ALC@unique}{Line}{Lines}

\usepackage[sort,nocompress]{cite}

\usepackage{amsopn}

\newcommand\B{\mathcal{B}}

\newcommand\C{\mathcal{C}}

\newcommand\Y{\mathcal Y}

\newcommand\dsp{\displaystyle}

\newcommand\oo{\mathcal O}

\newcommand{\vphi}{\varphi}

\hypersetup{
	colorlinks,
	citecolor=red,
	filecolor=black,
	linkcolor=blue, 
	urlcolor=green!50!black,
}

\def\dx{\,\textnormal{d}x}
\def\dt{\textnormal{d}t}

\title{Insensitizing control problems for the stabilized  Kuramoto-Sivashinsky  system}

\author{Kuntal Bhandari\thanks{Institute of Mathematics of the Czech Academy of Sciences, \v{Z}itn\'a 25, 11567 Praha 1, Czech Republic (email: bhandari@math.cas.cz).}
	\and V\'ictor Hern\'andez-Santamar\'ia\thanks{Instituto de Matem\'aticas, Universidad Nacional Aut\'onoma de M\'exico, Circuito Exterior C.U., C.P.04510 CDMX, M\'exico (email: victor.santamaria@im.unam.mx).}}

\headers{Insensitizing control problems for the stabilized  KS  system}{K. Bhandari,   V. Hern\'andez-Santamar\'ia}

\begin{document}

	\maketitle

	\begin{abstract}
In this work, we address the existence of insensitizing controls for a nonlinear coupled system of fourth- and second-order parabolic equations known as the stabilized Kuramoto-Sivashinsky model. The main idea is to look for controls such that some functional of the state (the so-called sentinel) is locally insensitive to the perturbations of the initial data. Since the underlying model is coupled, we shall consider a sentinel in which we may observe one or two components of the system in a localized observation set. By some classical arguments, the insensitizing problem can be reduced to a null-controllability one for a cascade system where the number of equations is doubled. Upon linearization, the null-controllability for this new system is studied by means of Carleman estimates but unlike other insensitizing problems for scalar models, the election of the Carleman tools and the overall control strategy depends on the initial choice of the sentinel due to the (lack of) couplings arising in the extended system. Finally, the local null-controllability of the extended (nonlinear) system (and thus the insensitizing property) is obtained by applying a local inversion theorem. 
	\end{abstract}

	\begin{keywords}
		Parabolic system, insensitizing controls, Carleman inequalities, observability, inverse mapping theorem. 
		\end{keywords}

	\begin{AMS} 
		35K52 - 93B05 - 93B07 - 93B35.
	\end{AMS} 
	
	\section{Introduction and general setting}\label{section-introduction}
	
	\subsection{Problem statement and bibliographic comments}


Let $T>0$ and $\omega \subset (0,1)$ be any nonempty open set. 
	We   introduce the space $Q_T: = (0,T) \times (0,1)$ and
	 its boundary $\Sigma_T:=(0,T) \times  \{ 0,1\}$. We also  take  another nonempty open set  $\oo \subset (0,1)$  hereinafter referred as the observation set.

Let us consider the following control system with incomplete data
\begin{align}\label{system-main}
	\begin{dcases}
		y_t + y_{xxxx} +\gamma y_{xx} + yy_x = z_x +  h_1\mathds{1}_{\omega}+\xi_1  & \text{in } Q_T, \\
		z_t - z_{xx} + \beta z_x =  y_x + h_2 \mathds{1}_{\omega}   + \xi_2 & \text{in } Q_T,\\
		y=y_x=z=0        &\text{in } \Sigma_T, \\
		y(0)=y_0+\tau \bar y_0, \ \ z(0) =z_0+\tau \bar z_0  &\text{in } (0,1),
 	\end{dcases}
	\end{align}
where $\gamma>0$ and $\beta$ is any real number. 

In \eqref{system-main}, $y=y(t,x)$ and $z=z(t,x)$ are the state variables, $h_i=h_i(t,x)$, $i=1,2$ are control functions acting on the control set $\omega$, $\xi_i=\xi_i(t,x)$, $i=1,2$, are given external source terms and the initial state $(y(0),z(0))$ is partially unknown in the following sense:

\smallskip 

\begin{itemize}
\item $(y_0,z_0)\in [L^{2}(0,1)]^2$ are given,
\item $(\bar y_0,\bar z_0)\in [L^{2}(0,1)]^2$ are unknown and satisfy $\|\bar y_0\|_{L^2(0,1)}=\|\bar z_0\|_{L^2(0,1)}=1$. They represent some {\em uncertainty} on the initial data. 

\item $\tau\in\mathbb R$ is unknown and small enough. 
\end{itemize}

From the modeling point of view, when $h_i\equiv\xi_i\equiv 0$, $i=1,2$, and $\tau=0$, system \eqref{system-main} is the so-called stabilized Kuramoto-Sivashinsky system, which was proposed in \cite{FMK03} as a model of front propagation in reaction-diffusion phenomena and combines dissipative features with dispersive ones.

\smallskip 

In this work, our main goal is to study an insensitizing control problem for \eqref{system-main}. This problem, originally introduced by J.-L. Lions in \cite{Lio90}, can be stated as follows: we observe the solution of system \eqref{system-main} through a functional $J_\tau$ (the so-called \textit{sentinel}) defined on the set of solutions to \eqref{system-main}, which is in this case given by
\begin{equation}\label{eq:sentinel}
J_\tau(y,z)=\frac{\alpha}{2}\iint_{\mathcal O\times(0,T)} |y|^2\dx\dt+\frac{1-\alpha}{2}\iint_{\mathcal O\times(0,T)}|z|^2\dx\dt, \quad \alpha\in[0,1].
\end{equation}
Then, the insensitizing control problem is to find controls $(h_1,h_2)$ such that the uncertainty in the initial data does not affect the measurement $J_\tau$, that is
\begin{equation}\label{eq:ins_satisfy}
\left.\frac{\partial J_\tau(y,z)}{\partial \tau}\right|_{\tau=0}=0 \quad \forall (\bar y_0,\bar z_0)\in [L^{2}(0,1)]^2 \  \textnormal{with } \|\bar y_0\|_{L^{2}(0,1)}=\|\bar z_0\|_{L^{2}(0,1)} = 1.
\end{equation}
When \eqref{eq:ins_satisfy} holds, the sentinel $J_\tau$ is said to be {\em locally insensitive} to the perturbations of the initial data. In other words, \eqref{eq:ins_satisfy} indicates that
the sentinel does not detect the variations
of the initial data $(y_0,z_0)$  by the  unknown small perturbation $\tau (\overline y_0, \overline z_0)$ in the observation domain $\mathcal O$.

\smallskip

The parameter $\alpha$ has been introduced  in \eqref{eq:sentinel} to take into account the contribution of each state variable in the sentinel. Note that the insensitivity condition \eqref{eq:ins_satisfy} should be satisfied for any perturbation of the initial datum of both components, hence removing one observation in \eqref{eq:sentinel} (i.e. taking $\alpha=0$ or $\alpha=1$) reduces the information available and the problem becomes more interesting. Indeed, we will see later that additional difficulties arise in either of those cases and for that reason we will mainly focus on them. 

\smallskip 

The first results concerning the existence of insensitizing controls were obtained for linear and semilinear heat equations in \cite{BF95,deT00}. After that, many works have been devoted to study the insensitizing problem from different perspectives: in \cite{BGBP04,BGP04b,BGP04c}, the authors study such problem for linear and semilinear heat equations with different types of nonlinearities and/or boundary conditions, while in \cite{Sergio} the problem of insensitizing a sentinel depending on the gradient of the solution of a linear parabolic equation is addressed. For insensitizing problems of equations in fluid mechanics we refer to the works \cite{Gue13,CG14,CGG15,Cn17} and for a phase field system to \cite{CCC16}. Most recently, the insensitizing control problem has been addressed from a numerical point of view in \cite{BHSdeT19}, for fourth-order parabolic equations in \cite{Kas20} and with respect to shape variations in \cite{LPS19,ELP20}. We also mention a recent work \cite{T20} where the authors studied the  insensitizing  controls for
the fourth-order dispersive nonlinear Schr\"odinger equation with cubic nonlinearity. 

\subsection{Main results}
Our aim is  to prove the existence of control functions $(h_1, h_2)$ which insensitizes the functional $J_\tau$  given by \eqref{eq:sentinel}.
In this spirit, our  control result is the following.
\begin{theorem}\label{thm:main}
	Assume that $\oo \cap \omega \neq \emptyset$ and $y_0\equiv z_0\equiv 0$.  Then for any $\alpha \in [0,1]$, there exist constants $C>0$ and $\delta>0$  
	such that for any $(\xi_1,\xi_2)\in [L^2(Q_T)]^2$ verifying 
	\begin{equation}\label{eq:cond_rho-2}
		\|e^{C/t}(\xi_1, \xi_2)\|_{[L^2(Q_T)]^2} \leq \delta 
	\end{equation}
	one can prove the existence of  some controls $(h_1,h_2)\in [L^2((0,T)\times \omega)]^2$ which insensitize the functional $J_\tau$ in the sense of \eqref{eq:ins_satisfy}. 
\end{theorem}

To prove the above theorem, we shall equivalently prove another result given by \Cref{thm:main_extended} below.  In fact, following well-known arguments (see e.g. \cite[Proposition 1]{BF95} or \cite[Appendix]{deTZ09}), it can be proved that the insensitivity condition \eqref{eq:ins_satisfy} is equivalent to a null-control problem for an extended system. More precisely, we have the following.         
\begin{prop}\label{Prop-primal}
Consider the  extended system 
\begin{align}\label{sys_equiv_1}
	&\begin{dcases}
			 y_t + 	 y_{xxxx} +\gamma 	 y_{xx} + yy_x  = 	 z_x +  h_1\mathds{1}_{\omega}+\xi_1  & \textnormal{in } Q_T, \\
			 z_t - 	 z_{xx} + \beta 	 z_x = 	 y_x + h_2 \mathds{1}_{\omega}   + \xi_2 & \textnormal{in } Q_T,\\
			y= 	 y_x= 	 z=0        &\textnormal{in } \Sigma_T, \\
		 y(0)= y_0, \ \ 	 z(0) = z_0  &\textnormal{in } (0,1),
 	\end{dcases} \\ \label{sys_equiv_2}
	&\begin{dcases}
		-p_t + p_{xxxx} +\gamma p_{xx} -yp_x = -q_x + \alpha y\mathds{1}_{\mathcal O}  & \textnormal{in } Q_T, \\
		-q_t - q_{xx} - \beta q_x =  -p_x + (1-\alpha) 	 z \mathds{1}_{\mathcal O} & \textnormal{in } Q_T,\\
		p=p_x=q=0        &\textnormal{in } \Sigma_T, \\
		p(T)=0, \ \ q(T) =0  &\textnormal{in } (0,1) .
 	\end{dcases}
\end{align}
Then, the controls $(h_1,h_2)$ verify the insensitivity condition \eqref{eq:ins_satisfy} for the sentinel \eqref{eq:sentinel} if and only if the associated solution to \eqref{sys_equiv_1}--\eqref{sys_equiv_2} satisfies
\begin{equation}\label{eq:null_cascade}
(p(0),q(0))=(0,0) \quad \textnormal{in } (0,1). 
\end{equation}
\end{prop}

In view of this result, in what follows we only focus on studying controllability properties for the extended system \eqref{sys_equiv_1}--\eqref{sys_equiv_2}. Indeed, we prove the following theorem which is the main result of our paper. 
\begin{theorem}[Local null-controllability of the extended system]\label{thm:main_extended}
	Assume that $\oo \cap \omega \neq \emptyset$ and $y_0\equiv z_0\equiv 0$.  Then for any $\alpha \in [0,1]$, there exist constants $C>0$ and $\delta>0$  
	such that for any $(\xi_1,\xi_2)\in [L^2(Q_T)]^2$ verifying 
	\begin{equation}\label{eq:cond_rho}
		\|e^{C/t}(\xi_1, \xi_2)\|_{[L^2(Q_T)]^2} \leq \delta ,
	\end{equation}
there exist controls $(h_1,h_2)\in [L^2((0,T)\times \omega)]^2$ such that the solution $(y,z,p,q)$ to \eqref{sys_equiv_1}--\eqref{sys_equiv_2} satisfies $p(0)=q(0)=0$ in $(0,1)$.   
\end{theorem}

We note that the controls $(h_1, h_2)$ act indirectly on the state $(p,q)$ by means of the couplings terms exerted on the observation set $\mathcal O$, that is, we have more equations than controls. As it has been pointed out in \cite{AKBGBdeT11}, this situation is more complicated than controlling scalar systems and, as we have said before, the introduction of the parameter $\alpha$ introduces an additional difficulty. Note that when $\alpha=0$ or $\alpha=1$ one of the couplings in system \eqref{sys_equiv_2} is removed and the action of the controls $(h_1,h_2)$ enters indirectly on the backward system only through one coupling term. As we will see later, this translates into using different Carleman tools for studying the observability of the corresponding adjoint system and establishing the controllability of \eqref{sys_equiv_1}--\eqref{sys_equiv_2}.

\begin{remark} Some remarks are in order. 
\begin{itemize} 
\item  As in other insensitizing problems, the assumption on the zero initial condition is roughly related to the fact that system \eqref{sys_equiv_1}--\eqref{sys_equiv_2} is composed by forward and backward equations. As noticed in the work \cite{deTZ09}, even for the simple heat equation is not an easy task to characterize the space of initial datums that can be insensitized; see also 
\cite{deT00}. 

\smallskip 

\item  The assumption $\mathcal O\cap \omega\neq \emptyset$ is essential to prove an observability inequality (see \Cref{Obser-ineq} for instance), which is the main ingredient in the proof of \Cref{thm:main_extended}. Notwithstanding, in \cite{KdeT10}, the authors have proved that in the simpler case of the heat equation this condition is not necessary if one considers an $\epsilon$-insensitizing problem (i.e., $\left|\frac{\partial J_\tau(y,z)}{\partial \tau}|_{\tau=0}\right|\leq \epsilon$). In our case, this remains as an open question. 
\end{itemize}
\end{remark}

  To prove Theorem \ref{thm:main_extended}, we shall first prove a null-controllability result for the following linearized (around zero) model associated to \eqref{sys_equiv_1}--\eqref{sys_equiv_2},
	\begin{align}\label{sys_linear-1}
		&\begin{dcases}
			y_t + 	 y_{xxxx} +\gamma 	 y_{xx}  = 	 z_x +  h_1\mathds{1}_{\omega} + f_1  & \textnormal{in } Q_T, \\
			z_t - 	 z_{xx} + \beta 	 z_x = 	 y_x + h_2 \mathds{1}_{\omega} + f_2   & \textnormal{in } Q_T,\\
			y= 	 y_x= 	 z=0        &\textnormal{in } \Sigma_T, \\
			y(0)= y_0, \ \ 	 z(0) = z_0  &\textnormal{in } (0,1),
		\end{dcases} \\ \label{sys_linear-2}
		&\begin{dcases}
			-p_t + p_{xxxx} +\gamma p_{xx} = -q_x + \alpha 	 y\mathds{1}_{\mathcal O} + f_3  & \textnormal{in } Q_T, \\
			-q_t - q_{xx} - \beta q_x =  -p_x + (1-\alpha) 	 z \mathds{1}_{\mathcal O} + f_4 & \textnormal{in } Q_T,\\
			p=p_x=q=0        &\textnormal{in } \Sigma_T, \\
			p(T)=0, \ \ q(T) =0  &\textnormal{in } (0,1) .
		\end{dcases}
	\end{align}
with  source terms $f_1, f_2, f_3 , f_4$   from  the space $L^2(Q_T)$.

As usual, the controllability problem boils down to study 
 the observability properties of the adjoint system associated to \eqref{sys_linear-1}--\eqref{sys_linear-2}.  Written in a more compact notation, the adjoint system reads as
\begin{align}\label{adj_sys}
	&\begin{dcases}
		-u_t + u_{xxxx} +\gamma u_{xx} = F_1 -w_x +  \alpha\, \zeta \mathds{1}_{\mathcal O}  & \textnormal{in } Q_T, \\
		-w_t - w_{xx} - \beta w_x =  F_2 -u_x + (1-\alpha)\theta \mathds{1}_{\mathcal O} & \textnormal{in } Q_T,\\
		\zeta_t + \zeta_{xxxx} +\gamma \zeta_{xx} = F_3 + \theta_x  & \textnormal{in } Q_T, \\
		\theta_t - \theta_{xx} + \beta \theta_x = F_4+  \zeta_x & \textnormal{in } Q_T,\\
		u=u_x=w=\zeta=\zeta_x=\theta=0        &\textnormal{in } \Sigma_T, \\
		u(T)=0, \ \ w(T) =0, &\textnormal{in } (0,1),
		\\ \zeta(0)=\zeta_0, \ \ \theta(0) =\theta_0  &\textnormal{in } (0,1).
 	\end{dcases}
\end{align}
with given  data $(\zeta_0, \theta_0)$ and source terms $F_j$ for $j=1,2,3,4$ from some suitable Hilbert spaces. 

The strategy amounts to apply  Carleman estimates for each equation of system \eqref{adj_sys} and then use the first and second equations to estimate locally the terms related to $\zeta$ and $\theta$. Note that for $\alpha\in (0,1)$, we have a natural way to estimate such terms thanks to the hypothesis $\mathcal O\cap\omega\neq \emptyset$, but as soon as $\alpha=0$ or $\alpha=1$, we lose information on either $\zeta$ or $\theta$ and we have to use the first-order couplings from the third and fourth equation of \eqref{adj_sys} to do local energy estimates. To circumvent this, we shall use some Carleman tools from the works \cite{Cerpa-Mercado-Pazoto} and \cite{Cerpa_Careno} allowing us to derivative with respect to $x$ the equations verified by $\zeta$ or $\theta$ and then estimate locally the first-order derivative of these variables.

\medskip 

\paragraph{\bf Paper organization} The rest of the paper is organized as follows. 
\begin{itemize} 
\item[--] In Section \ref{sec-well-posedness} we shortly discuss the well-posedness of the extended system \eqref{sys_equiv_1}--\eqref{sys_equiv_2}.

\item[--] In Section \ref{sec:carleman_estimates} we address several Carleman estimates for different $\alpha$.  More precisely, the Subsections \ref{carleman-alpha=0}, \ref{carleman-alpha=1} and \ref{carleman-alpha=0,1}  contain the Carleman inequalities associated to  the  adjoint system \eqref{adj_sys} for  $\alpha=0$, $\alpha=1$ and $\alpha\in (0,1)$ respectively.

\item[--]    Thereafter, in Section \ref{Section-local-null}, we discuss the main controllability results  for diffrent $\alpha\in [0,1]$.  We mainly focus on the two significant cases: $\alpha=0$ and $\alpha=1$. 

\smallskip   

To be more precise, the Subsections \ref{sec-obser-0}--\ref{sec-null-0} contains the observability inequality and null-controllability for the linearized model \eqref{sys_linear-1}--\eqref{sys_linear-2} when $\alpha=0$. Then, in Subsection \ref{section-locall-null-alpha=0} we prove the local null-controllability of the extended system \eqref{sys_equiv_1}--\eqref{sys_equiv_2} for $\alpha=0$.

\smallskip 

In the Subsections \ref{sec-obs-1}--\ref{sec-null-control-1} we discuss the null-controllability of the linearized model when $\alpha=1$ and the proof of the local null-controllability for this case will be as similar as the case $\alpha=0$ and thus we left this to the reader.

\item[--]  Finally, Section \ref{sec:final} is devoted to present some concluding remarks. 
\end{itemize}

\medskip 
\paragraph{\bf Notations}
Throughout the paper, $C>0$ denotes a generic constant that may vary line to line but is independent in the Carleman parameters $s$ or $\lambda$.

By $\dsp \iint$, we denote the integral on $Q_T$ and by $\dsp \iint_O$, we denote the integral on $(0,T)\times O$ for any non-empty open set $O \subset (0,1)$.

\medskip 
\section{Well-posedness}\label{sec-well-posedness}

 In this section, we shortly discuss the well-posedness of the $4\times 4$ control system \eqref{sys_equiv_1}--\eqref{sys_equiv_2}.


\subsection{The linear adjoint system}
Following the results in \Cref{Prop-appendix-well-posed} and  \Cref{Prop-appendix-regularity}, one can write the proposition below. 
\begin{prop}\label{prop-adjoint} 
\begin{enumerate}
  \item    For given data $(\zeta_0,\theta_0) \in [L^{2}(0,1)]^2$ and source terms $F_j \in L^2(Q_T)$ for $j=1,2,3,4$,  
  there exists unique $(u,w, \zeta, \theta)$ solution to \eqref{adj_sys} such that
\begin{align}  
	\label{well-posed-1}	u, \zeta \in  \C^0([0,T]; L^2(0,1)) \cap L^2(0,T; H^2_0(0,1)) \cap H^1(0,T; H^{-2}(0,1)), \\
	\label{well-posed-2}	w, \theta \in \C^0([0,T]; L^2(0,1)) \cap L^2(0,T; H^1_0(0,1)) \cap H^1(0,T; H^{-1}(0,1)) , 
\end{align} 
satisfying 
\begin{multline*}
	\|u\|_{\C^0(L^2)\cap L^2(H^2_0) \cap H^1(H^{-2})} + 	\|w\|_{\C^0(L^2)\cap L^2(H^1_0) \cap H^1(H^{-1})}   
	+ 	\|\zeta\|_{\C^0(L^2)\cap L^2(H^2_0) \cap H^1(H^{-2})}\\ + 	\|\theta\|_{\C^0(L^2)\cap L^2(H^1_0) \cap H^1(H^{-1})}  \leq C \Big(\|(\zeta_0, \theta_0)\|_{[L^2(0,1)]^2} + \sum_{j=1}^4 \|F_j\|_{L^2(Q_T)}\Big)  .
\end{multline*}

\item 
  If we choose the data $(\zeta_0, \theta_0) \in H^2_0(0,1)\times H^1_0(0,1)$, then
the solution to \eqref{adj_sys} satisfies the  following regularity results:
\begin{align} 
	\label{regu-1}	u, \zeta \in  \C^0([0,T]; H^2_0(0,1)) \cap L^2(0,T; H^4(0,1) \cap H^2_0(0,1)) \cap H^1(0,T; L^2(0,1)), \\
	\label{regu-2}	w, \theta \in \C^0([0,T]; H^1_0(0,1)) \cap L^2(0,T; H^2(0,1) \cap H^1_0(0,1)) \cap H^1(0,T; L^2(0,1)),
\end{align} 
with in addition, 
\begin{multline*}
	\|u\|_{\C^0(H^2_0)\cap L^2(H^4\cap H^2_0) \cap H^1(L^{2})} + 	\|w\|_{\C^0(H^1_0)\cap L^2(H^2) \cap H^1(L^2)}  
	+ 	\|\zeta\|_{\C^0(H^2_0)\cap L^2(H^4\cap H^2_0) \cap H^1(L^{2})} \\ + 
	\|\theta\|_{\C^0(H^1_0)\cap L^2(H^2) \cap H^1(L^2)} 
	 \leq  C \Big(\|(\zeta_0, \theta_0)\|_{H^2_0(0,1)\times H^1_0(0,1)} + \sum_{j=1}^4 \|F_j\|_{L^2(Q_T)}\Big).
\end{multline*}
\end{enumerate}
\end{prop}

	\subsection{The main linear system}
	
In a same spirit, we have the following result for the direct linear system \eqref{sys_linear-1}--\eqref{sys_linear-2}. 

	\begin{prop}\label{prop-linear}
		For given $(y_0, z_0)\in [L^2(0,1)]^2$ and $f_j\in L^2(Q_T)$ for $j=1,2,3,4$ 
		 there exists unique $(y,z,p,q)$ solution to  \eqref{sys_linear-1}--\eqref{sys_linear-2} such that 
	\begin{align}  
\label{sol-lin-1}		y, p \in  \C^0([0,T]; L^2(0,1)) \cap L^2(0,T; H^2_0(0,1)) \cap H^1(0,T; H^{-2}(0,1)), \\
\label{sol-lin-2}	z, q \in \C^0([0,T]; L^2(0,1)) \cap L^2(0,T; H^1_0(0,1)) \cap H^1(0,T; H^{-1}(0,1)) , 
	\end{align} 
	satisfying 
	\begin{multline}\label{bound-sol-linear}
		\|y\|_{\C^0(L^2)\cap L^2(H^2_0) \cap H^1(H^{-2})} + 	\|z\|_{\C^0(L^2)\cap L^2(H^1_0) \cap H^1(H^{-1})}   
		+ 	\|p\|_{\C^0(L^2)\cap L^2(H^2_0) \cap H^1(H^{-2})}\\ + 	\|q\|_{\C^0(L^2)\cap L^2(H^1_0) \cap H^1(H^{-1})}  \leq C \Big(\|(y_0, z_0)\|_{[L^2(0,1)]^2} + \sum_{j=1}^4 \|f_j\|_{L^2(Q_T)} + \|(h_1,h_2)\|_{[L^2((0,T)\times \omega)]^2}\Big)  .
	\end{multline}
	
	\end{prop}

	


\subsection{The nonlinear system}

\begin{theorem}
	There exists positive real number $\delta_0$ such that for any $(y_0,z_0)\in [L^2(0,1)]^2$, \\ $(h_1,h_2)\in [L^2((0,T)\times \omega)]^2$   and $(\xi_1, \xi_2)\in [L^2(Q_T)]^2$ satisfying 
	\begin{align}\label{cond-initial}
		\|(y_0, z_0)\|_{[L^2(0,1)]^2} + \|(h_1, h_2)\|_{[L^2((0,T)\times \omega)]} + \|(\xi_1, \xi_2)\|_{[L^2(Q_T)]^2} \leq \delta_0 ,
	\end{align}
the nonlinear problem \eqref{sys_equiv_1}--\eqref{sys_equiv_2} has unique solution $(y,z,p,q)$ with 
\begin{align*}
	(y,z, p, q)\in \C^0([0,T]; [L^2(0,1)]^4) \cap L^2(0,T; H^2_0(0,1) \times H^1_0(0,1) \times H^2_0(0,1) \times H^1_0(0,1)  )   .
\end{align*}
\end{theorem}

\begin{proof}

	Let us define the map 
	\begin{align*}
		&\Lambda: [L^\infty(0,T; L^2(0,1)) \cap L^2(0,T; H^2_0(0,1))]^2    \to [L^\infty(0,T; L^2(0,1)) \cap L^2(0,T; H^2_0(0,1))]^2  , \\
		&\Lambda(\widetilde y, \widetilde p) = (y,p),
	\end{align*}
	where $(y,z,p,q)$ 
	 is the unique solution to \eqref{sys_equiv_1}--\eqref{sys_equiv_2} satisfying \eqref{sol-lin-1}--\eqref{sol-lin-2}  with $f_1 = \xi_1 - \widetilde y \, \widetilde y_x$, $f_2=\xi_2$, $f_3=\widetilde y \, \widetilde p_x$ and $f_4=0$.

 We compute that 
	\begin{align*}
	\|\widetilde y \, \widetilde y_x\|_{L^2(Q_T)} &= \bigg(	\int_0^T \int_0^1 |\widetilde y \, \widetilde y_x|^2\bigg)^{\frac{1}{2}}  \leq \bigg(\int_0^T \|\widetilde y(t)\|^2_{L^{2}(0,1)} \|\widetilde y_x(t)\|^2_{L^{\infty}(0,1)}\bigg)^{\frac{1}{2}} \\
	& \leq \|\widetilde y\|_{L^\infty(0,T; L^{2}(0,1))} \|\widetilde y_x \|_{L^2(0,T; H^1_0(0,1))} \\
	 & \leq \|\widetilde y\|_{L^\infty(0,T; L^{2}(0,1))} \|\widetilde y \|_{L^2(0,T; H^2_0(0,1))} ,
	\end{align*}
and similarly, 
\begin{align*}
		\|\widetilde y \, \widetilde p_x\|_{L^2(Q_T)} \leq \|\widetilde y\|_{L^\infty(0,T; L^{2}(0,1))} \|\widetilde p \|_{L^2(0,T; H^2_0(0,1))}.
\end{align*}
Therefore, by means of \Cref{prop-linear}, we have 
	\begin{multline}\label{bound-linear}
	\|y\|_{\C^0(L^2) \cap L^2(H^2_0) \cap H^1(H^{-2})}
		+\|z\|_{\C^0(L^2) \cap L^2(H^1_0) \cap H^1(H^{-1}) }
		+\|p\|_{\C^0(L^2) \cap L^2(H^2_0) \cap H^1(H^{-2})} \\
		+ \|q\|_{\C^0(L^2) \cap L^2(H^1_0) \cap H^1(H^{-1})}
	 \leq C_0 \Big(\|(y_0, z_0)\|_{[L^2(0,1)]^2} + \|(h_1,h_2)\|_{[L^2((0,T)\times \omega)]^2}  \\  
	 +\|(\xi_1, \xi_2)\|_{[L^2(Q_T)]^2} + \|\widetilde y\|_{L^\infty(L^{2})} \|\widetilde y \|_{L^2(H^2_0)} + \|\widetilde y\|_{L^\infty(L^{2})} \|\widetilde p \|_{L^2(H^2_0)} \Big) .
	\end{multline}
for some  constant $C_0>0$.


	
Now, we	denote the set 
	\begin{align*}
	\B_R = \Big\{ (y,p) \in & [L^\infty(0,T;L^{2}(0,1)) \cap L^2(0,T;H^{2}_0(0,1))]^2 \, \mid \, \\
& \qquad \qquad 	\big(\| y\|_{L^\infty(L^{2}) \cap L^2(H^2_0)} + \|p\|_{L^\infty(L^{2}) \cap L^2(H^2_0)}\big) \leq R  \Big\}. 
	\end{align*} 

--	Then starting with $(\widetilde y, \widetilde p)\in \B_R$, we have from \eqref{bound-linear} 
	\begin{equation}\label{bound-linear-2}
	\begin{aligned}
		\|y\|_{L^\infty(L^2) \cap L^2(H^2_0)}
		+\|z\|_{L^\infty(L^2) \cap L^2(H^1_0)}
		+\|p\|_{L^\infty(L^2) \cap L^2(H^2_0)} 
		+ \|q\|_{L^\infty(L^2) \cap L^2(H^1_0)} \\
		\leq C_0 \Big(\|(y_0, z_0)\|_{[L^2(0,1)]^2} + \|(h_1,h_2)\|_{[L^2((0,T)\times \omega)]^2}    
		+\|(\xi_1, \xi_2)\|_{[L^2(Q_T)]^2} + R^2 \Big) .
	\end{aligned}
\end{equation}
It follows that if $R< \frac{1}{C_0}$ and 
\begin{align*}
\|(y_0, z_0)\|_{[L^2(0,1)]^2} + \|(h_1,h_2)\|_{[L^2((0,T)\times \omega)]^2}  
+\|(\xi_1, \xi_2)\|_{[L^2(Q_T)]^2} \leq \frac{R-C_0 R^2}{C_0},
\end{align*}
	one has $\Lambda(\B_R) \subset \B_R$ which implies  $\B_R$  is  stable under  the map $\Lambda$.  We take $\delta_0=\frac{R-C_0 R^2}{C_0}$ in \eqref{cond-initial}.

	\medskip 
	
-- Let us now prove that $\Lambda$ is a contraction map. Choose two elements $(\widetilde y, \widetilde p)$, $(\widehat y, \widehat p)$ from the space $\B_R$ and denote the associated solutions by $(y_1,z_1,p_1,q_1)$ and $(y_2,z_2,p_2,q_2)$ respectively. 

We further denote $(y,z,p,q)= (y_1-y_2,  z_1-z_2, p_1-p_2, q_1-q_2)$ which satisfies the following set of equations
\begin{align}\label{sys_diff_1}
	&\begin{dcases}
		y_t + 	 y_{xxxx} +\gamma 	 y_{xx}  = z_x  + \widehat y \, \widehat y_x -  \widetilde y\, \widetilde y_x   & \textnormal{in } Q_T, \\
		z_t - 	 z_{xx} + \beta 	 z_x = 	 y_x  & \textnormal{in } Q_T,\\
		y= 	 y_x= 	 z=0        &\textnormal{in } \Sigma_T, \\
		y(0)= 0, \ \ 	 z(0) = 0  &\textnormal{in } (0,1),
	\end{dcases} \\ \label{sys_diff_2}
	&\begin{dcases}
		-p_t + p_{xxxx} +\gamma p_{xx}  = -q_x + \alpha y\mathds{1}_{\mathcal O}  + \widetilde y \, \widetilde p_x - \widehat y\, \widehat p_x  
		  & \textnormal{in } Q_T, \\
		-q_t - q_{xx} - \beta q_x =  -p_x + (1-\alpha) 	 z \mathds{1}_{\mathcal O} & \textnormal{in } Q_T,\\
		p=p_x=q=0        &\textnormal{in } \Sigma_T, \\
		p(T)=0, \ \ q(T) =0  &\textnormal{in } (0,1) .
	\end{dcases}
\end{align}

We first compute 
\begin{align*}
	\bigg(	\int_0^T \int_0^1 |\widetilde y \, \widetilde p_x - \widehat y\, \widehat p_x|^2 \bigg)^{\frac{1}{2}} 
	&\leq  
	\bigg(2	\int_0^T \int_0^1 |\widetilde y  (\widetilde p_x - \widehat p_x)|^2 + |\widehat p_x (\widetilde  y - \widehat y)|^2 \bigg)^{\frac{1}{2}} \\
	& \leq \sqrt{2} \|\widetilde y\|_{L^\infty(L^2)}\| \widetilde p - \widehat p \|_{L^2(H^2_0)} + \sqrt{2} \|\widehat p\|_{L^2(H^2_0)} \|\widetilde y - \widehat y \|_{L^\infty(L^2)}   .
\end{align*}
Similarly, one has 
\begin{align*}
	\|\widetilde y \, \widetilde y_x - \widehat y\, \widehat y_x\|_{L^2(Q_T)} 
	\leq \sqrt{2} \|\widetilde y\|_{L^\infty(L^2)}\| \widetilde y - \widehat y \|_{L^2(H^2_0)} + \sqrt{2} \|\widehat y\|_{L^2(H^2_0)} \|\widetilde y - \widehat y \|_{L^\infty(L^2)}.
	\end{align*}
Using the above two facts, the solution to \eqref{sys_diff_1}--\eqref{sys_diff_2} satisfies 
\begin{equation}
	\begin{aligned}
	&	\|y\|_{L^\infty(L^2) \cap L^2(H^2_0)} 
	+\|z\|_{L^\infty(L^2) \cap L^2(H^1_0)}
	+\|p\|_{L^\infty(L^2) \cap L^2(H^2_0)} 	+ \|q\|_{L^\infty(L^2) \cap L^2(H^1_0)} \\
	&\leq C_0 \Big(\| \widetilde y\, \widetilde y_x - \widehat y \, \widehat y_x\|_{L^2(Q_T)} + \| \widetilde y \, \widetilde p_x -  \widehat y\, \widehat p_x\|_{L^2(Q_T)} \Big)\\
&	\leq \sqrt{2}C_0 \|\widetilde y\|_{L^\infty(L^2)} 
	\Big( \| \widetilde y - \widehat y \|_{L^2(H^2_0)} +  \| \widetilde p - \widehat p \|_{L^2(H^2_0)}\Big) \\
	& \qquad \quad + \sqrt{2}C_0 \| \widetilde y - \widehat y \|_{L^\infty(L^2)} \Big( \|\widehat y\|_{L^2(H^2_0)} + \|\widehat p\|_{L^2(H^2_0)} \Big) \\
	& \leq \sqrt{2}C_0 R 	\Big( \| \widetilde y - \widehat y \|_{L^\infty(L^2)\cap L^2(H^2_0)} + \| \widetilde p - \widehat p \|_{L^\infty(L^2)\cap L^2(H^2_0)}\Big) .
	\end{aligned}
\end{equation}
Now,  choose $R>0$ in such a way that $\sqrt{2}C_0 R<1$, so that  the map $\Lambda$ is contracting. Thus, using the Banach fixed point theorem,  there exists a unique fixed point  of $\Lambda$ in $\B_R$, which gives the  unique solution $(y,z,p,q)$ to \eqref{sys_equiv_1}--\eqref{sys_equiv_2}. 

The proof is complete.
\end{proof}

\section{Carleman estimates for different $\boldsymbol \alpha$}\label{sec:carleman_estimates}

This section is devoted to obtain several Carleman inequalities depending on the parameter $\alpha$, namely for the 
following cases:
\begin{itemize}
	\item[(i)]  $\alpha=0$
	\item[(ii)] $\alpha =1$, 
	\item[(iii)]   $\alpha \in (0,1)$. 
\end{itemize}
Note that, the cases $\alpha=0$ and $\alpha=1$ are more interesting  since we remove one of the two observing variables.  In fact, depending on the above three choices of $\alpha$, the strategy to obtain a suitable Carleman estimate also changes.

\medskip

Let us
define some Carleman weights which have been introduced  in the articles \cite{Sergio, Cerpa-Mercado-Pazoto}.

\smallskip 

\paragraph{\bf Weight functions} Recall that $\oo \cap \omega \neq \emptyset$ and therefore, there is an open set $\omega_0\subset \subset \oo  \cap \omega$.  In what follows,  we establish the Carleman estimates with the observation domain $\omega_0$.

Consider a function $\nu \in \C^4([0,1])$ satisfying   
\begin{align}\label{definition_nu} 
	\begin{dcases}
		\nu(x)>0 \  \ \forall x\in (0,1), \ \ \nu(0)=\nu(1)=0, \\
		|\nu^\prime(x) | \geq c >0 \ \  \forall x \in \overline{(0,1) \setminus \omega_0} \ \ \text{ for some } c > 0 .
	\end{dcases}
\end{align} 
In particular, we have  $\nu^{\prime}(0)>0$ and $\nu^\prime(1)<0$. 

Now, for some  constants $\lambda>1$ and $k>m>0$,   we define the weight functions 
\begin{align}\label{weight_function}
	\vphi_m(t,x)= \frac{e^{\lambda(1+\frac{1}{m})k \|\nu\|_{\infty}   }-  e^{\lambda\big( k\|\nu\|_{\infty} + \nu(x) \big)}}{t^m(T-t)^m}	, \quad \xi_{m}(t,x) = \frac{e^{\lambda\big( k\|\nu\|_{\infty} + \nu(x) \big)}}{t^m(T-t)^m}, \quad \forall (t,x) \in Q_T.
\end{align}
%
Here, observe that $\vphi_m$ and $\xi_m$ are positive functions in $[0,1]$ due to the choices of $\lambda$, $k$ and $m$. 

\medskip 

Some immediate results associated with the weights are following.
\begin{itemize} 
\item[--] For any $b>0$, there exists some constant $C>0$ such that 
\begin{align}\label{weight_estimates}
	|(e^{-2s\vphi_m} \xi^b_m)_x| \leq Cs\lambda \xi_m (e^{-2s\vphi_m} \xi^b_m),
\end{align}
where $\lambda>1$ and $k>m>0$.

\item[--] Similarly, for any $b>0$, there is some constant $C>0$ such that 
\begin{align}\label{weight_estimate-t}
	|(e^{-2s\vphi_m} \xi^b_m)_t| \leq Cs \xi_m^{1+\frac{1}{m}} (e^{-2s\vphi_m} \xi^b_m),
\end{align}
with $\lambda>1$ and $k>m>0$. 

\item[--] 
For any $m>1$, we  observe that
\begin{align} \notag 
	t^{m-1}(T-t)^{m-1} &\leq C T^{2m-2} , \\ \notag 
	\text{i.e., } \ 1 \leq C T^{2m-2} \frac{t(T-t)}{t^m (T-t)^m} &\leq C T^{2m-2} \xi_m^{1-1/m} \\
	\text{i.e., } \	\xi_m^{1/m} &\leq CT^{2m-2} \xi_m. \label{esti_time_deri}
\end{align}
Using this in \eqref{weight_estimate-t}, one has 
\begin{align}\label{weight_estimate-t-m}
	|(e^{-2s\vphi_m} \xi^b_m)_t| \leq C T^{2m-2} s \xi^2_m (e^{-2s\vphi_m} \xi^b_m), \quad \text{for } m>1. 
\end{align}

\end{itemize}

\medskip 

\paragraph{\bf Some standard Carleman estimates} 
Let us write the following Carleman estimates for the second order parabolic operator. 
\begin{theorem}\label{thm-Fur-Ima}
	Let $\vphi_m$ and $\xi_m$ be given by \eqref{weight_function} with $m\geq 1$ and  $r\in \mathbb R$.  Then, there exists positive constants $\lambda_0$, $s_0=\sigma_0(T^{2m-1}+T^{2m})$ with some $\sigma_0>0$ and $C$, such that we have 
	\begin{equation}\label{Notation_heat}
	\begin{aligned} 
	I_H(q;r): 
&	= s^{r} \lambda^{r+1} \iint e^{-2s\vphi_m} \xi_m^r |q|^2 + s^{r-2} \lambda^{r-1} \iint e^{-2s\vphi_m} \xi_m^{r-2} |q_x|^2 \\
& \qquad  \quad 	 + s^{r-4} \lambda^{r-3} \iint e^{-2s\vphi_m} \xi_m^{r-4} \left(|q_t|^2 +|q_{xx}|^2\right) 
	\\
&	\leq C \left( s^{r-3}\lambda^{r-3} \iint e^{-2s\vphi_m} \xi_m^{r-3}|q_t+q_{xx}|^2  + s^r \lambda^{r+1}\iint_{\omega_0}  e^{-2s\vphi_m} \xi_m^{r}|q|^2\right),
	\end{aligned} 
\end{equation}
for every $\lambda \geq \lambda_0$, $s\geq s_0$, and 
$q\in L^2(0,T; H^2(0,1)\cap H^1_0(0,1))\cap H^1(0,T; L^2(0,1))$.  
\end{theorem} 
The above result is well-known due to the pioneer work by Fursikov and Imanuvilov \cite{Fur-Ima}.

\smallskip 

We also need the following Carleman estimate for the fourth order parabolic operator. 
\begin{theorem}\label{thm-Zhou}
		Let $\vphi_m$ and $\xi_m$ be given by \eqref{weight_function} with $m\geq 2/5$.   Then, there exists positive constants $\lambda_0$, $s_0=\sigma_0(T^{2m-2/5}+T^{2m})$ with some $\sigma_0>0$ and $C$, such that we have 
	\begin{equation}\label{Notation_KS}
		\begin{aligned} 
		I_{KS}(q):&=s^7\lambda^8 \iint e^{-2s\vphi_m} \xi^7_m |q|^2 + s^5\lambda^6 \iint e^{-2s\vphi_m} \xi^5_m |q_{x}|^2 +  s^3\lambda^4\iint e^{-2s\vphi_m} \xi^3_m |q_{xx}|^2 \\
		& \qquad \qquad 	+ s\lambda^2 \iint e^{-2s\vphi_m} \xi_m |q_{xxx}|^2   +	s^{-1} \iint e^{-2s\vphi_m}\xi^{-1}_m (|q_t|^2+ |q_{xxxx}|^2\big)\\
		&\leq 
		C\left(\iint e^{-2s\vphi_m} |-q_t + q_{xxxx} + \gamma q_{xx}|^2 + s^7 \lambda^8\iint_{\omega_0} e^{-2s\vphi_m} \xi^7_m |q|^2    \right),
			\end{aligned}
	\end{equation}	
for every $\lambda \geq \lambda_0$, $s\geq s_0$, and 
$q\in L^2(0,T; H^4(0,1)\cap H^2_0(0,1))\cap H^1(0,T; L^2(0,1))$.  
	\end{theorem}
The proof for the above result can be found in \cite{Zhou}. We also refer the work \cite{Cerpa-Mercado-Pazoto}.


\subsection{Carleman estimate for the case when $\boldsymbol{\alpha=0}$}\label{carleman-alpha=0}

The adjoint system \eqref{adj_sys} for the case $\alpha=0$ reads as 
\begin{align}\label{adj_sys_a_0}
	&\begin{dcases}
		-u_t + u_{xxxx} +\gamma u_{xx} = F_1 -w_x   & \textnormal{in } Q_T, \\
		-w_t - w_{xx} - \beta w_x = F_2  -u_x + \theta \mathds{1}_{\mathcal O} & \textnormal{in } Q_T,\\
		\zeta_t + \zeta_{xxxx} +\gamma \zeta_{xx} = F_3+\theta_x  & \textnormal{in } Q_T, \\
		\theta_t - \theta_{xx} + \beta \theta_x = F_4 + \zeta_x & \textnormal{in } Q_T,\\
		u=u_x=w=\zeta=\zeta_x=\theta=0        &\textnormal{in } \Sigma_T, \\
		u(T)=0, \ \ w(T) =0, &\textnormal{in } (0,1),
		\\ \zeta(0)=\zeta_0, \ \ \theta(0) =\theta_0  &\textnormal{in } (0,1).
 	\end{dcases}
\end{align}

Before going to prove  the required Carleman inequality associated to the system \eqref{adj_sys_a_0}, we recall the weight functions  $\vphi_m$ and $\xi_m$ as given by \eqref{weight_function}. We also define  
\begin{align}\label{max_min}
	\begin{dcases} 
		\widehat \vphi_m(t) = \max_{x\in [0,1]} \vphi_m(t,x)= \vphi_m(t,0)=\vphi_m(t,1), \\ 
		\xi^*_m(t) = \min_{x\in [0,1]} \xi_m(t,x)= \xi_m(t,0)=\xi_m(t,1).
	\end{dcases}
\end{align}

We now a prove the following Carleman inequality for the adjoint states  $(u,w,\zeta, \theta)$ with 
$$(u,\zeta)\in [L^2(0,T;H^4(0,1)\cap H^2_0(0,1)) \cap H^1(0,T; L^2(0,1))]^2,$$ 
$$(w,\theta)\in [L^2(0,T;H^2(0,1)\cap H^1_0(0,1)) \cap H^1(0,T; L^2(0,1))]^2,$$
(of course, one should start with regular data for instance $(\zeta_0,\theta_0)\in H^2_0(0,1)\times H^1_0(0,1)$ and $F_j\in L^2(Q_T)$ for $j=1,2,3,4$).

\begin{theorem}[Carleman inequality: the case $\alpha=0$]\label{Theorem-Carleman-second-app}
	Let $(\vphi_m, \xi_m)$ and $(\widehat \vphi_m, \xi^*_m)$ be given by \eqref{weight_function} and \eqref{max_min} respectively with $m\geq 2$. Then, there exist positive constants $\lambda^*$, $s^*:= \sigma^* (T^m+ T^{2m-2/5}+T^{2m-1} + T^{2m})$ with some $\sigma^*>0$ and $C$ such that we have the following estimate satisfied by the solution to \eqref{adj_sys_a_0}:
	\begin{multline}\label{Carleman-second-app}
		I_{KS}(u) + I_{H}(w;7) + s^7\lambda^8 \iint e^{-2s\vphi_m}\xi^7_m |\zeta_x|^2 + s^7\lambda^8 \iint e^{-2s\widehat \vphi_m}(\xi^*_m)^7 |\zeta|^2 + I_H(\theta;9) \\
		\leq 
			 C \bigg[\iint e^{-2s\vphi_m} \left(|F_1|^2  + s^5\lambda^5  \xi^5_m|F_3|^2 + s^7\lambda^8 \xi^7_m |F_4|^2\right) + s^{37}\lambda^{22}\iint e^{-10s\vphi_m + 8s\widehat \vphi_m} \xi^{37}_m |F_2|^2 \\
	+	  s^{39}\lambda^{24} \iint_{\omega_0} e^{-10s\vphi_m + 8s\widehat \vphi_m} \xi^{39}_m |u|^2 +   s^{41}\lambda^{26} \iint_{\omega_0} e^{-10s\vphi_m + 8s\widehat \vphi_m} \xi^{41}_m|w|^2\bigg]
	\end{multline}
	for all $\lambda\geq \lambda^*$ and  $s\geq s^*$, where $I_{KS}(\cdot)$ and $I_H(\cdot;\cdot)$ are given by \eqref{Notation_KS} and \eqref{Notation_heat} respectively. 
\end{theorem}

\smallskip 

Let us shortly point out the idea behind the proof for the  Carleman inequality \eqref{Carleman-second-app}.
\begin{itemize}
		\item[(i)] For the variables $u$  satisfying fourth order parabolic equation, we use the Carleman estimate given by \Cref{thm-Zhou}. 
		

	\item[(ii)]  We observe that the Carleman estimate for $\zeta$   will always be associated with an observation integral of $\zeta$  and there is no chance to absorb it by any of the leading integrals. In fact, there is a coupling by $\zeta_x$ to the equation of $\theta$ and therefore, we are going to use a Carleman estimate for the variable $\zeta_x$, see \Cref{Lemma-Careleman-zeta_x}.
	
	In this context, we recall the work \cite{Cerpa_Careno}, where such a Carleman estimate has been established.

	\item[(iii)] For $w$ and $\theta$, we use the classical Carleman inequalities by Fursikov and Imanuvilov (see \Cref{thm-Fur-Ima}) for the heat equation, maybe with different powers of Carleman parameters. 
\end{itemize}

\begin{remark}
	In the Carleman estimate \eqref{Carleman-second-app}, there appears the weight function $e^{-10s\vphi_m + 8s\widehat \vphi_m}$.    But, by \Cref{Lemma-auxiliary}, there exists some $c_0>0$ such that
	\begin{align*}
		-10s\vphi_m + 8s\widehat \vphi_m \leq \frac{-c_0 s}{t^m (T-t)^m},
	\end{align*}
	which ensures obtaining a suitable observability inequality from the Carleman estimate \eqref{Carleman-second-app}, see \Cref{Obser-ineq}. 
\end{remark}


\smallskip

Let us give the required Carleman estimates for the quantities $u$, $w$, $\zeta_x$ and $\theta$ as mentioned above.
%

%
\begin{lemma}[Carleman inequality for $u$, the case $\alpha=0$]\label{lemma-carleman-u-app-2}
	Let $\vphi_m$ and $\xi_m$ be as given by  \eqref{weight_function} with $m\geq 2/5$. Then, there exist positive constants $\bar \lambda_1$, $\bar s_1:=\bar \sigma_1 (T^{2m}+T^{2m-2/5})$ with some $\sigma_1>0$ and $C$, such that we have the following estimate for $u\in L^2(0,T;H^4(0,1)\cap H^2_0(0,1)) \cap H^1(0,T; L^2(0,1))$,
	\begin{align} \label{carleman_u-app-2}
		I_{KS}(u)
		\leq C \left(\iint e^{-2s\vphi_m} (|F_1|^2+ |w_x|^2)  +s^7\lambda^8\iint_{\omega_0}  e^{-2s\vphi_m} \xi^7_m |u|^2      \right)
	\end{align}
	for all $\lambda\geq \bar \lambda_1$, $s\geq \bar s_1$, where $I_{KS}(\cdot)$ is defined by \eqref{Notation_KS}.    
\end{lemma}	

For the component $w$, we shall use the standard Carleman inequality for heat equations. 
\begin{lemma}[Carleman inequality for $w$, the case $\alpha=0$]\label{lemma-carleman-w-app-2}
	Let $\vphi_m$ and $\xi_m$ be defined as in \eqref{weight_function} with $m\geq 1$. Then, there exist positive constants $\bar \lambda_2$, $\bar s_2:=\bar \sigma_2 (T^{2m} +T^{2m-1})$ with some $\bar \sigma_2>0$ and $C$, such that we have the following estimate for $w\in L^2(0,T; H^2(0,1)\cap H^1_0(0,1)) \cap H^1(0,T; L^2(0,1))$,
	\begin{align}\label{carleman_w-app-2}
		I_H(w;7)
		\leq C \left(s^4\lambda^4 \iint e^{-2s\vphi_m}\xi^4_m \big(|F_2|^2 + |u_x|^2 +|\theta|^2\big) + s^7\lambda^8\iint_{\omega_0}  e^{-2s\vphi_m} \xi^7_m |w|^2      \right)
	\end{align}
	for all $\lambda\geq \bar \lambda_2$ and $s\geq \bar s_2$, where $I_{H}(\cdot; \cdot)$ is defined by \eqref{Notation_heat}.     
\end{lemma}

We now write the Carleman estimate for $\zeta_x$, thanks to the work \cite{Cerpa_Careno}.
\begin{lemma}[Carlemn inequality for $\zeta_x$, the case $\alpha=0$]\label{Lemma-Careleman-zeta_x}
	Let $(\vphi_m, \xi_m)$ and $(\widehat \vphi_m, \xi^*_m)$ be given by \eqref{weight_function} and \eqref{max_min} respectively with $m\geq 2$, $k>m$. Then, there exist positive constants $\bar \lambda_3$, $\bar s_3:=\bar \sigma_3 (T^m+ T^{2m})$ with some $\bar \sigma_3>0$ and $C$ such that we have the following estimate for $\zeta_x \in L^2(0,T;H^3(0,1)\cap H^1_0(0,1))$,
	\begin{multline}\label{eq: Carleman-zeta_x}
		s^7 \lambda^8 \iint \left(e^{-2s\vphi_m } \xi_m^7 |\zeta_x|^2 + e^{-2s\widehat \vphi_m} (\xi_m^*)^7 |\zeta|^2\right)  + s^{5-\frac{2}{m}} \lambda^{5}\|e^{-s\widehat \vphi_m} (\xi^*_m)^{\frac{5}{2}-\frac{1}{m}} \zeta  \|^2_{L^2(0,T; H^4(0,1))} 
		\\
		\leq C\bigg[s^5\lambda^5 \iint    e^{-2s\vphi_{m}}  \xi^5_m \big(|\theta_x|^2 +  |F_3|^2\big)  
		+ T^{10m} \iint   e^{-2s\vphi_{m}} \xi^5_m |\theta_{xx}|^2 
		+  s^7 \lambda^8 \iint_{\omega_0}  e^{-2s\vphi_m}\xi_m^7 |\zeta_x|^2\bigg] 
	\end{multline}
	for all $\lambda\geq \bar \lambda_3$ and $s\geq \bar  s_3$.
\end{lemma} 

The proof of \Cref{Lemma-Careleman-zeta_x} can be verified using the techniques developed in \cite[Section 3.2]{Cerpa_Careno}. For sake of completeness, we give a short sketch below. 
\begin{proof}[Proof of \Cref{Lemma-Careleman-zeta_x}]
	Recall the equation of $\zeta$ from the system \eqref{adj_sys_a_0} and consider the following equation for $\widetilde \zeta:= \zeta_x$, given by 
	\begin{align}\label{eq:zeta_x}
		\begin{dcases}
			\widetilde \zeta_t + 	\widetilde \zeta_{xxxx} + \gamma 	\widetilde \zeta_{xx} =  \theta_{xx} + F_{3,x}  &\text{in } Q_T, \\
			\widetilde \zeta=0, \ \ 	\widetilde \zeta_x= 	\zeta_{xx} &\text{in }\Sigma_T, \\
			\widetilde \zeta(0, \cdot)= 0 &\text{in } (0,1),	
		\end{dcases}
	\end{align}
	where it is clear that $\theta_{xx}\in L^2(Q_T)$ and   we have  $\widetilde \zeta_x(\cdot) \in L^2(0,T)$ in $\Sigma_T$.  
	
	Now, thanks to \cite[Theorem 3.5]{Cerpa_Careno}, we have the following auxiliary estimate for $\zeta_x$,
	\begin{multline}\label{carleman-zeta_x-auxi}
		s^7 \lambda^8 \iint e^{-2s\vphi_m } \xi_m^7 |\zeta_x|^2 \leq 
		C \bigg[ \iint e^{-2s\vphi_m}  |\theta_{xx}|^2 + s^2\lambda^2 \iint e^{-2s\vphi_m} \xi_m^2 |F_3|^2   
		\\
		+ s^5 \lambda^5 \int_0^T e^{-2s\widehat \vphi_m} (\xi_m^*)^5 \big(|\zeta_{xx}(t,0)|^2 + |\zeta_{xx}(t,1)|^2\big)    +  
		s^7\lambda^8 \iint_{\omega_0} e^{-2s\vphi_m} \xi^7_m |\zeta_x|^2 \bigg],
	\end{multline}
	for every $\lambda\geq c_1$ and $s\geq c_2(T^{m}+T^{2m-2/5})$ for some $c_1, c_2>0$.

	First observe that,
	\begin{align}\label{subestimate} 
		s^7 \lambda^8 \iint e^{-2s\widehat \vphi_m } (\xi^*_m)^7 |\zeta|^2  \leq C 	s^7 \lambda^8 \iint e^{-2s\vphi_m } \xi_m^7 |\zeta_x|^2,
	\end{align}
	since $\zeta(t,0)=\zeta(t,1)=0$, $\forall t\in [0,T]$. 
	
	\smallskip
	Next, we define $\rho(t)=s^{5/2-1/m}\lambda^{5/2}e^{-s\widehat \vphi_m}(\xi^*_m)^{5/2-1/m}$, $\forall t\in [0,T]$, and consider the equation satisfied by $\zeta_{\rho}:=\rho \zeta$, 
	\begin{align*}
		\begin{dcases} 
			(\zeta_\rho)_t + (\zeta_\rho)_{xxxx} +\gamma (\zeta_\rho)_{xx} = \rho \theta_x + \rho F_3  - \rho_t \zeta &\text{in } Q_T, \\
			\zeta_\rho= (\zeta_\rho)_x = 0 &\text{in } \Sigma_T, \\
			\zeta_\rho(0,\cdot)=0 &\text{in } (0,1).
		\end{dcases}
	\end{align*}
	Then, using the regularity result for KS equations given  by \Cref{Prop-appendix-regularity},  we have 
	\begin{multline}\label{regularity_estimate}
		\|\zeta_\rho\|^2_{L^2(0,T; H^4(0,1))}\leq C s^{5-2/m}\lambda^5\Big(\|e^{-s\widehat \vphi_m}(\xi^{*}_m)^{5/2-1/m} \theta_x\|^2_{L^2(Q_T)} \\
		 + \|e^{-s\widehat \vphi_m}(\xi^{*}_m)^{5/2-1/m} F_3\|^2_{L^2(Q_T)}  
	 +  s^{7} \lambda^{5}\|e^{-s\widehat \vphi_m}(\xi^{*}_m)^{7/2}\zeta\|^2_{L^2(Q_T)}\Big),
	\end{multline}
	for all $s\geq C(T^{m}+T^{2m})$. 
	
	Also, by some standard interpolation result and Young's inequality, one has 
	\begin{multline}\label{regularity-2} 
		s^5 \lambda^5 \Big(\| e^{-s\widehat \vphi_m} (\xi^*_m)^{5/2} \zeta_{xx}(t,0) \|^2_{L^2(0,T)} + \| e^{-s\widehat \vphi_m} (\xi^*_m)^{5/2} \zeta_{xx}(t,1)\|^2_{L^2(0,T)} \Big) \\
		\leq C s^7 \lambda^7 \|e^{-s\widehat \vphi_m} (\xi^*_m)^{7/2} \zeta\|^2_{L^2(0,T;H^1(0,1))} + C s^4 \lambda^4 \|e^{-s\widehat \vphi_m} (\xi^*_m)^{2} \zeta\|^2_{L^2(0,T;H^4(0,1))}.    
	\end{multline} 
	
	So, thanks to the estimates \eqref{subestimate}, \eqref{regularity_estimate} and \eqref{regularity-2}, we have   from \eqref{carleman-zeta_x-auxi}, 
	\begin{multline}\label{carleman-zeta_x-auxi-2}
		s^7 \lambda^8 \iint \Big(e^{-2s \vphi_m } \xi_m^7 |\zeta_x|^2 + 	 e^{-2s\widehat \vphi_m } (\xi^*_m)^7 |\zeta|^2\Big)   
		+ s^{5-\frac{2}{m}} \lambda^5 \|e^{-s\widehat \vphi_m} (\xi^*_m)^{\frac{5}{2}-\frac{1}{m}} \zeta\|^2_{L^2(0,T; H^4(0,1))} \\
		\leq  
		C 	 \bigg[ \iint \Big( s^{5-\frac{2}{m}} \lambda^5 e^{-2s\widehat \vphi_m} (\xi^*_m)^{5-\frac{2}{m}} |\theta_x|^2 +  e^{-2s\vphi_m} |\theta_{xx}|^2\Big)
		\\	+ s^7 \lambda^5 \iint e^{-2s\widehat \vphi_m}(\xi^*_m)^7 |\zeta|^2
		+ \iint \Big( s^{5-\frac{2}{m}} \lambda^5 \xi^{5-\frac{2}{m}}   + s^2 \lambda^2\xi^2 \Big) e^{-2s \vphi_m}|F_{3}|^2 \\
			+ s^7 \lambda^7 \iint  e^{-2s\widehat \vphi_m}(\xi^*_m)^7 |\zeta_x|^2 
		+ s^4\lambda^4\|e^{-s\widehat \vphi} (\xi^*_m)^2 \zeta\|^2_{L^2(0,T;H^4(0,1))} + 
		s^7\lambda^8 \iint_{\omega_0} e^{-2s\vphi_m} \xi^7_m |\zeta_x|^2 \bigg]. 
	\end{multline}
	It is clear that the terms associated to $\zeta$ and $\zeta_x$ in $Q_T$ can be easily absorbed in terms of the l.h.s. of \eqref{carleman-zeta_x-auxi-2}. Then, using the facts that $s^{-2/m}\leq C/T^4$ (since $s\geq C T^{2m}$), $\xi^{-2/m}_m\leq CT^4$,  we can deduce the required Carleman inequality \eqref{eq: Carleman-zeta_x}.   
\end{proof}

Next, for the variable $\theta$, we write the following Carleman inequality (usual for the heat equation).
\begin{lemma}[Carleman inequality for $\theta$, the case $\alpha=0$] \label{lemma-carleman-theta-app-2}
	Let $\vphi_m$ and $\xi_m$ be defined as in \eqref{weight_function} with $m\geq 1$. Then, there exist positive constants $\bar \lambda_4$, $\bar s_4:=\bar \sigma_4 (T^{2m} +T^{2m-1})$ with some $\bar \sigma_4>0$ and $C$, such that we have the following estimate for $\theta\in L^2(0,T; H^2(0,1)\cap H^1_0(0,1)) \cap H^1(0,T; L^2(0,1))$,
	\begin{align}\label{carleman_theta-app-2}
		I_H(\theta;9)
		\leq C \left(s^6\lambda^6 \iint e^{-2s\vphi_m} \xi^6_m (|\zeta_x|^2+ |F_4|^2) + s^9 \lambda^{10} \iint_{\omega_0}  e^{-2s\vphi_m} \xi^9_m |\theta|^2   \right)
	\end{align}
	for all $\lambda\geq \bar \lambda_4$ and $s\geq \bar s_4$, where $I_{H}(\cdot, \cdot)$ is defined by \eqref{Notation_heat}.     
\end{lemma}	

\smallskip 

\begin{proof}[\bf Proof of \Cref{Theorem-Carleman-second-app}]
	To obtain the Carleman estimate \eqref{Carleman-second-app}, let us first add all four Carleman estimates \eqref{carleman_u-app-2}, \eqref{carleman_w-app-2}, \eqref{eq: Carleman-zeta_x} and \eqref{carleman_theta-app-2}, which yields 
	\begin{multline}\label{Add-four-carlemans}
		I_{KS}(u) + I_{H}(w;7) 
		+ 	s^7 \lambda^8 \iint \left(e^{-2s\vphi_m } \xi_m^7 |\zeta_x|^2 + e^{-2s\widehat \vphi_m} (\xi_m^*)^7 |\zeta|^2\right) \\
		+
		s^{5-\frac{2}{m}} \lambda^5 \|e^{-s\widehat \vphi_m} (\xi^*_m)^{\frac{5}{2}-\frac{1}{m}} \zeta\|_{L^2(0,T; H^4(0,1))}
		+    I_{H}(\theta;9) 
		\\
		\leq 
		C \bigg[   \iint e^{-2s\vphi_m} \left(|w_x|^2+ s^4\lambda^4\xi^4_m|u_x|^2 +s^4\lambda^4 \xi^4_m|\theta|^2 + s^5\lambda^5\xi^5_m|\theta_x|^2  + s^6\lambda^6 \xi^6_m |\zeta_x|^2 \right)  \\
		+  \iint e^{-2s\vphi_m} \left( |F_1|^2 + s^4\lambda^4 \xi^4_m|F_2|^2 + s^5\lambda^5 \xi^5_m |F_3|^2 +s^6\lambda^6 \xi^6_m |F_4|^2 \right) 
		\\
		+ T^{10m} \iint e^{-2s\vphi_m} \xi^5_m |\theta_{xx}|^2
		+ s^7 \lambda^8 \iint_{\omega_0}  e^{-2s\vphi_m}\xi_m^7 |u|^2
		+ s^7 \lambda^8 \iint_{\omega_0}  e^{-2s\vphi_m}\xi_m^7 |w|^2
		\\
		+  s^7 \lambda^8 \iint_{\omega_0}  e^{-2s\vphi_m}\xi_m^7 |\zeta_x|^2
		+ s^9 \lambda^{10} \iint_{\omega_0}  e^{-2s\vphi_m}\xi_m^9 |\theta|^2
		\bigg] ,
	\end{multline}
	for all $\lambda \geq \lambda^*:= \max\{\bar \lambda_j: 1\leq j\leq 4\}$ and $s\geq  \bar c_1(T^m+T^{2m} + T^{2m-1} +T^{2m-2/5})$ for some $\bar c_1\geq \max\{\bar \sigma_j: 1\leq j\leq 4\}$. 
	
	\smallskip 
	{\bf Step 1: Absorbing the lower order integrals.}
	Using the fact that $1\leq CT^{2m}\xi_m$, we can deduce that 
	\begin{multline}\label{sourec_integrals}
		 \iint e^{-2s\vphi_m} \left(|w_x|^2+ s^4\lambda^4\xi^4_m|u_x|^2 +s^4\lambda^4 \xi^4_m|\theta|^2 + s^5\lambda^5\xi^5_m|\theta_x|^2  + s^6\lambda^6 \xi^6_m |\zeta_x|^2 \right)\\
		+ T^{10m} \iint e^{-2s\vphi_m} \xi^5_m |\theta_{xx}|^2 
		\\
		 \leq 
		C  \iint e^{-2s\vphi_m} \big(T^{10m} \xi^5_m |w_x|^2 + s^4\lambda^4 T^{2m}\xi^5_m|u_x|^2 + s^4\lambda^4 T^{18m}\xi^9_m|\theta|^2 \big) \\ 
		+ 	C  \iint e^{-2s\vphi_m}  \big(s^5\lambda^5 T^{4m} \xi^7_m|\theta_x|^2 + s^6\lambda^6 T^{2m} \xi^7_m|\zeta_x|^2 + T^{10m} \xi^5_m |\theta_{xx}|^2 \big),
	\end{multline}
	and thus by 
	choosing $\lambda\geq \lambda^*$ and $s\geq \bar c_2 T^{2m}$ for some $\bar c_2>0$, one can absorb all the source integrals appearing in the r.h.s. of \eqref{sourec_integrals} by the associated higher order integrals in the l.h.s. of \eqref{Add-four-carlemans}. 
	
	\medskip 
	
	{\bf Step 2: Absorbing the observation integral associated to $\zeta_x$.}
	Using the equation $\zeta_x= \theta_t -\theta_{xx}+\beta \theta_x-F_4$, and then following the techniques developed in \cite[Step 4--Proposition 3.6]{Cerpa_Careno}, we can eliminate the observation integral of $\zeta_x$. The resultant estimate can be then written as 
	\begin{multline}\label{Carleman_without_zeta_x}
		I_{KS}(u) + I_{H}(w;7) 
		+ 	s^7 \lambda^8 \iint \left(e^{-2s\vphi_m } \xi_m^7 |\zeta_x|^2 + e^{-2s\widehat \vphi_m} (\xi_m^*)^7 |\zeta|^2\right) 
		+    I_{H}(\theta;9) 
		\\
		\leq C \bigg[ s^7 \lambda^8 \iint_{\omega_0}  e^{-2s\vphi_m}\xi_m^7 |u|^2
		+ s^7 \lambda^8 \iint_{\omega_0}  e^{-2s\vphi_m}\xi_m^7 |w|^2 + s^{23}\lambda^{16}\iint_{\omega_0} e^{-6s \vphi_m + 4s\widehat \vphi_m} \xi^{23}_m  |\theta|^2  \\
		+ \iint e^{-2s\vphi_m} \left(|F_1|^2 + s^4\lambda^4 \xi^4_m |F_2|^2 + s^5\lambda^5  \xi^5_m|F_3|^2 + s^7\lambda^8 \xi^7_m |F_4|^2\right) \bigg],
	\end{multline}
	for any $\lambda \geq \lambda^*$ and  $s\geq \bar c_3(T^m+T^{2m} + T^{2m-1} +T^{2m-2/5})$, for some $\bar c_3>0$.

	\medskip
	
	{\bf Step 3: Absorbing the observation integral associated to $\theta$.} 
	We choose a smooth function $\phi\in \C^\infty_c(\omega_0)$ such that $0\leq \phi\leq 1$ in $\omega_0$ and $\phi=1$ in $\widehat \omega \subset \subset \omega_0$ for some  $\widehat \omega \neq \emptyset$ open set and without loss of generality, let us consider the Carleman estimate \eqref{Carleman_without_zeta_x} with the observation domain $\widehat \omega$.

	Then, we   focus on the term 
	\begin{align}\label{obs_theta_1}
		J:= s^{23}\lambda^{16}	\iint_{\widehat \omega} e^{-6s\vphi_m + 4s\widehat \vphi_m} \xi^{23}_m  |\theta|^2 \leq s^{23}\lambda^{16}	\iint_{\omega_0} \phi e^{-6s\vphi_m + 4s\widehat \vphi_m} \xi^{23}_m |\theta|^2.
	\end{align}
	From the second equation in \eqref{adj_sys_a_0}, we have 
	\begin{align*}
		\theta = -w_t - w_{xx} -\beta w_x + u_x - F_2 \quad \text{in } \oo \ \text{(consequently in $\omega_0$)}.
	\end{align*}
	Using this, we see
	\begin{align*}
		&s^{23}\lambda^{16}	\iint_{\omega_0} \phi e^{-6s\vphi_m + 4s\widehat \vphi_m} \xi^{23}_m  |\theta|^2 \\
		= & s^{23}\lambda^{16}	\iint_{\omega_0} \phi  e^{-6s\vphi_m + 4s\widehat \vphi_m} \xi^{23}_m  \theta ( u_x -w_t -w_{xx} -\beta w_x -F_2 )\\
		= & : \sum_{1\leq i\leq 5} J_{i}.
	\end{align*}
	
\smallskip
\noindent 
--  {\em Estimate for $J_1$.} We first look into the term $J_{1}$. 
	\begin{align*}
		J_{1}&=  s^{23}\lambda^{16}	\iint_{\omega_0} \phi e^{-6s\vphi_m + 4s \widehat \vphi_m} \xi^{23}_m   \theta u_x \\
		&= -s^{23}\lambda^{16}	\iint_{\omega_0} (\phi e^{-6s\vphi_m + 4s\widehat \vphi_m} \xi^{23}_m)_{x}  \theta u -   s^{23}\lambda^{16}	\iint_{\omega_0} \phi e^{-6s\vphi_m + 4s\widehat \vphi_m}\xi^{23}_m  \theta_x u.
	\end{align*}
Now, using \eqref{weight_estimates}, we observe  that for any $n\in \mathbb N^*$ and $b>0$,  the $n$-th derivative with respect to $x$ can be estimated as follows: 
\begin{align}\label{deri-esti}
	\left|\big( \phi e^{-2s\vphi_m} \xi^b_{m}   \big)_{n,x}    \right| \leq C s^n \lambda^n \, e^{-2s\vphi_m} \xi^{b+n}_m.     
\end{align}
In particular, 
	\begin{align}\label{esti-deri-app-2}
		|(\phi e^{-6s\vphi_m} \xi^{23}_m)_{n,x}| \leq C s^n\lambda^n e^{-6s\vphi_m} \xi^{23+n}_m, \quad \text{for } n \in \mathbb N^*.
	\end{align}
	Then, applying the Young's inequality with any $\epsilon>0$, we obtain   
	\begin{multline}\label{esti_J_1}
		|J_{1}| \leq C s^{24}\lambda^{17} \iint_{\omega_0}  \phi e^{-6s\vphi_m + 4s\widehat \vphi_m} \xi^{24}_m  |\theta u|  +  s^{23}\lambda^{16}  \iint_{\omega_0}  \phi e^{-6s\vphi_m + 4s\widehat \vphi_m} \xi^{23}_m  |\theta_x u|\\ 
		\leq \epsilon s^9 \lambda^{10}  \iint e^{-2s\vphi_m} \xi^9_m |\theta|^2  + \epsilon s^7 \lambda^{8}  \iint e^{-2s\vphi_m} \xi^7_m |\theta_x|^2 +  \frac{C}{\epsilon} s^{39}\lambda^{24} \iint_{\omega_0} e^{-10s\vphi_m + 8s\widehat \vphi_m} \xi^{39}_m|u|^2.
	\end{multline}

	%
	
\smallskip
\noindent 
{--  \em Estimate for $J_2$.} Next, we look into the term $J_{2}$, 
	\begin{align*}
		J_{2} & = - s^{23}\lambda^{16}	\iint_{\omega_0} \phi e^{-6s\vphi_m + 4s \widehat \vphi_m} \xi^{23}_m   \theta w_t \\
		& = s^{23}\lambda^{16} \iint_{\omega_0} \phi e^{-6s\vphi_m+4s\widehat \vphi_m} \xi^{23}_m \theta_t w +    s^{23}\lambda^{16}\iint_{\omega_0} \phi (e^{-6s\vphi_m+4s\widehat \vphi_m} \xi^{23}_m)_t \theta w,
	\end{align*}
	since $e^{-6s\vphi_m+ 4s\widehat \vphi_m}$ is $0$ at $t=0$ and $t=T$, due to the result in \Cref{Lemma-auxiliary}. 
	
	Now, recall the estimate \eqref{esti_time_deri}, so that one has 
	\begin{align*}
		\left|(e^{-6s\vphi_m+4s\widehat \vphi_m} \xi^{23}_m)_t  \right| \leq CT^{2m-2} s e^{-6s\vphi_m+4s\widehat \vphi_m} \xi^{25}_m.
	\end{align*}
	Thanks to this and for any $\epsilon>0$, we have (again, by using  the Young's inequality) that
	\begin{align}\label{esti_J_2}
		|J_{2}| \leq \epsilon s^5\lambda^6 \iint e^{-2s\vphi_m} \xi^5_m |\theta_t|^2 + \epsilon s^9\lambda^{10} \iint e^{-2s\vphi_m} \xi^9_m |\theta|^2 + \frac{C}{\epsilon} s^{41}\lambda^{26} \iint e^{-10s\vphi_m+ 8s\widehat \vphi_m} \xi^{41}_m|w|^2.
	\end{align}

\smallskip
\noindent 
{--  \em Estimate for $J_3$.} Let us now focus on $J_{3}$; upon consecutive integration by parts with respect to $x$, one can deduce that
	\begin{multline} \label{esti_J_3-aux}
		|J_{3}| \leq s^{23}\lambda^{16} \iint_{\omega_0} \left| \big(\phi e^{-6s\vphi_m + 4s\widehat \vphi_m} \xi^{23}_m\big)_{xx} \theta w \right| + 2 s^{23}\lambda^{16} \iint_{\omega_0} \left| \big(\phi e^{-6s\vphi_m + 4s\widehat \vphi_m} \xi^{23}_m\big)_{x} \theta_x w \right| \\
		+ s^{23}\lambda^{16} \iint_{\omega_0} \left|\phi e^{-6s\vphi_m + 4s\widehat \vphi_m} \xi^{23}_m \theta_{xx} w \right| \\
		\leq 
		C s^{25}\lambda^{18} \iint_{\omega_0} e^{-6s\vphi_m + 4s\widehat \vphi_m} \xi^{25}_m |\theta w| + C s^{24}\lambda^{17} \iint_{\omega_0} e^{-6s\vphi_m + 4s\widehat \vphi_m} \xi^{24}_m  |\theta_x w|  \\
		+C s^{23}\lambda^{16} \iint_{\omega_0}  e^{-6s\vphi_m + 4s\widehat \vphi_m} \xi^{23}_m |\theta_{xx} w|,
	\end{multline}
	thanks to the result \eqref{esti-deri-app-2}.
	
	Now, for any $\epsilon>0$, applying the Young's inequality, we have from \eqref{esti_J_3-aux},  
	\begin{multline} \label{esti_J_3}
		|J_3|
		\leq \epsilon s^9 \lambda^{10}  \iint e^{-2s\vphi_m} \xi^9_m |\theta|^2  + \epsilon s^7 \lambda^{8}  \iint e^{-2s\vphi_m} \xi^7_m |\theta_x|^2 +\epsilon s^5 \lambda^{6}  \iint e^{-2s\vphi_m} \xi^5_m |\theta_{xx}|^2 \\
		+   \frac{C}{\epsilon} s^{41}\lambda^{26} \iint_{\omega_0} \phi e^{-10s\vphi_m + 8s\widehat \vphi_m }\xi^{41}_m|w|^2. 
	\end{multline}

\smallskip
\noindent 
{--  \em Estimate for $J_4$.} Analogous to the estimate of $J_1$ in \eqref{esti_J_1},  the term $J_{4}$ satisfies 
\begin{multline}\label{esti_J_4}
	|J_{4}| 
	\leq \epsilon s^9 \lambda^{10}  \iint e^{-2s\vphi_m} \xi^9_m |\theta|^2  + \epsilon s^7 \lambda^{8}  \iint e^{-2s\vphi_m} \xi^7_m |\theta_x|^2  \\
	+  \frac{C}{\epsilon} s^{39}\lambda^{24} \iint_{\omega_0} e^{-10s\vphi_m + 8s\widehat \vphi_m} \xi^{39}_m|w|^2.
\end{multline}

\smallskip
\noindent 
{--  \em Estimate for $J_5$.} It is esay to observe that 
\begin{equation}\label{esti_J_5}
	\begin{aligned}
		|J_5| \leq \epsilon s^9\lambda^{10} \iint e^{-2s\vphi_m} \xi_m^9 |\theta|^2 + \frac{C}{\epsilon} s^{37} \lambda^{22} \iint e^{-10 s\vphi_m + 8 s\widehat \vphi_m} \xi^{37}_m |F_2|^2 ,
	\end{aligned}
\end{equation}
for any given $\epsilon>0$. 
	
	\smallskip

	Finally, gathering  the estimates of $J_1$, $J_2$, $J_3$, $J_4$ and $J_5$  given by \eqref{esti_J_1}, \eqref{esti_J_2}, \eqref{esti_J_3}, \eqref{esti_J_4} and \eqref{esti_J_5} respectively, we have 
	\begin{multline}\label{esti_theta_observation}
		s^{23}\lambda^{16}	\iint_{\omega_0} \phi e^{-6s\vphi_m + 4s\widehat \vphi_m} \xi^{23}_m |\theta|^2 
		\leq C \epsilon  s^{9}\lambda^{10} \iint e^{-2s\vphi_m} \xi^9_m |\theta|^2 \\ 
		+ C \epsilon s^{7}\lambda^{8} \iint e^{-2s\vphi_m} \xi^7_m |\theta_x|^2
		+ C \epsilon s^{5}\lambda^{6} \iint e^{-2s\vphi_m} \xi^5_m \big(|\theta_t|^2+|\theta_{xx}|^2\big)\\
		 + 
		\frac{C}{\epsilon} \iint_{\omega_0} e^{-10s\vphi_m + 8s\widehat \vphi_m} \Big(s^{39}\lambda^{24} \xi^{39}_m |u|^2 + s^{41}\lambda^{26}\xi^{41}_m |w|^2  \Big) +  \frac{C}{\epsilon} s^{37}\lambda^{22}\iint  e^{-10s\vphi_m + 8s\widehat \vphi_m} \xi^{37}_m |F_2|^2  .
	\end{multline}
	Fix $\epsilon>0$ small enough so that,  the integrals in $Q_T$ can be absorbed in terms of the l.h.s. in \eqref{Carleman_without_zeta_x}, and this yields the required Carleman inequality \eqref{Carleman-second-app} for every $\lambda\geq \lambda^*$ and $s\geq \sigma^* (T^m + T^{2m-1} +T^{2m-2/5} + T^{2m}),$ for some $\lambda^*>0$ and  $\sigma^*>0$ chosen largely enough. 
	
	Hence, the proof of \Cref{Theorem-Carleman-second-app} is finished.
\end{proof}

\medskip

\subsection{Carleman estimate for the case when $\boldsymbol{\alpha=1}$}\label{carleman-alpha=1}

The adjoint system \eqref{adj_sys} for the case $\alpha=1$ reads as 
\begin{align}\label{adj_sys_a_1}
	&\begin{dcases}
		-u_t + u_{xxxx} +\gamma u_{xx} = F_1 -w_x +   \zeta \mathds{1}_{\mathcal O}  & \textnormal{in } Q_T, \\
		-w_t - w_{xx} - \beta w_x =  F_2 -u_x  & \textnormal{in } Q_T,\\
		\zeta_t + \zeta_{xxxx} +\gamma \zeta_{xx} = F_3+\theta_x  & \textnormal{in } Q_T, \\
		\theta_t - \theta_{xx} + \beta \theta_x = F_4+ \zeta_x & \textnormal{in } Q_T,\\
		u=u_x=w=\zeta=\zeta_x=\theta=0        &\textnormal{in } \Sigma_T, \\
		u(T)=0, \ \ w(T) =0, &\textnormal{in } (0,1),
		\\ \zeta(0)=\zeta_0, \ \ \theta(0) =\theta_0  &\textnormal{in } (0,1). 
 	\end{dcases}
\end{align}

We now write the Carleman inequality associated to the system \eqref{adj_sys_a_1}. 
In this case we also consider $(\zeta_0,\theta_0)\in [L^2(0,1)]^2$ and $F_1, F_2, F_3\in L^2(Q_T)$ but  we need to consider $F_4$ to be slightly mpre regular, namely $F\in L^2(Q_T)$ such that $\iint e^{-2s\vphi_m} (|F_4|^2+|F_{4,x}|^2) < +\infty$. 
\begin{theorem}[Carleman inequality: the case $\alpha=1$]  \label{theorem-carleman}
	Let $\vphi_m$ and $\xi_m$ be given by \eqref{weight_function} with $m>3$. Then, there exist positive constants  $\lambda_0$, $s_0:=\sigma_0(T^m+T^{2m}+T^{2m-1}+ T^{2m-2/5}+T^{4m/3}+T^{3m/2})$ with some $\sigma_0>0$ and $C$, such that we have the following inequality satisfied by the solution to \eqref{adj_sys_a_1}: 
	\begin{multline}\label{Carleman-first-approach}
		I_{KS}(u) + I_H(w;3) + I_{KS}(\zeta) 
		+	s^3\lambda^4\iint e^{-2s\vphi_m} \xi^3_m |\theta_x|^2 + s\lambda^2 \iint e^{-2s\vphi_m} \xi_m |\theta_{xx}|^2 
		\\ 
		\leq
		 C \iint e^{-2s\vphi_m} \left(s^{71}\lambda^{72}\xi^{71}_m|F_1|^2 + |F_2|^2   + s^3\lambda^4 \xi^3_m|F_3|^2\right) \\
		 +Cs\lambda^2\iint e^{-2s\vphi_m}\xi^2_m\left( |F_4|^2 + |F_{4,x}|^2 \right) 
	+ C s^{79}\lambda^{80}\iint_{\omega_0}  e^{-2s\vphi_m} \xi^{79}_m |u|^2 +    Cs^{73}\lambda^{74}\iint_{\omega_0}  e^{-2s\vphi_m} \xi^{73}_m |w|^2  ,
	\end{multline}
	for all $\lambda\geq \lambda_0$ and $s\geq s_0$, where $I_{KS}(\cdot)$ and $I_{H}(\cdot; \cdot)$ are introduced in \eqref{Notation_KS} and \eqref{Notation_heat} respectively. 
\end{theorem}

\smallskip 

To prove the above Carleman inequality with the observations only on $u$ and $w$, we do the following steps.
\begin{itemize}
	\item[(i)]  First, observe that the usual Carleman estimate for $\theta$ will always give an observation integral of $\theta$ and there is no chance to absorb it by any leading integrals. In fact, there is a coupling by $\theta_x$ to the equation of $\zeta$ and therefore, it is relevant to seek for a Carleman estimate associated with the variable $\theta_x$. 
	
	In this context, we recall the work \cite{Cerpa-Mercado-Pazoto}, where they proved such a Carleman estimate to demonstrate a joint Carleman inequality for the adjoint to the KS system coupled with a heat equation and we shall use it in our present article. 
	
	\item[(ii)] For the variables $u$ and $\zeta$,  we use a Carleman estimate from the work \cite{Zhou} or \cite{Cerpa-Mercado-Pazoto} as mentioned earlier, see \Cref{thm-Zhou} in the present paper.
	
	\item[(iii)] Finally,	for $w$, we make use the standard Carleman inequality for the heat equation, see \Cref{thm-Fur-Ima} (thanks to the pioneering work \cite{Fur-Ima}).
	
	
\end{itemize}

\medskip 

Below, we prescribe the individual  Carleman estimates for each of $u$, $w$, $\zeta$ and $\theta_x$.

\begin{lemma}[Carleman inequality for $u$, the case $\alpha=1$]\label{lemma-carleman-u}
	Let $\vphi_m$ and $\xi_m$ be as given by  \eqref{weight_function} with $m\geq 2/5$. Then, there exist positive constants $\lambda_1$, $s_1:=\sigma_1 (T^{2m}+T^{2m-2/5})$ with some $\sigma_1>0$ and $C$, such that we have the following estimate for $u\in L^2(0,T;H^4(0,1)\cap H^2_0(0,1)) \cap H^1(0,T; L^2(0,1))$,
	\begin{align} \label{carleman_u}
		I_{KS}(u)
		\leq C \left(\iint e^{-2s\vphi_m} \Big(|F_1|^2+|w_x|^2 + |\zeta|^2\Big) +s^7\lambda^8\iint_{\omega_0}  e^{-2s\vphi_m} \xi^7_m |u|^2      \right)
	\end{align}
	for all $\lambda\geq \lambda_1$, $s\geq s_1$, where $I_{KS}(\cdot)$ is defined by \eqref{Notation_KS}.    
\end{lemma}

The next lemma is concerned with a Carleman estimate for $w$. 
\begin{lemma}[Carleman inequality for $w$, the case $\alpha=1$]\label{lemma-carleman-w}
	Let $\vphi_m$ and $\xi_m$ be defined as in \eqref{weight_function} with $m\geq 1$. Then, there exist positive constants $\lambda_2$, $s_2:=\sigma_2 (T^{2m} +T^{2m-1})$ with some $\sigma_2>0$ and $C$, such that we have the following estimate for $w\in L^2(0,T; H^2(0,1)\cap H^1_0(0,1)) \cap H^1(0,T; L^2(0,1))$,
	\begin{align}\label{carleman_w}
		I_H(w;3)
		\leq C \left(\iint e^{-2s\vphi_m}(|F_2|^2+|u_x|^2) + s^3\lambda^4\iint_{\omega_0}  e^{-2s\vphi_m} \xi^3_m |w|^2      \right)
	\end{align}
	for all $\lambda\geq \lambda_2$ and $s\geq s_2$, where $I_{H}(\cdot ; \cdot)$ is defined by \eqref{Notation_heat}.     
\end{lemma}	

We also write a Carleman estimate for $\zeta$ which is similar with the one for $u$.
\begin{lemma}[Carleman inequality for $\zeta$, the case $\alpha=1$]\label{lemma-carleman-zeta}
	Let $\vphi_m$ and $\xi_m$ be defined as in \eqref{weight_function} with $m\geq 2/5$. Then, there exist positive constants $\lambda_3$, $s_3:=\sigma_3 (T^{2m}+T^{2m-2/5})$ with some $\sigma_3>0$ and $C$, such that we have the following estimate for $\zeta\in L^2(0,T;H^4(0,1)\cap H^2_0(0,1)) \cap H^1(0,T; L^2(0,1))$,
	\begin{align}\label{estimate_zeta}
		I_{KS}(\zeta)
		\leq C \left(\iint e^{-2s\vphi_m}(|F_3|^2+ |\theta_x|^2 )+ s^7\lambda^8\iint_{\omega_0}  e^{-2s\vphi_m} \xi^7_m |\zeta|^2      \right)
	\end{align}
	for all $\lambda\geq \lambda_3$ and $s\geq s_3$, where $I_{KS}(\cdot)$ has been defined in \eqref{Notation_KS}. 
\end{lemma}	

Lastly, by following \cite[Theorem 3.1]{Cerpa-Mercado-Pazoto}, we have a Carleman estimate for $\theta_x$ as given below, which can be proved using a result given by \cite[Lemma 6]{Gue13}. At this point, we can start with regular enough initial data $\theta_0$ in the equation \eqref{adj_sys_a_1} so that $\theta \in L^2(0,T; H^3(0,1)\cap H^1_0(0,1)) \cap H^1(0,T; H^1(0,1))$.
\begin{lemma}[Carleman inequality for $\theta_x$, the case $\alpha=1$] \label{lemma-carleman-theta_x}
	Let $\vphi_m$ and $\xi_m$ be defined by \eqref{weight_function} with $m>3$, $k>m$. Then, there exist positive constants $\lambda_4$, $s_4:=\sigma_4(T^{2m}+T^{2m-1})$ with some $\sigma_4>0$ and $C$, such that we have the following estimate for $\theta \in L^2(0,T; H^3(0,1)\cap H^1_0(0,1)) \cap H^1(0,T; H^1(0,1))$, 
	\begin{multline} \label{estimate_theta}
		s^3\lambda^4\iint e^{-2s\vphi_m} \xi^3_m |\theta_x|^2 + s\lambda^2 \iint e^{-2s\vphi_m} \xi_m |\theta_{xx}|^2 \\ 
		\leq C s\lambda^2 \iint e^{-2s\vphi_m}\xi^2_m \left(  |F_4|^2+|F_{4,x}|^2+|\zeta_x|^2  + |\zeta_{xx}|^2 \right) + C s^3\lambda^4 \iint_{\omega_0}  e^{-2s\vphi_m}\xi^3_m |\theta_x|^2, 
	\end{multline}
	for all $\lambda\geq \lambda_4$, $s\geq s_4$.    
\end{lemma}

\medskip 

Now, we are in the situation to prove our main Carleman inequality, that is  \Cref{theorem-carleman}.

\medskip

\begin{proof}[\bf{Proof of \Cref{theorem-carleman}}]  We divide it into several steps. 
	
	\smallskip 
	
	{\bf Step 1: Absorbing the lower order integrals.} Observe the following result:
	\begin{align*}
		1 \leq T^{2pm} \xi^p_m, \quad \forall p\in \mathbb N^*.
	\end{align*} 

--	Using this,  the lower order integrals in the r.h.s.  of the Carleman inequality \eqref{carleman_u} can be estimated as 
	\begin{align}\label{auxi-3}
		\iint e^{-2s\vphi_m} \Big(  |w_x|^2 + |\zeta|^2 \Big) \leq T^{2m}	\iint e^{-2s\vphi_m} \xi_m |w_x|^2 + T^{14 m}\iint e^{-2s\vphi_m} \xi^7_m |\zeta|^2,
	\end{align}
	whereas, the source integral in the r.h.s. of \eqref{carleman_w} satisfies
	\begin{align}\label{auxi-4}
		\iint e^{-2s\vphi_m}  |u_x|^2 \leq T^{10m}	\iint e^{-2s\vphi_m} \xi^5_m |u_x|^2.
	\end{align}
	Thus, by choosing any $\lambda\geq \lambda_0:= \max \{ \lambda_j: 1\leq j\leq 4\}$ fixed and $s\geq \sigma_5 T^{2m}$ for some $\sigma_5>0$, one can absorb all the integrals appearing in the r.h.s. of \eqref{auxi-3} and \eqref{auxi-4} by the associated leading integrals in the l.h.s. of \eqref{carleman_u} \eqref{carleman_w} and \eqref{estimate_zeta}.

	\smallskip 
	
--	Next,  the source integral in the r.h.s. of \eqref{estimate_zeta} enjoys 
	\begin{align}\label{auxi-1} 
		\iint e^{-2s\vphi_m} |\theta_x|^2 \leq T^{6m} \iint e^{-2s\vphi_m}\xi^3_m |\theta_x|^2,
	\end{align}
	%
	%
and the r.h.s. of  the Carleman inequality \eqref{estimate_theta} (for $\theta_x$) can be estimated as    
	\begin{multline}\label{auxi-2} 
		s\lambda^2\iint e^{-2s\vphi_m} \xi^2_m \left( |\zeta_x|^2 + |\zeta_{xx}|^2 \right)
		\\
		\leq 
		 s\lambda^2 \left(T^{6m} \iint e^{-2s\vphi_m} \xi^5_m|\zeta_x|^2 + T^{2m}\iint e^{-2s\vphi_m} \xi^3_m |\zeta_{xx}|^2\right).
	\end{multline}
	Then, by choosing any $\lambda\geq \lambda_0$ and $s\geq \sigma_6 (T^{m}+T^{3m/2})$ for some $\sigma_6>0$, the quantity in \eqref{auxi-1} can be absorbed by the 1st leading integral in the l.h.s. of \eqref{estimate_theta} and the quantities appearing in the r.h.s. of \eqref{auxi-2}, by the associated  leading integrals in the l.h.s. of \eqref{carleman_w} and \eqref{estimate_zeta}.

	\medskip

	So, after adding the inequalities: \eqref{carleman_u}, \eqref{carleman_w}, \eqref{estimate_zeta} and \eqref{estimate_theta}, and using the above absorption techniques,  we  obtain the following auxiliary estimate:
	\begin{multline}\label{Add_theta_zeta}
		I_{KS}(u) + I_{H}(w,3)+I_{KS}(\zeta) 
		+	s^3\lambda^4\iint e^{-2s\vphi_m} \xi^3_m |\theta_x|^2 + s\lambda^2 \iint e^{-2s\vphi_m} \xi_m |\theta_{xx}|^2 
		\\ 
		\leq C s^7\lambda^8\iint_{\omega_0}  e^{-2s\vphi_m} \xi^7_m |u|^2 +    Cs^3\lambda^4\iint_{\omega_0}  e^{-2s\vphi_m} \xi^3_m |w|^2   +Cs^7\lambda^8\iint_{\omega_0}  e^{-2s\vphi_m} \xi^7_m |\zeta|^2   \\  + Cs^3\lambda^4\iint_{\omega_0}  e^{-2s\vphi_m} \xi^3_m |\theta_x|^2 + C \iint e^{-2s\vphi_m} (|F_1|^2 + |F_2|^2   + |F_3|^2 ) \\
		+ C s\lambda^2 \iint e^{-2s\vphi_m}\xi^2_m (|F_4|^2 + |F_{4,x}|^2 ) , 
	\end{multline}
	for all $\lambda\geq \lambda_0$ and $s\geq \sigma_0(T^m+ T^{2m} + T^{2m-1}+T^{2m-2/5} +T^{3m/2}+T^{4m/3})$ where $\sigma_0:= \max\{\sigma_j \, ; \, 1 \leq j \leq 6\}$.  
	
	\smallskip 
	
	Now, our  duty is to absorb the observation integrals associated with $\theta_x$ and $\zeta$ by some leading integrals in the l.h.s. of \eqref{Add_theta_zeta}.  
	
	\smallskip 
	
	{\bf Step 2:  Absorbing the observation integral associated to $\theta_x$.} 	
	In the sequel, we choose a nonempty set $\widehat \omega_2\subset \subset  \widehat \omega_1 \subset \subset  \omega_0$ and a function 
	\begin{align*}
		\phi\in \C^\infty_c(\widehat \omega_1) \ \text{ with } \ 0\leq \phi\leq 1 \text{ in } \widehat \omega_1, \ \   \phi=1 \text{ in } \widehat \omega_2, 
	\end{align*}
	%
	and we consider the auxiliary Carleman estimate \eqref{Add_theta_zeta} with the observation domain $\widehat \omega_2$.  
	
	\smallskip 
	
	Our goal is to eliminate the observation integral of $\theta_x$.
	Using the equation of $\zeta$ from \eqref{adj_sys_a_1}, we have
	\begin{align}\label{estimate_theta_x}  
		& s^3\lambda^4\iint_{\widehat\omega_2}  e^{-2s\vphi_m} \xi^3_m |\theta_x|^2 \\ \notag  
		\leq &  s^3\lambda^4\iint_{\widehat\omega_1} \phi e^{-2s\vphi_m} \xi^3_m \theta_x (\zeta_t +\zeta_{xxxx}+ \gamma \zeta_{xx}-F_3)       \\ \notag  
		= 	& s^3\lambda^4 \iint_{\widehat\omega_1} \phi e^{-2s\vphi_m} \xi^3_m \theta_x (\zeta_t + \zeta_{xx}) +  s^3\lambda^4 \iint_{\widehat\omega_1} \phi e^{-2s\vphi_m} \xi^3_m \theta_x (\zeta_{xxxx} + (\gamma-1)\zeta_{xx})  \\ \notag 
		& \  - s^3\lambda^4 \iint_{\widehat \omega_1} \phi e^{-2s\vphi_m} \xi^3_m \theta_x F_3   \\ \notag 
		:=& I_1+I_2+I_3 .
	\end{align} 

\smallskip 
\noindent 
	-- {\em Estimate for $I_1$.}  We have 
	\begin{multline}\label{esti_I-1-1}
		I_{1}= 	s^3\lambda^4 \iint_{\widehat\omega_1} \phi e^{-2s\vphi_m} \xi^3_m \theta_x (\zeta_t + \zeta_{xx})\\
		= -s^3\lambda^4 \iint_{\widehat\omega_1} \phi (e^{-2s\vphi_m} \xi^3_m)_t \theta_x \zeta -s^3\lambda^4 \iint_{\widehat\omega_1} \phi e^{-2s\vphi_m} \xi^3_m (\theta_t)_x \zeta - s^3\lambda^4  \iint_{\widehat\omega_1} (\phi e^{-2s\vphi_m} \xi^3_m)_x \theta_x \zeta_x \\
		+ s^3\lambda^4 \iint_{\widehat\omega_1} (\phi e^{-2s\vphi_m} \xi^3_m )_x \theta_{xx} \zeta +  s^3\lambda^4 \iint_{\widehat\omega_1} \phi e^{-2s\vphi_m} \xi^3_m  (\theta_{xx})_x \zeta.
	\end{multline} 
	In above, there is no boundary terms at $t=0, T$ (while integrating by parts in time), since $e^{-2s\vphi_m(t,x)} \to 0$ as $t\to 0^+$ or $T^{-}$. We also perfomed
	two consecutive integration by parts in space as follows
	\begin{multline*}
		s^3\lambda^4 \iint_{\widehat\omega_1} \phi e^{-2s\vphi_m} \xi^3_m \theta_x \zeta_{xx} = - s^3\lambda^4  \iint_{\widehat\omega_1} (\phi e^{-2s\vphi_m} \xi^3_m)_x \theta_x \zeta_x - s^3\lambda^4  \iint_{\widehat\omega_1} \phi e^{-2s\vphi_m} \xi^3_m \theta_{xx}
		\zeta_x
		\\
		=- s^3\lambda^4  \iint_{\widehat\omega_1} (\phi e^{-2s\vphi_m} \xi^3_m)_x \theta_x \zeta_x 	+ s^3\lambda^4 \iint_{\widehat\omega_1} (\phi e^{-2s\vphi_m} \xi^3_m )_x \theta_{xx} \zeta +  s^3\lambda^4 \iint_{\widehat\omega_1} \phi e^{-2s\vphi_m} \xi^3_m  (\theta_{xx})_x \zeta.
	\end{multline*}
	Now, going back to \eqref{esti_I-1-1}, and using the equation 
	$\theta_t-\theta_{xx}= -\beta\theta_x +\zeta_x$, we have 
	\begin{align}\label{esti_I-1-2}
		I_1 = -s^3\lambda^4 \iint_{\widehat \omega_1} \phi e^{-2s\vphi_m} \xi^3_m (-\beta \theta_x +\zeta_x)_x \zeta + X_1 , 
	\end{align}
	where 
	\begin{multline*}
		X_1 := -s^3\lambda^4 \iint_{\widehat\omega_1} \phi (e^{-2s\vphi_m} \xi^3_m)_t \theta_x \zeta - s^3\lambda^4  \iint_{\widehat\omega_1} (\phi e^{-2s\vphi_m} \xi^3_m)_x \theta_x \zeta_x 
		+ s^3\lambda^4 \iint_{\widehat\omega_1} (\phi e^{-2s\vphi_m} \xi^3_m )_x \theta_{xx} \zeta.
	\end{multline*}
	Now, recall \eqref{weight_estimate-t-m} (for $m>1$) to write 
	$$|(e^{-2s\vphi_m} \xi^3_m)_t|\leq C T^{2m-2} s\xi_m^{2} (e^{-2s\vphi_m} \xi^3_m) . $$
	We use this fact in the first integral of $X_1$, and then thanks to the Cauchy-Schwarz inequality, we have
	for any $\epsilon>0$, that    
	\begin{multline}\label{esti_X-1}
		|X_1|\leq \epsilon s^3\lambda^4 \iint e^{-2s\vphi_m} \xi^3_m |\theta_x|^2 + \epsilon s\lambda^2 \iint e^{-2s\vphi_m}  \xi_m |\theta_{xx}|^2 \\
		+\frac{C}{\epsilon} s^5\lambda^6 \iint_{\widehat \omega_1}  e^{-2s\vphi_m} \xi^5_m |\zeta_x|^2 
		+ \frac{C}{\epsilon} s^5\lambda^6 \iint_{\widehat \omega_1} e^{-2s\vphi_m} \xi^5_m |\zeta|^2 .
	\end{multline}
	

\smallskip 

Let us estimate the other integrals of $I_1$ in \eqref{esti_I-1-2}. We have for any $\epsilon>0$ (again by applying Cauchy-Schwarz inequality) 
\begin{multline}\label{esti_auxi_I-1}
	s^3\lambda^4	\left| \iint_{\widehat \omega_1} \phi e^{-2s\vphi_m} \xi^3_m \Big(-\beta \theta_x +\zeta_x  \Big)_x \zeta   \right| \leq \epsilon s\lambda^2 \iint e^{-2s\vphi_m} \xi_m |\theta_{xx}|^2  \\
	+  \epsilon s^3\lambda^4 \iint e^{-2s\vphi_m} \xi^3_m |\zeta_{xx}|^2  
	+ \frac{C}{\epsilon} s^5\lambda^6 \iint_{\widehat \omega_1} e^{-2s\vphi_m} \xi^5_m |\zeta|^2.
\end{multline}

Therefore, the estimates \eqref{esti_X-1} and \eqref{esti_auxi_I-1} yields
\begin{multline}\label{esti_auxi_I_1-2}
	|I_1|\leq  \epsilon s^3\lambda^4 \iint e^{-2s\vphi_m} \xi^3_m |\theta_x|^2 + 2 \epsilon s\lambda^2 \iint e^{-2s\vphi_m}  \xi_m |\theta_{xx}|^2 
	+	\epsilon s^3\lambda^4 \iint e^{-2s\vphi_m} \xi^3_m |\zeta_{xx}|^2 
	 \\
	+ \frac{C}{\epsilon} s^5\lambda^6 \iint_{\widehat \omega_1} e^{-2s\vphi_m} \xi^5_m |\zeta_x|^2 
	+ \frac{C}{\epsilon} s^5\lambda^6 \iint_{\widehat \omega_1} e^{-2s\vphi_m} \xi^7_m |\zeta|^2 .
\end{multline}


\medskip
\noindent 
--  {\em Estimate for $I_2$.} Let us recall the quantity $I_2$ from \eqref{estimate_theta_x} and we have 
\begin{align*}
I_2&= s^3\lambda^4 \iint_{\widehat\omega_1} \phi e^{-2s\vphi_m} \xi^3_m \theta_x (\zeta_{xxxx} + (\gamma-1)\zeta_{xx}) \\
&=  - s^3\lambda^4 \iint_{\widehat\omega_1} (\phi e^{-2s\vphi_m} \xi^3_m)_x \theta_x \zeta_{xxx} -  s^3\lambda^4 \iint_{\widehat\omega_1} \phi e^{-2s\vphi_m} \xi^3_m\big( \theta_{xx} \zeta_{xxx} - (\gamma-1) \theta_x \zeta_{xx}\big) ,  
\end{align*}
where we perform an integration by parts on the term involving fourth order derivative in $\zeta$. It follows that
\begin{multline}\label{esti_I-2}
|I_2|\leq  \epsilon s^3\lambda^4 \iint_{\widehat\omega_1}  e^{-2s\vphi_m} \xi^3_m |\theta_x|^2 + \epsilon s\lambda^2 \iint_{\widehat\omega_1} e^{-2s\vphi_m} \xi_m |\theta_{xx}|^2 \\
+ \frac{C}{\epsilon} s^3\lambda^4 \iint_{\widehat\omega_1} e^{-2s\vphi_m} \xi^3_m |\zeta_{xx}|^2 + \frac{C}{\epsilon} s^5\lambda^6 \iint_{\widehat\omega_1} e^{-2s\vphi_m} \xi^5_m |\zeta_{xxx}|^2.
\end{multline}

\smallskip 
\noindent
-- {\em Estimate for $I_3$.} Finally, we have 
\begin{equation}\label{esti_I_3}
	\begin{aligned}
		|I_3| \leq \epsilon s^3\lambda^4 \iint_{\widehat \omega_1}  e^{-2s\vphi_m} \xi^3_m |\theta_x|^2 + \frac{C}{\epsilon} s^3 \lambda^4 \iint e^{-2s\vphi_m}\xi^3_m |F_3|^2,  
	\end{aligned}
\end{equation}
for any given $\epsilon>0$. 

Now, using the estimates of $I_1$, $I_2$ and $I_3$ given by \eqref{esti_auxi_I_1-2}, \eqref{esti_I-2} and \eqref{esti_I_3} respectively, we have from \eqref{estimate_theta_x} that 
\begin{multline}\label{esti_theta_x-1}
s^3\lambda^4 \iint_{\widehat \omega_1} e^{-2s\vphi_m}  \xi^3_m |\theta_x|^2 \leq C \epsilon s^3\lambda^4 \iint e^{-2s\vphi_m} \xi^3_m |\theta_x|^2 + C \epsilon s\lambda^2 \iint e^{-2s\vphi_m}  \xi_m |\theta_{xx}|^2 \\
+ C \epsilon s^3\lambda^4 \iint e^{-2s\vphi_m}  \xi^3_m |\zeta_{xx}|^2 +\frac{C}{\epsilon} s^5\lambda^6 \iint_{\widehat \omega_1} e^{-2s\vphi_m} \xi^5_m |\zeta|^2 
\\
+ \frac{C}{\epsilon} \iint_{\widehat \omega_1} \Big( s^5\lambda^6  e^{-2s\vphi_m} \xi^5_m |\zeta_x|^2 
+ s^3\lambda^4  e^{-2s\vphi_m} \xi^3_m |\zeta_{xx}|^2 
+  s^5\lambda^6 e^{-2s\vphi_m} \xi^5_m |\zeta_{xxx}|^2 \Big) \\
+ \frac{C}{\epsilon} s^3 \lambda^4 \iint e^{-2s\vphi_m}\xi^3_m |F_3|^2.
\end{multline}
Fix $\epsilon>0$ small enough so that we can absorb the first three integrals in the r.h.s. of \eqref{esti_theta_x-1} by the associated leading terms in the l.h.s. of \eqref{Add_theta_zeta}.

\smallskip

Next, one can deduce the following result: for any $\epsilon>0$, there exists $C>0$ such that,
\begin{multline}\label{esti_theta_x-2}
\iint_{\widehat \omega_1} \Big( s^5\lambda^6  e^{-2s\vphi_m} \xi^5_m |\zeta_x|^2 
+ s^3\lambda^4  e^{-2s\vphi_m} \xi^3_m |\zeta_{xx}|^2 
+  s^5\lambda^6 e^{-2s\vphi_m} \xi^5_m |\zeta_{xxx}|^2 \Big) \\
\leq \epsilon \iint e^{-2s\vphi_m}\Big( (s\xi_m)^{-1} |\zeta_{xxxx}|^2 + s\lambda^2 \xi_m |\zeta_{xxx}|^2  + s^3\lambda^4 \xi^3_m |\zeta_{xx}|^2  \Big)  + \frac{C}{\epsilon} s^{39}\lambda^{40}  \iint_{\omega_0} e^{-2s\vphi_m} \xi^{39}_m |\zeta|^2 ,
\end{multline}
assuming that there is a set $\widehat \omega_0$ such that $\widehat \omega_1 \subset \subset \widehat \omega_0 \subset \subset \omega_0$.

The above proof can be done by performing several integration by parts in space and applying the Cauchy-Schwarz inequalty accordingly. We omit the details here. 


\smallskip 

Then, for $\epsilon>0$ small enough, we can absorb the first three integrals in the r.h.s. of \eqref{esti_theta_x-2} by the corresponding leading integrals in the l.h.s. of \eqref{Add_theta_zeta} and as a consequence, we have 
\begin{multline}\label{Add_theta_zeta-2}
I_{KS}(u) + I_H(w,3) + I_{KS}(\zeta) 
+	s^3\lambda^4\iint e^{-2s\vphi_m} \xi^3_m |\theta_x|^2 + s\lambda^2 \iint e^{-2s\vphi_m} \xi_m |\theta_{xx}|^2 
\\ 
\leq C s^7\lambda^8\iint_{\omega_0}  e^{-2s\vphi_m} \xi^7_m |u|^2 +    Cs^3\lambda^4\iint_{\omega_0}  e^{-2s\vphi_m} \xi^3_m |w|^2   +Cs^{39}\lambda^{40}\iint_{\omega_0}  e^{-2s\vphi_m} \xi^{39}_m |\zeta|^2  \\
+ C \iint e^{-2s\vphi_m} (|F_1|^2 + |F_2|^2 )   
+ Cs^3\lambda^4 \iint e^{-2s\vphi_m} \xi^3_m|F_3|^2 \\
+ Cs\lambda^2 \iint e^{-2s\vphi_m}\xi^2_m( |F_4|^2 + |F_{4,x}|^2 ). 
\end{multline}

\medskip 

{\bf Step 3: Absorbing the observation integral associated to $\zeta$.} 
In the previous step, one can assume a couple of nonempty sets $\widehat \omega_2 \subset \subset \widehat \omega_1 \subset \subset \omega_1 \subset \subset \omega_0$ and prove the auxiliary inequality \eqref{Add_theta_zeta-2} with the observation domain $\omega_1$.  

Then,  choose a function $\phi_1 \in \C^\infty_c(\omega_0)$ with $0\leq \phi_1 \leq 1$ in $\omega_0$ and $\phi_1 =1$ in $\omega_1$. Also, one has from the 
$4\times 4$ adjoint system \eqref{adj_sys_a_1} that 
$$ \zeta = -u_t+ u_{xxxx} +\gamma u_{xx} + w_x - F_1 \ \text{ in }  \omega_0, \text{ since }  \omega_0 \subset \subset \oo.$$ 
Therefore, we observe that 
\begin{multline}\label{esti_zeta}
s^{39}\lambda^{40}\iint_{\omega_1}  e^{-2s\vphi_m} \xi^{39}_m |\zeta|^2 \leq s^{39}\lambda^{40}\iint_{\omega_0} \phi_1   e^{-2s\vphi_m} \xi^{39}_m |\zeta|^2 \\
= s^{39}\lambda^{40}\iint_{\omega_0}  \phi_1 e^{-2s\vphi_m} \xi^{39}_m \zeta( -u_t +u_{xxxx} +\gamma u_{xx} + w_x -F_1\big)  := I_4+ I_5 + I_6+ I_7+I_8.
\end{multline}

\smallskip
\noindent 
--  {\em Estimate for $I_4$.}
First, we compute 
\begin{multline}
I_4: = - s^{39}\lambda^{40}\iint_{\omega_0}  \phi_1 e^{-2s\vphi_m} \xi^{39}_m \zeta u_t = s^{39}\lambda^{40}\iint_{\omega_0}  \phi_1 e^{-2s\vphi_m} \xi^{39}_m \zeta_t  u  \\ 
+ s^{39}\lambda^{40}\iint_{\omega_0}  \phi_1 \big(e^{-2s\vphi_m} \xi^{39}_m\big)_t \zeta u. 
\end{multline}
Let us recall \eqref{weight_estimate-t-m} (for $m>1$), to write
$$|(e^{-2s\vphi_m} \xi^{39}_m)_t|\leq C  T^{2m-2} s e^{-2s\vphi_m} \xi^{41}_m . $$
Using the above fact and the Cauchy-Schwarz inequality, we have for any $\epsilon>0$ that 
\begin{align}\label{estimate_I-3}
|I_4| \leq \epsilon s^{-1} \iint e^{-2s\vphi_m} \xi^{-1}_m |\zeta_t|^2 +  \epsilon s^7 \lambda^8 \iint e^{-2s\vphi_m} \xi^{7}_m |\zeta|^2 + \frac{C}{\epsilon} s^{79}\lambda^{80} \iint_{\omega_0} e^{-2s\vphi_m} \xi^{79}_m |u|^2 . 
\end{align}

\medskip
\noindent 
-- {\em Estimate for $I_5$.}
Next, by performing a successive number of integration by parts on $I_5$ w.r.t. $x$, we get
\begin{multline}\label{estimate_I-4}
|I_5| : = \left| s^{39}\lambda^{40}\iint_{\omega_0}  \phi_1 e^{-2s\vphi_m} \xi^{39}_m \zeta u_{xxxx} \right| \leq C \Big(  s^{39} \lambda^{40} \iint_{\omega_0}  \big(\phi_1 e^{-2s\vphi_m} \xi^{39}_{m}\big)_{xxxx} \zeta u  \\ +  s^{39} \lambda^{40} \iint_{\omega_0}  \big(\phi_1 e^{-2s\vphi_m} \xi^{39}_{m}\big)_{xxx} \zeta_x u   +  s^{39} \lambda^{40} \iint_{\omega_0}  \big(\phi_1 e^{-2s\vphi_m} \xi^{39}_{m}\big)_{xx} \zeta_{xx} u  \\
+ s^{39} \lambda^{40} \iint_{\omega_0}  \big(\phi_1 e^{-2s\vphi_m} \xi^{39}_{m}\big)_{x} \zeta_{xxx} u   + s^{39} \lambda^{40} \iint_{\omega_0}  \phi_1 e^{-2s\vphi_m} \xi^{39}_{m} \zeta_{xxxx} u \Big) .
\end{multline}
Let us recall the result \eqref{deri-esti}, so that we have the following:
\begin{align*}
\left|\big( \phi_1 e^{-2s\vphi_m} \xi^{39}_{m}   \big)_{n,x}    \right| \leq C s^n \lambda^n \, e^{-2s\vphi_m} \xi^{39+n}_m , \quad \text{for } n\in \mathbb N^*   ,      
\end{align*}
and  the estimate \eqref{estimate_I-4} follows 
\begin{multline}
|I_5|  \leq C \Big(  s^{43} \lambda^{44} \iint_{\omega_0}   e^{-2s\vphi_m} \xi^{43}_{m} \zeta u  +  s^{42} \lambda^{43} \iint_{\omega_0}   e^{-2s\vphi_m} \xi^{42}_{m} \zeta_x u   +  s^{41} \lambda^{42} \iint_{\omega_0}  e^{-2s\vphi_m} \xi^{39}_{m} \zeta_{xx} u  \\
+ s^{40} \lambda^{41} \iint_{\omega_0}   e^{-2s\vphi_m} \xi^{40}_{m} \zeta_{xxx} u   + s^{39} \lambda^{40} \iint_{\omega_0}  e^{-2s\vphi_m} \xi^{39}_{m} \zeta_{xxxx} u \Big) .
\end{multline}
Then, for any $\epsilon>0$, we obtain by using Cauchy-Schwarz inequality that 
\begin{multline}\label{estimate_I-4-eps}
|I_5|\leq \epsilon \bigg(  s^7\lambda^8 \iint e^{-2s\vphi_m} \xi^7_m |\zeta|^2 + s^5\lambda^6 \iint e^{-2s\vphi_m} \xi^5_m |\zeta_{x}|^2 +  s^3\lambda^4\iint e^{-2s\vphi_m} \xi^3_m |\zeta_{xx}|^2 \\
+ s\lambda^2 \iint e^{-2s\vphi_m} \xi_m |\zeta_{xxx}|^2   +	s^{-1} \iint e^{-2s\vphi_m}\xi^{-1}_m (|\zeta_t|^2+ |\zeta_{xxxx}|^2\big)     \bigg) \\ 
+ \frac{C}{\epsilon} s^{79}\lambda^{80} \iint_{\omega_0}  e^{-2s\vphi_m} \xi^{79}_m |u|^2 .
\end{multline}

\smallskip 
\noindent 
-- {\em Estimate for $I_6$.} In a similar manner, one can obtain an estimate for $I_6$, given by  
\begin{multline}\label{esti_I_5_eps}
	|I_6|  \leq \epsilon \bigg(  s^7\lambda^8 \iint e^{-2s\vphi_m} \xi^7_m |\zeta|^2 + s^5\lambda^6 \iint e^{-2s\vphi_m} \xi^5_m |\zeta_{x}|^2 +  s^3\lambda^4\iint e^{-2s\vphi_m} \xi^3_m |\zeta_{xx}|^2 \\
	+ \frac{C}{\epsilon}  s^{75}\lambda^{76}\iint_{\omega_0}  e^{-2s\vphi_m} \xi^{75}_m |u|^2 . 
	\end{multline}


\smallskip 
\noindent 
-- {\em Estimate for $I_7$.}
Let us  focus on the term $I_7$, we have 
\begin{multline}\label{estimate_I-6}
|I_7| \leq 	\left|-s^{39}\lambda^{40}\iint_{\omega_0}  (\phi_1 e^{-2s\vphi_m} \xi^{39}_m)_x \zeta w 
-	s^{39}\lambda^{40}\iint_{\omega_0}  \phi_1 e^{-2s\vphi_m} \xi^{39}_m \zeta_x w \right|\\
\leq \epsilon \left( s^7\lambda^8 \iint  e^{-2s\vphi_m} \xi^7_m |\zeta|^2 +  s^5\lambda^6 \iint  e^{-2s\vphi_m} \xi^5_m |\zeta_x|^2 \right) 
+ \frac{C}{\epsilon}  s^{73} \lambda^{74} \iint_{\omega_0}  e^{-2s\vphi_m} \xi^{73}_m |w|^2. 
\end{multline}

\smallskip 
\noindent 
-- {\em Estimate for $I_8$.} Finally, we see
\begin{multline}\label{esti_I_8}
	|I_8| = \left|  s^{39}\lambda^{40} \iint_{\omega_0} \phi_1 e^{-2s\vphi_m} \xi^{39}_m \zeta F_1 \right| \\
	\leq \epsilon s^7 \lambda^8 \iint e^{-2s\vphi_m}\xi^7_m |\zeta|^2 + \frac{C}{\epsilon} s^{71}\lambda^{72} \iint e^{-2s\vphi_m}\xi^{71}_m |F_1|^2 ,  
\end{multline}
for any  $\epsilon>0$.

\smallskip 
Thus, using the estimates \eqref{estimate_I-3}, \eqref{estimate_I-4-eps}, \eqref{esti_I_5_eps}, \eqref{estimate_I-6} and \eqref{esti_I_8} in \eqref{esti_zeta}, we have 
\begin{multline}\label{estimate-zeta-final}
s^{39}\lambda^{40}\iint_{\omega_1}  e^{-2s\vphi_m} \xi^{39}_m |\zeta|^2 \leq C \epsilon   \bigg(  s^7\lambda^8 \iint e^{-2s\vphi_m} \xi^7_m |\zeta|^2 + s^5\lambda^6 \iint e^{-2s\vphi_m} \xi^5_m |\zeta_{x}|^2 \\
+  s^3\lambda^4\iint e^{-2s\vphi_m} \xi^3_m |\zeta_{xx}|^2 
+ s\lambda^2 \iint e^{-2s\vphi_m} \xi_m |\zeta_{xxx}|^2   +	s^{-1} \iint e^{-2s\vphi_m}\xi^{-1}_m (|\zeta_t|^2+ |\zeta_{xxxx}|^2\big)     \bigg) \\ 
+ \frac{C}{\epsilon} s^{79}\lambda^{80} \iint_{\omega_0}  e^{-2s\vphi_m} \xi^{79}_m |u|^2 	+ \frac{C}{\epsilon}  s^{73} \lambda^{74} \iint_{\omega_0}  e^{-2s\vphi_m} \xi^{73}_m |w|^2 + \frac{C}{\epsilon} s^{71}\lambda^{72} \iint e^{-2s\vphi_m}\xi^{71}_m |F_1|^2. 
\end{multline}
Now, fix $\epsilon>0$ in \eqref{estimate-zeta-final} small enough so that all the integrals in $Q_T$ can be absorbed by the quanity $I_{KS}(\zeta)$ in \eqref{Add_theta_zeta-2}, and this yields the required Carleman inequality \eqref{Carleman-first-approach}. 
\end{proof}

\smallskip

\subsection{Carleman estimate for the case when $\boldsymbol{\alpha \in (0,1)}$}\label{carleman-alpha=0,1}

We just state the Carleman estimate for the case $\alpha\in (0,1)$. In this case, the proof is simpler than the previous two cases since we just need to use the 
 standard Carleman estimates of the fourth and second order parabolic equations.   

Recall the adjoint system \eqref{adj_sys} for any $\alpha \in (0,1)$. Then, one can obtain the following Carleman inequality. 
\begin{theorem}[Carleman inequality: the case $\alpha\in (0,1)$]\label{Thm-Carleman-first-app}
	Let the weight functions $(\vphi_m, \xi_m)$  be given by \eqref{weight_function}  with $m\geq 1$. Then, there exist positive constants $\widehat \lambda$, $\widehat s:= \widehat \sigma(T^m+ T^{2m-2/5}+ T^{2m-1}+ T^{2m})$ with some $\widehat \sigma>0$ and $C$ such that we have the following estimate satisfied by the solution to \eqref{adj_sys}:
	\begin{multline}\label{Carleman-app-1} 
		I_{KS}(u) + I_{H}(w;3) + I_{KS}(\zeta) + 
		I_{H}(\theta;3)     \\
		\leq C \iint e^{-2s\vphi_m} \left( s^7\lambda^8 \xi^7_m |F_1|^2 + s^3\lambda^4 \xi^3_m |F_2|^2 + |F_3|^2 + |F_4|^2   \right) \\
		+ C s^{15} \lambda^{16} \iint_{\omega_0} e^{-2s\vphi_m} \xi^{15}_m |u|^2 + C s^{9} \lambda^{10} \iint_{\omega_0} e^{-2s\vphi_m} \xi^9_m |w|^2,  
	\end{multline}
	for all $\lambda\geq \widehat  \lambda$ and  $s\geq \widehat s$, where $I_{KS}(\cdot)$ and $I_H(\cdot;\cdot)$ are given by \eqref{Notation_KS} and \eqref{Notation_heat} respectively. 
\end{theorem}

As stated before, the technique for proving \Cref{Thm-Carleman-first-app} is the following. 
\begin{itemize}
	\item[(i)] For the variables $u$ and $\zeta$, satisfying fourth order parabolic equations, we use the Carleman estimate given by \Cref{thm-Zhou}.

	\item[(ii)]  For $w$ and $\theta$, we use the classical Carleman inequalities for the heat equations, see \Cref{thm-Fur-Ima}. 
\end{itemize} 
 
\smallskip 

A sketch of the proof for \Cref{Thm-Carleman-first-app} is available in the arxiv version \cite[Section 2.1]{bhandari_Santamaria}.

\smallskip

\section{Null-controllability of the studied systems for different $\boldsymbol \alpha$}\label{Section-local-null}

In this section, we establish the local null-controllability of our extended system \eqref{sys_equiv_1}--\eqref{sys_equiv_2} for different values of $\alpha\in [0,1]$. But as mentioned earlier, the most interesting cases are $\alpha=0$ and $\alpha=1$. Therefore, we mainly discuss the controllability for  these two cases.

\smallskip 

As we have mentioned, the main ingredient to prove \Cref{thm:main_extended} is to first obtain a suitable observability inequality for system \eqref{adj_sys} which will ensure the null-controllability for the extended linearized systems \eqref{sys_linear-1}--\eqref{sys_linear-2}.  Due to the presence of the parameter $\alpha$ in the system, we have used different strategies leading to different Carleman estimates, see Theorems  \ref{Theorem-Carleman-second-app}, \ref{theorem-carleman} and \ref{Thm-Carleman-first-app}.

\smallskip 

In our present work, we mainly give a detailed proof of the observability inequality derived from \Cref{Theorem-Carleman-second-app}. The case $\alpha=1$ will be then shortly described in Section \ref{section-alpha=1}; more precisely, there we point out the main changes in the proof of observability inequality and the null-controllablity in comparison with the case $\alpha=0$.

\smallskip 

After the study of linear cases, we use the so-called {\em inverse mapping theorem} to handle the local null-controllability of the nonlinear systems \eqref{sys_equiv_1}--\eqref{sys_equiv_2}. 

\smallskip 

The required proofs for the more standard case $\alpha\in (0,1)$  can be done in a similar fashion by making minor adjustments. 

\smallskip 

We  proceed one by one.

\subsection{The case when $\boldsymbol{\alpha=0}$}

\subsubsection{Observability inequality ($\boldsymbol{\alpha=0}$)}
\label{sec-obser-0}


To do this, we shall first prove  a refined Carleman inequality with weight functions that do not vanish at $t=T$. More precisely, let us consider

\begin{equation}
\ell(t)=
\begin{cases}
t(T-t), & 0\leq t\leq  T/2, \\
T^2/4, &T/2\leq t\leq T,
\end{cases}
\end{equation}
and the following associated weight functions
\begin{align}\label{weight_function_not_0}
	\mathfrak{S}_m(t,x)= \frac{e^{\lambda(1+\frac{1}{m})k \|\nu\|_{\infty}   }-  e^{\lambda\big( k\|\nu\|_{\infty} + \nu(x) \big)}}{\ell(t)^m}	, \quad \mathfrak{Z}_{m}(t,x) = \frac{e^{\lambda\big( k\|\nu\|_{\infty} + \nu(x) \big)}}{\ell(t)^m}, \quad \forall (t,x) \in Q_T.
\end{align}
for any constants $\lambda>1$ and $k>m>0$. Additionally, we define 
\begin{align}\label{max_min_not_0}
		\widehat{\mathfrak S}_m(t) = \max_{x\in [0,1]} \mathfrak{S}_m(t,x), \quad \widehat{\mathfrak{Z}}_m(t) = \max_{x\in [0,1]} \mathfrak{Z}_m(t,x), \\
		\label{max_min_not_1}
			{\mathfrak S}^*_m(t) = \min_{x\in [0,1]} \mathfrak{S}_m(t,x), \quad \mathfrak{Z}^*_m(t) = \min_{x\in [0,1]} \mathfrak{Z}_m(t,x). 
\end{align}

We have the following.
\begin{proposition}[A refined Carleman estimate: the case $\alpha=0$]\label{prop:refined_a_0}
Let $m$, $k$,  $s$ and $\lambda$ be fixed constants according to \Cref{Theorem-Carleman-second-app}. Then, there exists a positive constant $C$ depending at most on $\omega, \mathcal O$, $T$, $m$, $k$, $s$, and $\lambda$ such that
\begin{multline}\label{refined_second-app} 
			\iint \left(e^{-2s \widehat{\mathfrak{S}}_m} (\mathfrak{Z}^*_m)^7 |u|^2 +  e^{-2s \widehat{\mathfrak{S}}_m} (\mathfrak{Z}^*_m)^7 |w|^2 + e^{-2s \widehat{\mathfrak{S}}_m} (\mathfrak{Z}^*_m)^7 |\zeta|^2+ e^{-2s \widehat{\mathfrak{S}}_m} (\mathfrak{Z}^*_m)^9 |\theta|^2 \right)   \\  
			  + \|\zeta(T)\|^2_{L^2(0,1)} + \| \theta(T)\|^2_{L^2(0,1)} 
				 \leq  
			C \bigg[\iint e^{-2s\mathfrak{S}^*_m} |F_1|^2 + \iint  e^{-10s\mathfrak{S}^*_m + 8s\widehat{\mathfrak{S}}_m} \widehat{\mathfrak{Z}}^{37}_m |F_2|^2 \\
			+ \iint e^{-2s\mathfrak{S}^*_m} \big(\widehat{\mathfrak{Z}}^5_m|F_3|^2 + \widehat{\mathfrak{Z}}^7_m |F_4|^2\big)  
		+	  \iint_{\omega_0} e^{-10s\mathfrak{S}^*_m + 8s\widehat{\mathfrak{S}}_m} \widehat{\mathfrak{Z}}^{39}_m |u|^2 +   \iint_{\omega_0}e^{-10s\mathfrak{S}^*_m + 8s\widehat{\mathfrak{S}}_m} \widehat{\mathfrak{Z}}^{41}_m|w|^2\bigg] ,
\end{multline}
 where $(u,w,\zeta,\theta)$ is the solution associated to \eqref{adj_sys_a_0}, for any given $(\zeta_0,\theta_0)\in [L^2(0,1)]^2$ and $F_j\in L^2(Q_T)$ for $j=1,2,3,4$.
\end{proposition}

\begin{proof}

	Let us first fix the Carleman parameters $s=s^*$ and $\lambda=\lambda^*$ in the estimate \eqref{Carleman-second-app} given by Theorem \ref{Theorem-Carleman-second-app}.

Now, we start by the  construction $\varphi_m=\mathfrak{S}_m$ and $\xi_m=\mathfrak{Z}_m$ in $(0,T/2)\times(0,1)$, hence
\begin{align*}
	&\int_0^{\frac{T}{2}}\int_0^1\left(e^{-2s \mathfrak{S}_m} \mathfrak{Z}^7_m |u|^2 +  e^{-2s \mathfrak{S}_m} \mathfrak{Z}^7_m |w|^2 + e^{-2s \widehat{\mathfrak{S}}_m} (\mathfrak{Z}^*)^7_m |\zeta|^2+ e^{-2s \mathfrak{S}_m} \mathfrak{Z}^9_m |\theta|^2 \right) \\
	& =\int_0^{\frac{T}{2}}\int_0^1\left(e^{-2s \vphi_m} \xi^7_m |u|^2 +  e^{-2s \vphi_m} \xi^7_m |w|^2 + e^{-2s \widehat{\vphi}_m} (\xi^*)^7_m |\zeta|^2+ e^{-2s \vphi_m} \xi^9_m |\theta|^2 \right) .
\end{align*}
Therefore, using the Carleman estimate  \eqref{Carleman-second-app}, 
 and the definitions \eqref{max_min_not_0}--\eqref{max_min_not_1},  we readily get
\begin{multline}\label{aux-car-1}
	\int_0^{\frac{T}{2}} \int_0^1 \left(e^{-2s \widehat{\mathfrak{S}}_m} (\mathfrak{Z}^*)^7_m |u|^2 +  e^{-2s \widehat{\mathfrak{S}}_m} (\mathfrak{Z}^*)^7_m |w|^2 + e^{-2s \widehat{\mathfrak{S}}_m} (\mathfrak{Z}^*)^7_m |\zeta|^2+ e^{-2s \widehat{\mathfrak{S}}_m} (\mathfrak{Z}^*)^9_m |\theta|^2 \right) \\
	\leq 	C \bigg[\iint e^{-2s\mathfrak{S}^*_m} |F_1|^2 + \iint  e^{-10s\mathfrak{S}^*_m + 8s\widehat{\mathfrak{S}}_m} \widehat{\mathfrak{Z}}^{37}_m |F_2|^2 + \iint e^{-2s\mathfrak{S}^*_m} \big(\widehat{\mathfrak{Z}}^5_m|F_3|^2 + \widehat{\mathfrak{Z}}^7_m |F_4|^2\big)  \\ 
	 +	  \iint_{\omega_0} e^{-10s\mathfrak{S}^*_m + 8s\widehat{\mathfrak{S}}_m} \widehat{\mathfrak{Z}}^{39}_m |u|^2 +   \iint_{\omega_0}e^{-10s\mathfrak{S}^*_m + 8s\widehat{\mathfrak{S}}_m} \widehat{\mathfrak{Z}}^{41}_m|w|^2\bigg].
	\end{multline}

\smallskip

For the domain $(T/2,T)\times(0,1)$, we argue as follows. Let us introduce a function $\eta\in \C^1([0,T])$ such that
\begin{equation}
\eta=0 \textnormal{ in } [0,T/4], \quad \eta=1 \textnormal{ in } [T/2,T], \quad |\eta^\prime|\leq C/T .
\end{equation}
Using \Cref{Prop-appendix-well-posed}, we apply the energy estimate to the equation verified by $(\eta \zeta, \eta \theta)$ with \\ $((\eta\zeta)(0), (\eta \theta)(0)) = (0,0)$ , so that one can deduce that
\begin{align}\notag 
&\|\eta \zeta\|^2_{L^\infty(T/4,T;L^2(0,1))} + \|\eta \theta\|^2_{L^\infty(T/4,T;L^2(0,1))} \\ 
\label{eq:est_eta_forward}
&\quad \leq C\Big(\|\eta F_3\|^2_{L^2((T/4, T)\times(0,1))} +  \|\eta F_4\|^2_{L^2((T/4, T)\times(0,1))} \\ \notag  
&  \qquad \qquad +\frac{1}{T^2}\| \zeta\|^2_{L^2((T/4,T/2)\times  (0,1))}+\frac{1}{T^2}\| \theta\|^2_{L^2((T/4,T/2)\times (0,1))}\Big) .
\end{align}
Analogously, we have for the equation verified by $(\eta u,\eta w)$ that
\begin{align}\notag 
&\|\eta u\|^2_{L^\infty(T/4,T;L^2(0,1))} + \|\eta w\|^2_{L^\infty(T/4,T;L^2(0,1))} \\ 
\label{eq:est_eta_backward}
&\quad \leq C \Big(\| \eta F_1\|^2_{L^2((T/4, T)\times (0,1))} + \|\eta F_2\|^2_{L^2((T/4, T)\times (0,1))} \\ \notag 
& \qquad \qquad +  \|\eta \theta\|^2_{L^2((T/4,T)\times\mathcal O)} +\frac{1}{T^2}\| u\|^2_{L^2((T/4,T/2)\times (0,1))}+ \frac{1}{T^2}\|w\|^2_{L^2((T/4,T/2)\times(0,1))}\Big)  .
\end{align}
Using the estimate of $\eta \theta$ from \eqref{eq:est_eta_forward}  in the right hand side of \eqref{eq:est_eta_backward} and then combining both
 \eqref{eq:est_eta_forward}--\eqref{eq:est_eta_backward} we obtain, 
\begin{align}\label{estimate-for-back}   
&\|\eta u\|^2_{L^\infty(T/4,T;L^2(0,1))} + \|\eta w\|^2_{L^\infty(T/4,T;L^2(0,1))}  + \|\eta \zeta\|^2_{L^\infty(T/4,T;L^2(0,1))} + \|\eta \theta\|^2_{L^\infty(T/4,T;L^2(0,1))} \\ \notag 
& \leq {C} \left(\sum_{j=1}^4\|\eta F_j \|^2_{L^2((T/4,T)\times (0,1))} + \|u\|^2_{L^2((T/4,T/2) \times (0,1))}+\| w\|^2_{L^2((T/4,T/2) \times(0,1))}\right) \\ \notag 
&\quad +{C}\Big(\|\zeta\|^2_{L^2((T/4,T/2) \times(0,1))}+\| \theta\|^2_{L^2((T/4,T/2) \times(0,1))}\Big) .
\end{align}
Since the weight functions $\mathfrak{S}_m$ and $\mathfrak{Z}_m$ are bounded (by below and above) in $[T/4,T]$, we have 
\begin{align}\label{aux-esti-1}
&	\int_{\frac{T}{4}}^{\frac{T}{2}} \int_0^1 \left(|u|^2 + |w|^2 + |\zeta|^2 +|\theta|^2 \right) \\ \notag 
& \leq 	\int_{\frac{T}{4}}^{\frac{T}{2}} \int_0^1 \left(e^{-2s \mathfrak{S}_m} \mathfrak{Z}^7_m |u|^2 +  e^{-2s \mathfrak{S}_m} \mathfrak{Z}^7_m |w|^2 + e^{-2s \widehat{\mathfrak{S}}_m} (\mathfrak{Z}^*)^7_m |\zeta|^2+ e^{-2s \mathfrak{S}_m} \mathfrak{Z}^9_m |\theta|^2 \right) \\ \notag 
& \leq  C \bigg[\iint e^{-2s\mathfrak{S}_m} |F_1|^2 + \iint  e^{-10s\mathfrak{S}_m + 8s\widehat{\mathfrak{S}}_m} {\mathfrak{Z}}^{37}_m |F_2|^2 + \iint e^{-2s\mathfrak{S}_m} \big(\mathfrak{Z}^5_m|F_3|^2 + \mathfrak{Z}^7_m |F_4|^2\big) \\ \notag 
& \quad +	  \iint_{\omega_0} e^{-10s\mathfrak{S}_m + 8s\widehat{\mathfrak{S}}_m} \mathfrak{Z}^{39}_m |u|^2 +   \iint_{\omega_0}e^{-10s\mathfrak{S}_m + 8s\widehat{\mathfrak{S}}_m} \mathfrak{Z}^{41}_m|w|^2\bigg] , 
\end{align}
thanks to the Carleman estimate \eqref{Carleman-second-app}.

We also can incorporate the weight functions in the left hand side of \eqref{estimate-for-back}, which yields together with \eqref{aux-esti-1}
 and from the definitions \eqref{max_min_not_0}--\eqref{max_min_not_1}, that 
\begin{multline}\label{Carleman-T-alpha=0}
	\int_{\frac{T}{2}}^T \int_0^1 \left(e^{-2s \widehat{\mathfrak{S}}_m} (\mathfrak{Z}^*)^7_m |u|^2 +  e^{-2s \widehat{\mathfrak{S}}_m} (\mathfrak{Z}^*)^7_m |w|^2 + e^{-2s \widehat{\mathfrak{S}}_m} (\mathfrak{Z}^*)^7_m |\zeta|^2+ e^{-2s \widehat{\mathfrak{S}}_m} (\mathfrak{Z}^*)^9_m |\theta|^2 \right)   \\
+ \|\eta u\|^2_{L^\infty(T/4,T;L^2(0,1))} + \|\eta w\|^2_{L^\infty(T/4,T;L^2(0,1))}  + \|\eta \zeta\|^2_{L^\infty(T/4,T;L^2(0,1))} + \|\eta \theta\|^2_{L^\infty(T/4,T;L^2(0,1))} \\ 
	 \leq  
	C \bigg[\iint e^{-2s\mathfrak{S}^*_m} |F_1|^2 + \iint  e^{-10s\mathfrak{S}^*_m + 8s\widehat{\mathfrak{S}}_m} \widehat{\mathfrak{Z}}^{37}_m |F_2|^2 + \iint e^{-2s\mathfrak{S}^*_m} \big(\widehat{\mathfrak{Z}}^5_m|F_3|^2 + \widehat{\mathfrak{Z}}^7_m |F_4|^2\big)  \\ 
	+	  \iint_{\omega_0} e^{-10s\mathfrak{S}^*_m + 8s\widehat{\mathfrak{S}}_m} \widehat{\mathfrak{Z}}^{39}_m |u|^2 +   \iint_{\omega_0}e^{-10s\mathfrak{S}^*_m + 8s\widehat{\mathfrak{S}}_m} \widehat{\mathfrak{Z}}^{41}_m|w|^2\bigg] .
\end{multline}
for some $C>0$.

Combining \eqref{aux-car-1} and \eqref{Carleman-T-alpha=0}, we have the desired result \eqref{refined_second-app}.  
\end{proof}

As a consequence of \Cref{prop:refined_a_0}, we have the following result.
\begin{proposition}[Observability inequality: the case $\alpha=0$]\label{Obser-ineq}
Let $m$, $k$, $s$ and $\lambda$ be fixed constants according to \Cref{Theorem-Carleman-second-app}. Then, there exists a positive constant $C$ depending at most on $\omega, \mathcal O$, $T$, $m$, $k$, $s$, and $\lambda$ such that for any given $(\zeta_0,\theta_0)\in [L^2(0,1)]^2$ and $F_j\in L^2(Q_T)$, $j=1,2,3,4$, the solution $(u,w,\zeta,\theta)$ to \eqref{adj_sys_a_0} satisfies
\begin{multline}\label{eq:obs_final}
\|e^{-s\widehat{\mathfrak{S}}_m} (\mathfrak{Z}^*)^{5/2-1/m} u \|^2_{L^2(Q_T)} + \|e^{-s\widehat{\mathfrak{S}}_m} (\mathfrak{Z}^*)^{5/2-1/m} w \|^2_{L^2(Q_T)} \\
+ 
\|e^{-s\widehat{\mathfrak{S}}_m} (\mathfrak{Z}^*)^{5/2-1/m} \zeta \|^2_{L^2(Q_T)} + \|e^{-s\widehat{\mathfrak{S}}_m} (\mathfrak{Z}^*)^{5/2-1/m} \theta \|^2_{L^2(Q_T)}\\
+ \int_0^1 \left(|\zeta(T,x)|^2 + |\theta(T,x)|^2 \right) 
\leq  C\bigg[\iint e^{-10s\mathfrak{S}^*_m+8s\widehat{\mathfrak{S}}_m} \widehat{\mathfrak{Z}}_m^{37} \left(|F_1|^2 + |F_2|^2+|F_3|^2 +|F_4|^2\right)  \\
+  \iint_{\omega_0} e^{-10s\mathfrak{S}^*_m + 8s\widehat{\mathfrak{S}}_m} \widehat{\mathfrak{Z}}^{39}_m |u|^2 +   \iint_{\omega_0}e^{-10s\mathfrak{S}^*_m + 8s\widehat{\mathfrak{S}}_m} \widehat{\mathfrak{Z}}^{41}_m|w|^2\bigg]  .
\end{multline}
\end{proposition}

\begin{proof}
Let us define $\rho^*(t)=e^{-s\widehat{\mathfrak{S}}_m} (\mathfrak{Z}^*)^{5/2-1/m}$ so that $\rho^*(0) =0$. 
Then, the equation of $(\zeta^*, \theta^*)=(\rho^*\zeta, \rho^*\theta)$ looks like 
\begin{align*}
	\begin{dcases}
		\zeta^*_t + \zeta^*_{xxxx} +\gamma \zeta^*_{xx} = \theta^*_x + \rho^* F_3 + \rho^*_t \zeta  & \textnormal{in } Q_T, \\
		\theta^*_t - \theta^*_{xx} + \beta \theta^*_x =  \zeta^*_x + \rho^* F_4 + \rho^*_t \theta & \textnormal{in } Q_T,\\
		\zeta^* = \zeta^*_x = \theta^* =0 & \textnormal{in } (0,T), \\
		\zeta^*(0) = \theta^*(0) =0 &\textnormal{in } (0,1),
	\end{dcases}
\end{align*}
and it satisfies the following estimate 
\begin{align}\label{esti-for-rhostar}
	\|\zeta^*\|^2_{L^\infty(0,T; L^2(0,1))} + \|\theta^*\|^2_{L^\infty(0,T; L^2(0,1))} 
	 \leq & C \Big(\|\rho^* F_3\|^2_{L^2(Q_T)} + \|\rho^* F_4\|^2_{L^2(Q_T)}  \\ \notag 
	 & + \|\rho^*_t \zeta\|^2_{L^2(Q_T)} +  \|\rho^*_t \theta\|^2_{L^2(Q_T)} \Big)
\end{align}

Similarly, the equation of $(u^*, w^*)=(\rho^* u, \rho^*w)$ is
\begin{align*}
	\begin{dcases}
		-u^*_t + u^*_{xxxx} +\gamma u^*_{xx} =  -w^*_x + \rho^*F_1 - \rho^*_t u    & \textnormal{in } Q_T, \\
	-w^*_t - w^*_{xx} - \beta w^*_x =   -u^*_x + \rho^*F_2 + \rho^*\theta \mathds{1}_{\mathcal O} - \rho^*_t w & \textnormal{in } Q_T,\\
	u^* = u^*_x = w =0 &\textnormal{in } (0,T), \\
	u^*(T) = w^*(T) =0 &\textnormal{in } (0,1),
	\end{dcases}
\end{align*} 
which satisfies 
\begin{multline}\label{esti-back-rhostar}
	\|u^*\|^2_{L^\infty(0,T; L^2(0,1))} + \|w^*\|^2_{L^\infty(0,T; L^2(0,1))} 
	\leq  C \Big(\|\rho^* F_1\|^2_{L^2(Q_T)} + \|\rho^* F_2\|^2_{L^2(Q_T)}  \\  
	+ \|\rho^*_t u\|^2_{L^2(Q_T)} +  \|\rho^*_t w\|^2_{L^2(Q_T)} +\|\rho^* \theta\|^2_{L^2(Q_T)} \Big).
\end{multline}
Now, it can be checked that $|(\widehat{\mathfrak{S}}_m)_t|\leq C (\mathfrak{Z}^*_m)^{1+1/m}$ and thus 
\begin{align*}
	|\rho^*_t| \leq C e^{-s \widehat{\mathfrak{S}}_m}  (\mathfrak{Z}^*_m)^{7/2} , \quad |\rho^*| \leq C e^{-s \mathfrak{S}^*_m} \widehat{\mathfrak{Z}}_m^{7/2}.
\end{align*}
Using this and together with \eqref{esti-for-rhostar}--\eqref{esti-back-rhostar}, we get
\begin{multline}\label{esti-back-for-rhostar}
	\|u^*\|^2_{L^\infty(0,T; L^2(0,1))} + \|w^*\|^2_{L^\infty(0,T; L^2(0,1))}  + 	\|\zeta^*\|^2_{L^\infty(0,T; L^2(0,1))} + \|\theta^*\|^2_{L^\infty(0,T; L^2(0,1))} \\
	\leq \iint \left(e^{-2s \widehat{\mathfrak{S}}_m} (\mathfrak{Z}^*_m)^7 |u|^2 +  e^{-2s \widehat{\mathfrak{S}}_m} (\mathfrak{Z}^*_m)^7 |w|^2 + e^{-2s \widehat{\mathfrak{S}}_m} (\mathfrak{Z}^*_m)^7 |\zeta|^2+ e^{-2s \widehat{\mathfrak{S}}_m} (\mathfrak{Z}^*_m)^7 |\theta|^2 \right)  \\
	+ \iint e^{-2s\mathfrak{S}^*_m} \widehat{\mathfrak{Z}}_m^7 \left(|F_1|^2 + |F_2|^2+|F_3|^2 +|F_4|^2\right)   .
\end{multline}
Then, by applying the modified Carleman estimate \eqref{refined_second-app} and using the fact that $-10s\mathfrak{S}^*_m+ 8s\widehat{\mathfrak{S}}_m\geq -2s\mathfrak{S}^*_m$ (thanks to the definitions of $\widehat{\mathfrak{S}}_m$ and $\mathfrak{S}^*_m$ given by \eqref{max_min_not_0}--\eqref{max_min_not_1}),  we get the required observability inequality  \eqref{eq:obs_final}.
\end{proof}

%
%


\subsubsection{Null-controllability of the linearized system  ($\boldsymbol{\alpha=0}$)}\label{sec-null-0}

We denote the following operators 
\begin{align}\label{op-L1}
\mathcal L_1 := \partial_t + \partial_{xxxx} + \gamma \partial_{xx} , \\
\label{op-L2}
\mathcal L_2 := \partial_t - \partial_{xx} + \beta \partial_{x},
\end{align}
and their respecting adjoint operators by
\begin{align}\label{op-L1star}
	\mathcal L^*_1 = -\partial_t + \partial_{xxxx} + \gamma \partial_{xx} , \\
	\label{op-L2star}
	\mathcal L^*_2 =- \partial_t - \partial_{xx} - \beta \partial_{x}. 
\end{align}
Also, we denote the following Banach space 
\begin{align}\label{space_E}
	\mathcal E := \Big\{ (y, z, p, q, h_1, h_2) \ | \ & e^{5s\mathfrak{S}^*_m - 4s \widehat{\mathfrak{S}}_m } \widehat{\mathfrak{Z}}^{-37/2}_m   (y, z, p, q) \in [L^2(Q_T)]^4 ,  \\ \notag 
	& e^{5s\mathfrak{S}^*_m - 4s \widehat{\mathfrak{S}}_m }\big( \widehat{\mathfrak{Z}}_m^{-39/2} h_1 \mathds{1}_\omega ,  \widehat{\mathfrak{Z}}_m^{-41/2} h_2 \mathds{1}_\omega\big) \in [L^2((0,T)\times \omega)]^2 , \\ \notag 
	& e^{5s\mathfrak{S}^*_m - 4s \widehat{\mathfrak{S}}_m } \widehat{\mathfrak{Z}}_m^{-41/2} (y, z, p, q) \in \C^0([0,T]; [L^2(0,1)]^4)  \\ \notag 
	& \qquad \cap L^2(0,T; H^2_0(0,1)\times H^{1}_0(0,1)\times H^{2}_0(0,1)\times H^{1}_0(0,1)) , \\ \notag 
	& e^{s\widehat{\mathfrak{S}}_m} (\mathfrak{Z}^*_m)^{-5/2+1/m}(\mathcal L_1 y -z_x - h_1 \mathds{1}_{\omega}) \in L^2(Q_T), \\ \notag 
      &  e^{s\widehat{\mathfrak{S}}_m} (\mathfrak{Z}^*_m)^{-5/2+1/m}(\mathcal L_2 z -y_x - h_2 \mathds{1}_{\omega}) \in L^2(Q_T), \\ \notag 
      & e^{s\widehat{\mathfrak{S}}_m} (\mathfrak{Z}^*_m)^{-5/2+1/m}(\mathcal L^*_1 p +q_x) \in L^2(Q_T), \\    \notag                          
         &  e^{s\widehat{\mathfrak{S}}_m} (\mathfrak{Z}^*_m)^{-5/2+1/m}(\mathcal L^*_2 q +p_x - z\mathds{1}_{\mathcal O} ) \in L^2(Q_T), \\ \notag 
          & \qquad p(T, \cdot)=q(T, \cdot)=0 \text{ in } (0,1) \Big\}. 
\end{align}

\begin{proposition}[Null-controllability: the case $\alpha=0$]\label{prop-null-cont-alpha=0}
	Let $m$, $k$,  $s$ and $\lambda$ be fixed constants according to \Cref{Theorem-Carleman-second-app}. Let $f_1, f_2, f_3, f_4$  be the functions satisfying 
	\begin{align}\label{condition-f_1-f_2} 
		e^{s\widehat{\mathfrak{S}}_m} (\mathfrak{Z}^*_m)^{-5/2+1/m} (f_1, f_2, f_3,f_4) \in [L^2(Q_T)]^4  .
	\end{align} 
Then, there exists controls $(h_1, h_2)$ and a solution $(y, z, p, q)$ to \eqref{sys_linear-1}--\eqref{sys_linear-2} (when $\alpha=0$) such that we have 
$p(0)=q(0)=0$ in $(0,1)$. 
	\end{proposition}
\begin{proof}
   We consider the following space 
   \begin{align*}
   	\mathcal Q_0 : = \Big\{(u,w, \zeta, \theta) \in \C^\infty(\overline{Q_T}) \ | \ u=u_x =w=\zeta=\zeta_x = \theta=0 \text{ on } \Sigma_T \Big\},
  \end{align*}	
and define the  bi-linear operator $\mathcal K: \mathcal Q_0 \times \mathcal Q_0 \to \mathbb R$ given by
\begin{align}\label{bi-linear-form}
	&\mathcal K ( (u, w, \zeta, \theta), (\underline u, \underline w, \underline \zeta, \underline \theta  )  )  \\ \notag 
	:&= \iint e^{-10s\mathfrak{S}^*_m+8s\widehat{\mathfrak{S}}_m} \widehat{\mathfrak{Z}}_m^{37} \Big[ (\mathcal L^*_1 u+w_x)(\mathcal L^*_1 \underline u + \underline{w}_x)    + 
	 (\mathcal L^*_2 w +u_x+\theta \mathds{1}_{\mathcal O})(\mathcal L^*_2 \underline w + \underline{u}_x + \underline{\theta} \mathds{1}_{\mathcal O} ) \\ \notag 
	 &\qquad  \qquad \qquad \ \ \ ~~~~~
	 + (\mathcal L_1 \zeta -\theta_x)(\mathcal L_1 \underline \zeta-\underline{\theta}_x ) +  (\mathcal L_2\theta -\zeta_x)(\mathcal L_2 \underline \theta-\underline{\zeta}_x )   \Big] \\ \notag 
	 & \quad + \int_0^T \int_{\omega} e^{-10 s \mathfrak{S}^*_m + 8s\widehat{\mathfrak{S}}_m} \big( \widehat{\mathfrak{Z}}_m^{39} u \underline{u} + \widehat{\mathfrak{Z}}_m^{41} w \underline{w} \big) ,
	\end{align}
as well as the  linear operator $l:\mathcal Q_0 \to \mathbb R$ given by
\begin{align*}
l((u,w,\zeta,\theta))
	: = \langle f_1, u\rangle_{L^2(Q_T)} + \langle f_2, w\rangle_{L^2(Q_T)}
	  + \langle f_3, \zeta\rangle_{L^2(Q_T)} + \langle f_4, \theta\rangle_{L^2(Q_T)}. 
\end{align*}

It is clear that the product \eqref{bi-linear-form} is defines an inner product since the observability inequality \eqref{eq:obs_final} holds. 
We denote by $\mathcal Q$, the closure of $\mathcal Q_0$ w.r.t. the norm $\mathcal K(\cdot,\cdot)^{1/2}$ and indeed it is an Hilbert space endowed with the inner product \eqref{bi-linear-form}. The linear functional $l$ is also bounded due to \eqref{eq:obs_final} and the hypothesis \eqref{condition-f_1-f_2}. Therefore, the Lax-Milgram's theorem ensures the existence of unique $(\widehat u, \widehat w, \widehat \zeta, \widehat \theta)\in \mathcal Q\times \mathcal Q$ satisfying 
\begin{align*}
	\mathcal K ( (\widehat u, \widehat w, \widehat \zeta, \widehat \theta), (\underline u, \underline w, \underline \zeta, \underline \theta  )  ) = l(  (\underline u, \underline w, \underline \zeta, \underline \theta  )  ) , \quad \forall   (\underline u, \underline w, \underline \zeta, \underline \theta  ) \in \mathcal Q. 
\end{align*}

Now, we set 
\begin{align}\label{solution-controlled-linear-for}
	&\widehat y = e^{-10s\mathfrak{S}^*_m+8s\widehat{\mathfrak{S}}_m} \widehat{\mathfrak{Z}}_m^{37}(\mathcal L^*_1 \widehat u + \widehat{w}_x ), \ \ \ 	\widehat z = e^{-10s\mathfrak{S}^*_m+8s\widehat{\mathfrak{S}}_m} \widehat{\mathfrak{Z}}_m^{37} (\mathcal L^*_2 \widehat w + \widehat{u}_x + \widehat \theta \mathds{1}_\mathcal O ) , \\
	\label{solution-controlled-linear-back}
	&	\widehat p = e^{-10s\mathfrak{S}^*_m+8s\widehat{\mathfrak{S}}_m} \widehat{\mathfrak{Z}}_m^{37} (\mathcal L_1 \widehat \zeta - \widehat{\theta}_x ), \ \ \ \	\widehat q = e^{-10s\mathfrak{S}^*_m+8s\widehat{\mathfrak{S}}_m} \widehat{\mathfrak{Z}}_m^{37} (\mathcal L_2 \widehat \theta - \widehat{\zeta}_x) ,
\end{align}
and 
\begin{align*}
	\widehat h_1 = e^{-10 s \mathfrak{S}^*_m + 8s\widehat{\mathfrak{S}}_m} \widehat{\mathfrak{Z}}^{39}_m \widehat u \mathds{1}_\omega , \quad  	\widehat h_2 = e^{-10 s \mathfrak{S}^*_m + 8s\widehat{\mathfrak{S}}_m} \widehat{\mathfrak{Z}}^{41}_m \widehat w \mathds{1}_\omega .
\end{align*}
Then, thanks to the observability inequality \eqref{eq:obs_final}, we have 
\begin{multline}\label{bound_solution_control}
	\iint e^{10s\mathfrak{S}^*_m - 8s\widehat{\mathfrak{S}}_m} \widehat{\mathfrak{Z}}_m^{-37} \left( |\widehat y|^2 + |\widehat z|^2 + |\widehat p|^2 +|\widehat q|^2 \right) \\
	+ \int_0^T \int_\omega  e^{10 s \mathfrak{S}^*_m - 8s\widehat{\mathfrak{S}}_m}\left( \widehat{\mathfrak{Z}}^{-39}_m |\widehat h_1|^2 + \widehat{\mathfrak{Z}}^{-41}_m |\widehat h_2|^2 \right) <+\infty  ,
\end{multline}
and this $(\widehat y, \widehat z, \widehat p, \widehat q)$ is unique solution to the linearized system \eqref{sys_linear-1}--\eqref{sys_linear-2} (with $\alpha=0$) in the sense of transposition with the control functions $\widehat h_1$ and $\widehat h_2$.  Moreover, from \eqref{solution-controlled-linear-back} it is clear that 
$$\widehat p(0)=0, \quad \widehat q(0)=0 \quad \text{in } (0,1).$$

\smallskip 

We further set 
\begin{align*}
	(y^*, z^*, p^*, q^*) = e^{5 s \mathfrak{S}^*_m - 4s\widehat{\mathfrak{S}}_m}\widehat{\mathfrak{Z}}^{-41/2}_m (\widehat y, \widehat z, \widehat p, \widehat q) ,
\end{align*} 
which satisfies 
\begin{align}\label{equation-star}
	\begin{dcases}
		\mathcal L_1 y^* - z^*_x =  
		 h^*_1 \mathds{1}_{\omega}  +  f^*_1 +  \big(e^{5 s \mathfrak{S}^*_m - 4s\widehat{\mathfrak{S}}_m}\widehat{\mathfrak{Z}}^{-41/2}_m\big)_t \widehat y  &\text{in } Q_T, 
		 \\
		\mathcal L_2 z^* - y^*_x =  
		 h^*_2 \mathds{1}_{\omega}  +  f^*_2 +  \big(e^{5 s \mathfrak{S}^*_m - 4s\widehat{\mathfrak{S}}_m}\widehat{\mathfrak{Z}}^{-41/2}_m\big)_t \widehat z  &\text{in } Q_T ,
		  \\
		\mathcal L^*_1 p^* + q^*_x = f^*_3   - \big(e^{5 s \mathfrak{S}^*_m - 4s\widehat{\mathfrak{S}}_m}\widehat{\mathfrak{Z}}^{-41/2}_m\big)_t \widehat p  &\text{in } Q_T , 
		\\
		\mathcal L^*_2 q^* + p^*_x = f^*_3 + z^*\mathds{1}_{\mathcal O} -   \big(e^{5 s \mathfrak{S}^*_m - 4s\widehat{\mathfrak{S}}_m}\widehat{\mathfrak{Z}}^{-41/2}_m\big)_t \widehat q   &\text{in } Q_T ,\\
		y^*=y^*_x = z^* =p^*=p^*_x=q^*=0  &\text{in } \Sigma_T, \\
		y^*(0)=z^*(0) = p^*(T)=q^*(T)=0 &\text{in } (0,1),
	\end{dcases}
\end{align}
where $h^*_i= e^{5 s \mathfrak{S}^*_m - 4s\widehat{\mathfrak{S}}_m}\widehat{\mathfrak{Z}}^{-41/2}_m \widehat h_i$ for $i=1,2$ and $f^*_j=e^{5 s \mathfrak{S}^*_m - 4s\widehat{\mathfrak{S}}_m}\widehat{\mathfrak{Z}}^{-41/2}_m f_j$ for $j=1,2,3,4$.

We first observe that 
\begin{align*}
e^{5 s \mathfrak{S}^*_m - 4s\widehat{\mathfrak{S}}_m} \widehat{\mathfrak{Z}}^{-41/2}_m \leq C T^{36m +1} e^{s\widehat{\mathfrak{S}}_m} (\mathfrak{Z}^*_m)^{-5/2+1/m}  
\end{align*}
since $e^{5 s \mathfrak{S}^*_m - 4s\widehat{\mathfrak{S}}_m}\leq e^{s\mathfrak{S}^*_m}$ and $\mathfrak{Z}_m^{-1}\leq CT^{2m}$.  

On the other hand,
 note that $|(\widehat{\mathfrak{S}}_m)_t|\leq C (\widehat{\mathfrak{Z}}_m)^{1+1/m}$ (also, $|({\mathfrak{S}^*}_m)_t|\leq C (\widehat{\mathfrak{Z}}_m)^{1+1/m}$), which yields 
\begin{align*} 
\big|\big(e^{5 s \mathfrak{S}^*_m - 4s\widehat{\mathfrak{S}}_m}\widehat{\mathfrak{Z}}^{-41/2}_m\big)_t\big|
\leq  C
e^{5 s \mathfrak{S}^*_m - 4s\widehat{\mathfrak{S}}_m} \widehat{\mathfrak{Z}}^{-39/2+1/m}_m \leq CT^{2m-1}e^{5 s \mathfrak{S}^*_m - 4s\widehat{\mathfrak{S}}_m} \widehat{\mathfrak{Z}}^{-37/2}_m   ,
\end{align*}


Using the above facts, 
  the conditions on $f_j$ given  by \eqref{condition-f_1-f_2} and the bound \eqref{bound_solution_control},   we can conclude that the right hand sides of the four pdes in \eqref{equation-star} belong to $L^2(Q_T)$. Therefore, 
 \begin{align*}
 	(y^*, z^*, p^*, q^*) \in \C^0([0,T]; [L^2(0,1)]^4) \cap L^2(0,T; H^2_0(0,1) \times H^1_0(0,1) \times H^2_0(0,1) \times H^1_0(0,1)) . 
 \end{align*}
Moreover, 
\begin{equation}
\begin{aligned}\label{bound-solution-star}
	\|y^*\|_{\C^0(L^2) \cap L^2(H^2_0)} + 	\|z^*\|_{\C^0(L^2) \cap L^2(H^1_0)} + 	\|p^*\|_{\C^0(L^2) \cap L^2(H^2_0)} + 	\|q^*\|_{\C^0(L^2) \cap L^2(H^1_0)}
	\leq C,
\end{aligned}
\end{equation}
 for some $C>0$, thanks to the assumption \eqref{condition-f_1-f_2} and the fact \eqref{bound_solution_control}. 
 
Thus, the functions $(\widehat y, \widehat z, \widehat p, \widehat q, \widehat h_1, \widehat h_2)\in \mathcal E$ defined in \eqref{space_E}. The proof is complete.
\end{proof}

\subsubsection{Local null-controllability of the nonlinear system  ($\boldsymbol{\alpha=0}$)}\label{section-locall-null-alpha=0}

In this section, we shall prove the main theorem of our paper for in the case when $\alpha=0$, that is \Cref{thm:main}. But as we mentioned in the beginning, this is equivalent to prove the local null-controllability of the extended system \eqref{sys_equiv_1}--\eqref{sys_equiv_2}, which is precisely \Cref{thm:main_extended}.

\smallskip 
To prove the local controllability result, we use  the following well-known theorem. 

\begin{theorem}[Inverse mapping theorem]\label{thm-inverse}
	Let $\B_1$, $\B_2$ be two Banach spaces and $\Y: \B_1 \to \B_2$ satisfying $\Y\in \C^1(\B_1; \B_2)$. Assume that $b_1\in \B_1$ and $\Y(b_1)=b_2\in \B_2$ and $\Y^\prime(b_1):\B_1\to \B_2$ is surjective. Then, there exists $\delta>0$ such that for every $\widetilde b\in \B_2$ satisfying $\|\widetilde b - b_2\|_{\B_2}<\delta$, there exists a solution of the equation
	$$\Y(b)=\widetilde b, \quad b\in \B_1. $$
\end{theorem}

\medskip

\begin{proof}[\bf Proof of \Cref{thm:main_extended}]  We now prove the main result of this work. The idea is to apply the above theorem with 
	\begin{align*}
		\B_1 = \mathcal E , 
	\end{align*}
\begin{align*}
	\B_2 = \mathcal F \times \mathcal F \times \mathcal F \times \mathcal F,
\end{align*}
	where 
	\begin{align}
		\mathcal F :=  \left\{ \xi  \mid e^{s \widehat{\mathfrak{S}}_m} (\mathfrak{Z}_m^*)^{-5/2+1/m} \xi \in L^2(Q_T)       \right\},
	\end{align} 
	and consider   $\Y:\B_1 \to \B_2$ such that
	\begin{align}
	&	\Y (y,z,p,q,h_1,h_2) \\  \notag 
	= &( \mathcal L_1 y - z_x - h_1 \mathds{1}_\omega + yy_x , \mathcal L_2  z  - y_x - h_2 \mathds{1}_\omega , \mathcal L_1^* p + q_x  -yp_x , \mathcal L_2^* q + p_x -z \mathds{1}_{\mathcal O}) .
	\end{align}
	
	\begin{itemize} 
\item 	Let us first check that $\Y \in \C^1(\B_1, \B_2)$. In fact, all the terms of $\Y$ are well-defined and linear except the terms $yy_x$ and $-yp_x$. Thus, it is enough to show that the map 
\begin{align*}
	\big((y,z,p,q, h_1,h_2), (\widetilde y, \widetilde z, \widetilde p, \widetilde q , \widetilde h_1, \widetilde h_2) \big) \longmapsto \big(y \widetilde y_x,  -y \widetilde p_x \big) 
\end{align*}
	from $\B_1\times \B_1 \to \mathcal F \times \mathcal F$ is continuous. Before that, we observe that
	\begin{align}\label{ineq-weight} 
	e^{s \widehat{\mathfrak{S}}_m} (\mathfrak{Z}_m^*)^{-5/2+1/m} \leq C e^{2(5s \mathfrak{S}^*_m - 4s \widehat{\mathfrak{S}}_m )}  \widehat{\mathfrak{Z}}^{-41}_m,
	\end{align}
which can be seen as follows: there exists some $d_0>0$ such that
	\begin{align*}
		2(5s \mathfrak{S}^*_m - 4s \widehat{\mathfrak{S}}_m) - s \widehat {\mathfrak{S}}_m
		& = \frac{s}{\ell(t)^m} \left[e^{\lambda(1+\frac{1}{m})k\|\nu\|_{\infty}} - 10 e^{\lambda(k\|\nu\|_{\infty}+\|\nu\|_{\infty})} + 9 e^{\lambda k\|\nu\|_{\infty}} \right]\\
		& =\frac{se^{\lambda k\|\nu\|_{\infty}}}{\ell(t)^m} \left[ e^{\lambda \frac{k}{m}\|\nu\|_{\infty}} -10 e^{\lambda \|\nu\|_{\infty}} + 9 \right]  \\
		&\geq  \frac{d_0 se^{\lambda k\|\nu\|_{\infty}}}{\ell(t)^m},
		\end{align*}
for chosen $k>m$ large enough. In other words,
\begin{align*}
	s\widehat{\mathfrak{S}}_m \leq 2(5s \mathfrak{S}^*_m - 4s \widehat{\mathfrak{S}}_m ) - \frac{d_0 se^{\lambda k\|\nu\|_{\infty}} }{\ell(t)^m} ,
\end{align*} 
which yields
\begin{align*}
	e^{s\widehat{\mathfrak{S}}_m} (\mathfrak{Z}^*_m)^{-5/2+1/m} \leq e^{2(5s \mathfrak{S}^*_m - 4s \widehat{\mathfrak{S}}_m )}  \widehat{\mathfrak{Z}}^{-41}_m \times e^{-\frac{s d_0 e^{\lambda k\|\nu\|_{\infty}} }{\ell(t)^m}} \widehat{\mathfrak{Z}}^{41}_m (\mathfrak{Z}^*_m)^{-5/2+1/m},
\end{align*}
and that \eqref{ineq-weight} follows.

	Now, using \eqref{ineq-weight}, we have 
	\begin{equation} \label{esti-continuity}
	\begin{aligned}
		&\left\|e^{s \widehat{\mathfrak{S}}_m} (\mathfrak{Z}_m^*)^{-5/2+1/m} \big(y \widetilde y_x,  -y \widetilde p_x \big) \right\|_{[L^2(Q_T)]^2}  \\
		\leq & C\left\|e^{2(5s \mathfrak{S}^*_m - 4s \widehat{\mathfrak{S}}_m )}  \widehat{\mathfrak{Z}}^{-41}_m y \widetilde y_x\right\|_{L^2(Q_T)} 
		+ 
	C	\left\|e^{2(5s \mathfrak{S}^*_m - 4s \widehat{\mathfrak{S}}_m )}  \widehat{\mathfrak{Z}}^{-41}_m y \widetilde p_x\right\|_{L^2(Q_T)}\\
		\leq & C\left\|e^{(5s \mathfrak{S}^*_m - 4s \widehat{\mathfrak{S}}_m )}  \widehat{\mathfrak{Z}}^{-41/2}_m y\right\|_{L^\infty(0,T; L^2(0,1))} \left\|e^{(5s \mathfrak{S}^*_m - 4s \widehat{\mathfrak{S}}_m )}  \widehat{\mathfrak{Z}}^{-41/2}_m \widetilde y\right\|_{L^2(0,T; H^2(0,1))} \\
		& \quad + C\left\|e^{(5s \mathfrak{S}^*_m - 4s \widehat{\mathfrak{S}}_m )}  \widehat{\mathfrak{Z}}^{-41/2}_m y\right\|_{L^\infty(0,T; L^2(0,1))} \left\|e^{(5s \mathfrak{S}^*_m - 4s \widehat{\mathfrak{S}}_m )}  \widehat{\mathfrak{Z}}^{-41/2}_m \widetilde p\right\|_{L^2(0,T; H^2(0,1))} \\
		\leq 
		& C\big\|(y,z,p,q,h_1,h_2)\big\|_{\B_1} \big\|(\widetilde y, \widetilde z, \widetilde p, \widetilde q, \widetilde h_1, \widetilde h_2)\big\|_{\B_1}   ,
	\end{aligned}
	\end{equation} 
Therefore, the claim $\Y\in \C^1(\B_1, \B_2)$ is established. 
\begin{remark} 
	The estimate \eqref{esti-continuity} 
 can be obtained as follows: 
	\begin{equation*}
		\begin{aligned}
			&\int_0^T \int_0^1 e^{4(5s \mathfrak{S}^*_m - 4s \widehat{\mathfrak{S}}_m )}  \widehat{\mathfrak{Z}}^{-41\times 2}_m |y \widetilde y_x|^2 \\
			& \leq \int_0^T  e^{4(5s \mathfrak{S}^*_m - 4s \widehat{\mathfrak{S}}_m )}  \widehat{\mathfrak{Z}}^{-41\times 2}_m \|y\|^2_{L^2(0,1)}\|\widetilde y_x \|^2_{L^\infty(0,1)} \\
			& \leq \int_0^T  e^{4(5s \mathfrak{S}^*_m - 4s \widehat{\mathfrak{S}}_m )}  \widehat{\mathfrak{Z}}^{-41\times 2}_m \|y\|^2_{L^2(0,1)}\|\widetilde y_x \|^2_{H^1_0(0,1)} \\ 
			&\leq \left\|e^{(5s \mathfrak{S}^*_m - 4s \widehat{\mathfrak{S}}_m )}  \widehat{\mathfrak{Z}}^{-41/2}_m y   \right\|^2_{L^\infty(0,T; L^2(0,1))}  \left\|e^{(5s \mathfrak{S}^*_m - 4s \widehat{\mathfrak{S}}_m )}  \widehat{\mathfrak{Z}}^{-41/2}_m \widetilde y \|  \right\|^2_{L^2(0,T; H^2_0(0,1))}.
			\end{aligned}
	\end{equation*}
A similar estimate can be found for the term associated to $y\widetilde p_x$.
\end{remark}

\medskip 

\item Next, we show that $\Y^\prime(0,0,0,0,0)$ is surjective. In fact, we have $\Y(0,0,0,0,0)=(0,0,0,0)$ and $\Y^\prime(0,0,0,0,0): \B_1 \to \B_2$  given by 
\begin{align*}
&\Y^\prime(0,0,0,0,0)(y,z,p,q,h_1,h_2) \\
=& ( \mathcal L_1 y - z_x - h_1 \mathds{1}_\omega, \mathcal L_2  z  - y_x - h_2 \mathds{1}_\omega , \mathcal L_1^* p + q_x, \mathcal L_2^* q + p_x -z \mathds{1}_{\mathcal O}) 
\end{align*}
is surjective, thanks to the null-controllability result \Cref{prop-null-cont-alpha=0}. 

Finally, we consider   $b_1=(0,0,0,0,0)$,  $b_2 =(0,0,0,0)$ and
$\widetilde b=(\xi_1, \xi_2,0,0)\in \B_2$, where $(\xi_1,\xi_2)$ is the given external source term in \eqref{system-main} or in \eqref{sys_equiv_1}--\eqref{sys_equiv_2}.  Then, according to \Cref{thm-inverse},   
 there exists a $\delta>0$ verifying 
\begin{align*}
	\left\|(\xi_1, \xi_2,0,0)\right\|_{\B_2} < \delta ,
\end{align*} 
we have the existence of solution-control pair  $(y,z,p,q,h_1,h_2)\in \B_1=\mathcal E$ to the system \eqref{sys_equiv_1}--\eqref{sys_equiv_2}. In particular, $p(0)=q(0)=0$ in $(0,1)$. This completes the proof of \Cref{thm:main_extended} which implies the proof for \Cref{thm:main}. 
	\end{itemize}
	\end{proof}

\begin{remark}\label{remark-nonlinear}
	Observe that, we have chosen $L^2(Q_T)$ right hand sides in the equations \eqref{sys_linear-1}--\eqref{sys_linear-2} and initial data $(y_0,z_0)\in [L^2(0,1)]^2$ to handle the controllability for the nonlinear system \eqref{sys_equiv_1}--\eqref{sys_equiv_2}. 
	In several previous works  regarding the controllability study of stabilized KS systems, for instance \cite{Cerpa-Mercado-Pazoto, Cerpa_Careno}, the authors considered $L^1(0,T; W^{-1,1}(0,1))$-- and $L^2(0,T;H^{-1}(0,1))$--source terms (for the KS  and heat equations resp.) and $H^{-2}(0,1)\times H^{-1}(0,1)$ initial data. In those cases, the only nonlinear term is $yy_x$ and this can be formally handled like 
	\begin{align*}
		yy_x \in L^1(0,T; W^{-1,1}(0,1)) \Longleftrightarrow |y|^2 \in L^1(0,T; L^1(0,1)),
	\end{align*}
	but the same technique cannot be applied to the term like $``yp_x"$ appearing in our $4\times 4$ coupled system. As a matter of fact, it is hard to obtain 
	a precise estimate like \eqref{esti-continuity} if we start with data weaker than $L^2$. 
\end{remark}

\subsection{The case when $\boldsymbol{\alpha=1}$}\label{section-alpha=1}

\subsubsection{Observability inequality ($\boldsymbol{\alpha=1}$)}
\label{sec-obs-1}

We recall the carleman estimate \eqref{Carleman-first-approach} in \Cref{theorem-carleman} for $\alpha=1$ and   prescribe a modified carleman estimate from  it.

\begin{proposition}[A refined Carleman estimate: the case $\alpha=1$]\label{prop:refined_a_1}
	Let $m$, $k$,  $s$ and $\lambda$ be fixed constants according to \Cref{theorem-carleman}. Then, there exists a positive constant $C$ depending at most on $\omega, \mathcal O$, $T$, $m$, $k$, $s$, and $\lambda$ such that
		\begin{multline}\label{refined_alpha-1} 
			\iint \left(e^{-2s \widehat{\mathfrak{S}}_m} (\mathfrak{Z}^*_m)^7 |u|^2 +  e^{-2s \widehat{\mathfrak{S}}_m} (\mathfrak{Z}^*_m)^3 |w|^2 + e^{-2s \widehat{\mathfrak{S}}_m} (\mathfrak{Z}^*_m)^7 |\zeta|^2+ e^{-2s \widehat{\mathfrak{S}}_m} (\mathfrak{Z}^*_m)^3 |\theta_x|^2 \right)   \\
	+	 \|\zeta(T)\|^2_{L^2(0,1)} + \| \theta(T)\|^2_{L^2(0,1)} 
			\leq  
			C \bigg[\iint e^{-2s\mathfrak{S}^*_m} \left(\widehat{\mathfrak{Z}}_m^{71}|F_1|^2 + |F_2|^2   + \widehat{\mathfrak{Z}}^3_m|F_3|^2\right) \\
		+\iint e^{-2s\mathfrak{S}^*_m}  \widehat{\mathfrak{Z}}^2_m \left(|F_4|^2 + |F_{4,x}|^2  \right) 
		+	  \iint_{\omega_0} e^{-2s\mathfrak{S}^*_m} 
		\widehat{\mathfrak{Z}}^{79}_m |u|^2 +   \iint_{\omega_0}e^{-2s\mathfrak{S}^*_m}  \widehat{\mathfrak{Z}}^{73}_m|w|^2\bigg] ,
		\end{multline}
	where $(u,w,\zeta,\theta)$ is the solution associated to \eqref{adj_sys_a_1}, for  given data $(\zeta_0,\theta_0)\in [L^2(0,1)]^2$ and  $F_j\in L^2(Q_T)$, $j=1,2,3,4$  such that 
	$\iint e^{-2s\mathfrak{S}^*_m} \big(\widehat{\mathfrak{Z}}_m^{71}|F_1|^2 + |F_2|^2 + \widehat{\mathfrak{Z}}_m^{3}|F_3|^2 + \widehat{\mathfrak{Z}}_m^{2}|F_4|^2 + \widehat{\mathfrak{Z}}_m^{2}|F_{4,x}|^2\big) < +\infty$.  
\end{proposition}

\begin{proof} 
	
	We apply similar technique as used to prove \Cref{prop:refined_a_0}. 
	
	In the beginning, we fix the Carleman parameters $s=s_0$ and $\lambda=\lambda_0$ in the estimate \eqref{Carleman-first-approach} given by  \Cref{theorem-carleman},
and recall that $\varphi_m=\mathfrak{S}_m$ and $\xi_m=\mathfrak{Z}_m$ in $(0,T/2)\times(0,1)$,  hence
	\begin{align*}
		&\int_0^{\frac{T}{2}}\int_0^1\left(e^{-2s \mathfrak{S}_m} \mathfrak{Z}^7_m |u|^2 +  e^{-2s \mathfrak{S}_m} \mathfrak{Z}^3_m |w|^2 + e^{-2s {\mathfrak{S}}_m} \mathfrak{Z}^7_m |\zeta|^2+ e^{-2s \mathfrak{S}_m} \mathfrak{Z}^3_m |\theta_x|^2 \right) \\
		& =\int_0^{\frac{T}{2}}\int_0^1\left(e^{-2s \vphi_m} \xi^7_m |u|^2 +  e^{-2s \vphi_m} \xi^3_m |w|^2 + e^{-2s {\vphi}_m} \xi^7_m |\zeta|^2+ e^{-2s \vphi_m} \xi^3_m |\theta_x|^2 \right) .
	\end{align*}
	Therefore, using the Carleman estimate  \eqref{Carleman-first-approach}, 
	and the definitions \eqref{max_min_not_0}--\eqref{max_min_not_1},  we readily get
\begin{multline}\label{aux-car-alpha=1}
		\int_0^{\frac{T}{2}} \int_0^1 \left(e^{-2s \widehat{\mathfrak{S}}_m} (\mathfrak{Z}^*_m)^7 |u|^2 +  e^{-2s \widehat{\mathfrak{S}}_m} (\mathfrak{Z}^*_m)^3 |w|^2 + e^{-2s \widehat{\mathfrak{S}}_m} (\mathfrak{Z}^*_m)^7 |\zeta|^2+ e^{-2s \widehat{\mathfrak{S}}_m} (\mathfrak{Z}^*_m)^3 |\theta_x|^2 \right) \\
		\leq  
		C \iint e^{-2s\mathfrak{S}^*_m} \left(\widehat{\mathfrak{Z}}_m^{71}|F_1|^2 + |F_2|^2   + \widehat{\mathfrak{Z}}^3_m|F_3|^2 + \widehat{\mathfrak{Z}}^2_m |F_4|^2 + \widehat{\mathfrak{Z}}^2_m |F_{4,x}|^2  \right) 
	\\
		+	  \iint_{\omega_0} e^{-2s\mathfrak{S}^*_m} 
		 \widehat{\mathfrak{Z}}^{79}_m |u|^2 +   \iint_{\omega_0}e^{-2s\mathfrak{S}^*_m}  \widehat{\mathfrak{Z}}^{73}_m|w|^2\bigg].	
\end{multline}
%
%

We now recall the cut-off function $\eta$ given by 
	\begin{equation*}
		\eta=0 \textnormal{ in } [0,T/4], \quad \eta=1 \textnormal{ in } [T/2,T], \quad |\eta^\prime|\leq C/T .
	\end{equation*}
Then, as we obtained the energy estimate \eqref{estimate-for-back} (for the case $\alpha=0$), we have the following estimate for the equation satisfied by $(\eta u, \eta w, \eta \zeta, \eta \theta)$ with $((\eta\zeta)(0), (\eta \theta)(0))=(0,0)$, 
 %
%
%
%
%
		\begin{equation} 
		\begin{aligned}\label{estimate-for-back-alpha=1}   
			&\|\eta u\|^2_{L^\infty(T/4,T;L^2(0,1))} + \|\eta w\|^2_{L^\infty(T/4,T;L^2(0,1))}  + \|\eta \zeta\|^2_{L^\infty(T/4,T;L^2(0,1))}  \\
	& \quad  + \|\eta \theta\|^2_{L^\infty(T/4,T;L^2(0,1))} + \|\eta \theta\|^2_{L^2(T/4,T; H^1_0(0,1))}  \\ 
			 \leq & C \bigg(\sum_{j=1}^4\|\eta F_j \|^2_{L^2((T/4,T)\times (0,1))} + \|u\|^2_{L^2((T/4,T/2) \times (0,1))}+\| w\|^2_{L^2((T/4,T/2) \times(0,1))}  \\ 
			&\qquad  \quad +\|\zeta\|^2_{L^2((T/4,T/2) \times(0,1))}+\| \theta\|^2_{L^2((T/4,T/2) \times(0,1))}\bigg)  \\ 
			 \leq & C \bigg(\sum_{j=1}^4\|\eta F_j \|^2_{L^2((T/4,T)\times (0,1))} + \|u\|^2_{L^2((T/4,T/2) \times (0,1))}+\| w\|^2_{L^2((T/4,T/2) \times(0,1))}  \\ 
				&\qquad  \quad +\|\zeta\|^2_{L^2((T/4,T/2) \times(0,1))}+\| \theta_x\|^2_{L^2((T/4,T/2) \times(0,1))}\bigg), 
		\end{aligned}
	\end{equation}
using the Poincar\'e inequality for $\theta$.  
Now, 	since the weight functions $\mathfrak{S}_m$ and $\mathfrak{Z}_m$ are bounded (by below and above) in $[T/4,T]$, see \eqref{weight_function_not_0},  we have 
		\begin{align}\label{aux-esti-alpha=1}
			&	\int_{\frac{T}{4}}^{\frac{T}{2}} \int_0^1 \left(|u|^2 + |w|^2 + |\zeta|^2 +|\theta_x|^2 \right) \\ \notag 
			& \leq 	\int_{\frac{T}{4}}^{\frac{T}{2}} \int_0^1 \left(e^{-2s \mathfrak{S}_m} \mathfrak{Z}^7_m |u|^2 +  e^{-2s \mathfrak{S}_m} \mathfrak{Z}^3_m |w|^2 + e^{-2s {\mathfrak{S}}_m} \mathfrak{Z}^7_m |\zeta|^2+ e^{-2s \mathfrak{S}_m} \mathfrak{Z}^3_m |\theta_x|^2 \right) \\ \notag 
			& \leq  	C \iint e^{-2s\mathfrak{S}^*_m} \left(\widehat{\mathfrak{Z}}_m^{71}|F_1|^2 + |F_2|^2   + \widehat{\mathfrak{Z}}^3_m|F_3|^2 + \widehat{\mathfrak{Z}}^2_m |F_4|^2 + \widehat{\mathfrak{Z}}^2_m |F_{4,x}|^2  \right) 
			\\ \notag
		&	+	  \iint_{\omega_0} e^{-2s\mathfrak{S}^*_m} 
			\widehat{\mathfrak{Z}}^{79}_m |u|^2 +   \iint_{\omega_0}e^{-2s\mathfrak{S}^*_m}  \widehat{\mathfrak{Z}}^{73}_m|w|^2\bigg].	
		\end{align}
		thanks to the Carleman estimate \eqref{Carleman-first-approach} and the definitions \eqref{max_min_not_0}--\eqref{max_min_not_1}.

	We also can incorporate the weight functions in the left hand side of \eqref{estimate-for-back-alpha=1}, which yields together with \eqref{aux-esti-alpha=1}
	and from the definitions \eqref{max_min_not_0}--\eqref{max_min_not_1}, that 
		\begin{multline}\label{Carleman-T-alpha=1}
			\int_{\frac{T}{2}}^T \int_0^1 \left(e^{-2s \widehat{\mathfrak{S}}_m} (\mathfrak{Z}^*_m)^7 |u|^2 +  e^{-2s \widehat{\mathfrak{S}}_m} (\mathfrak{Z}^*_m)^3 |w|^2 + e^{-2s \widehat{\mathfrak{S}}_m} (\mathfrak{Z}^*_m)^7 |\zeta|^2+ e^{-2s \widehat{\mathfrak{S}}_m} (\mathfrak{Z}^*_m)^3|\theta_x|^2 \right)   \\
			+ \|\eta u\|^2_{L^\infty(T/4,T;L^2(0,1))} + \|\eta w\|^2_{L^\infty(T/4,T;L^2(0,1))}  + \|\eta \zeta\|^2_{L^\infty(T/4,T;L^2(0,1))} + \|\eta \theta\|^2_{L^\infty(T/4,T;L^2(0,1))} \\ 
			\leq  
			C \bigg[\iint e^{-2s\mathfrak{S}^*_m} \left(\widehat{\mathfrak{Z}}_m^{71}|F_1|^2 + |F_2|^2   + \widehat{\mathfrak{Z}}^3_m|F_3|^2 + \widehat{\mathfrak{Z}}^2_m |F_4|^2 + \widehat{\mathfrak{Z}}^2_m |F_{4,x}|^2  \right) 
		\\
		+	  \iint_{\omega_0} e^{-2s\mathfrak{S}^*_m} 
		\widehat{\mathfrak{Z}}^{79}_m |u|^2 +   \iint_{\omega_0}e^{-2s\mathfrak{S}^*_m}  \widehat{\mathfrak{Z}}^{73}_m|w|^2\bigg].	
		\end{multline}
		for some $C>0$.
	
	Combining \eqref{aux-car-alpha=1} and \eqref{Carleman-T-alpha=1}, we have the desired result \eqref{refined_alpha-1}.  
\end{proof} 

\smallskip 
Let us find the observability inequality in this case.

\begin{proposition}[Observability inequality: the case $\alpha=1$]\label{Obser-ineq-alpha=1}
	Let $m$, $k$, $s$ and $\lambda$ be fixed constants according to \Cref{theorem-carleman}. Then, there exists a positive constant $C$ depending at most on $\omega, \mathcal O$, $T$, $m$, $k$, $s$, and $\lambda$ such that for any $(\zeta_0,\theta_0)\in [L^2(0,1)]^2$, the solution $(u,w,\zeta,\theta)$ to \eqref{adj_sys_a_1} satisfies
		\begin{multline}\label{eq:obs_final-alpha=1}
			\|e^{-s\widehat{\mathfrak{S}}_m}  u \|^2_{L^2(Q_T)} + \|e^{-s\widehat{\mathfrak{S}}_m} w \|^2_{L^2(Q_T)} 
			+ 
			\|e^{-s\widehat{\mathfrak{S}}_m}  \zeta \|^2_{L^2(Q_T)} \\ 
			+ \|e^{-s\widehat{\mathfrak{S}}_m} \theta \|^2_{L^2(Q_T)} 
			+ \int_0^1 \left(|\zeta(T,x)|^2 + |\theta(T,x)|^2  \right)
			\\
			\leq  
			C \bigg[\iint e^{-2s\mathfrak{S}^*_m} \widehat{\mathfrak{Z}}_m^{71}\left(|F_1|^2 + |F_2|^2   + |F_3|^2 +  |F_4|^2 + |F_{4,x}|^2  \right) 
		\\
		+	  \iint_{\omega_0} e^{-2s\mathfrak{S}^*_m} 
		\widehat{\mathfrak{Z}}^{79}_m |u|^2 +   \iint_{\omega_0}e^{-2s\mathfrak{S}^*_m}  \widehat{\mathfrak{Z}}^{73}_m|w|^2\bigg] ,
		\end{multline}
	where the source terms   $F_j\in L^2(Q_T)$, $j=1,2,3,4$ in \eqref{adj_sys_a_1} are verifying
	$ \iint e^{-2s\mathfrak{S}^*_m} \widehat{\mathfrak{Z}}_m^{71}\Big(|F_1|^2 + |F_2|^2   + |F_3|^2 + |F_4|^2 +  |F_{4,x}|^2 \Big) <+\infty$.   
\end{proposition}

\begin{proof}
	Let us define $\rho^*(t)=e^{-s\widehat{\mathfrak{S}}_m(t)}$ so that $\rho^*(0)=0$. 
	Then, the equation of $(\zeta^*, \theta^*)=(\rho^*\zeta, \rho^*\theta)$ satisfies (recall first the adjoint system \eqref{adj_sys_a_1})
	\begin{align*}
		\begin{dcases}
			\zeta^*_t + \zeta^*_{xxxx} +\gamma \zeta^*_{xx} = \theta^*_x + \rho^* F_3 + \rho^*_t \zeta  & \textnormal{in } Q_T, \\
			\theta^*_t - \theta^*_{xx} + \beta \theta^*_x =  \zeta^*_x + \rho^* F_4 + \rho^*_t \theta & \textnormal{in } Q_T,\\
			\zeta^* = \zeta^*_x = \theta^* =0 & \textnormal{in } (0,T), \\
			\zeta^*(0) = \theta^*(0) =0 &\textnormal{in } (0,1),
		\end{dcases}
	\end{align*}
	and it satisfies the following estimate 
	\begin{align}\label{esti-for-rhostar-alpha=1}
		\|\zeta^*\|^2_{L^\infty(0,T; L^2(0,1))} + \|\theta^*\|^2_{L^\infty(0,T; L^2(0,1))} 
		\leq & C \Big(\|\rho^* F_3\|^2_{L^2((0,T)\times (0,1))} + \|\rho^* F_4\|^2_{L^2((0,T)\times (0,1))}  \\ \notag 
		& + \|\rho^*_t \zeta\|^2_{L^2((0,T)\times (0,1))} +  \|\rho^*_t \theta\|^2_{L^2((0,T)\times (0,1))} \Big)
	\end{align}
	
	Similarly, the equation of $(u^*, w^*)=(\rho^* u, \rho^*w)$ is
	\begin{align*}
		\begin{dcases}
			-u^*_t + u^*_{xxxx} +\gamma u^*_{xx} =  -w^*_x + \zeta^*\mathds{1}_{\mathcal O} + \rho^*F_1 - \rho^*_t u    & \textnormal{in } Q_T, \\
			-w^*_t - w^*_{xx} - \beta w^*_x =   -u^*_x + \rho^*F_2  - \rho^*_t w & \textnormal{in } Q_T,\\
			u^* = u^*_x = w =0 &\textnormal{in } (0,T), \\
			u^*(T) = w^*(T) =0 &\textnormal{in } (0,1),
		\end{dcases}
	\end{align*} 
	which satisfies 
	\begin{multline}\label{esti-back-rhostar-alpha=1}
		\|u^*\|^2_{L^\infty(0,T; L^2(0,1))} + \|w^*\|^2_{L^\infty(0,T; L^2(0,1))} 
		\leq  C \Big(\|\rho^* F_1\|^2_{L^2((0,T)\times (0,1))} + \|\rho^* F_2\|^2_{L^2((0,T)\times (0,1))}  \\  
		+ \|\rho^*_t u\|^2_{L^2((0,T)\times (0,1))} +  \|\rho^*_t w\|^2_{L^2((0,T)\times (0,1))} +\|\rho^* \zeta\|^2_{L^2((0,T)\times (0,1))} \Big).
	\end{multline}

	Now, we check that $|(\widehat{\mathfrak{S}}_m)_t|\leq C (\mathfrak{Z}^*_m)^{1+1/m}$ so that $|\rho^*_t| \leq C e^{-s \widehat{\mathfrak{S}}_m}  (\mathfrak{Z}^*_m)^{3/2}$ (since $m>3$ in this case).
%
%
Using this, together with \eqref{esti-for-rhostar-alpha=1}--\eqref{esti-back-rhostar-alpha=1} (also by applying the Poincar\'e inequality on $\theta$ in the r.h.s. of \eqref{esti-for-rhostar-alpha=1}), we get
	\begin{multline}\label{esti-back-for-rhostar-alpha=1}
		\|u^*\|^2_{L^\infty(0,T; L^2(0,1))} + \|w^*\|^2_{L^\infty(0,T; L^2(0,1))}  + 	\|\zeta^*\|^2_{L^\infty(0,T; L^2(0,1))} + \|\theta^*\|^2_{L^\infty(0,T; L^2(0,1))} \\
		\leq \iint \left(e^{-2s \widehat{\mathfrak{S}}_m} (\mathfrak{Z}^*_m)^3 |u|^2 +  e^{-2s \widehat{\mathfrak{S}}_m} (\mathfrak{Z}^*_m)^3 |w|^2 + e^{-2s \widehat{\mathfrak{S}}_m} (\mathfrak{Z}^*_m)^3 |\zeta|^2+ e^{-2s \widehat{\mathfrak{S}}_m} (\mathfrak{Z}^*_m)^3 |\theta_x|^2 \right)  \\
		+ \iint e^{-2s\mathfrak{S}^*_m}  \left(|F_1|^2 + |F_2|^2+|F_3|^2 +|F_4|^2+|F_{4,x}|^2\right)   .
	\end{multline}
	Then, by applying the modified Carleman estimate \eqref{Carleman-T-alpha=1}, we get the required observability inequality  \eqref{eq:obs_final-alpha=1}.
\end{proof}

\subsubsection{Null-controllability of the linearized system ($\boldsymbol{\alpha=1}$)}\label{sec-null-control-1}
We recall the operators $\mathcal L_1$, $\mathcal L_2$ from \eqref{op-L1}--\eqref{op-L2} and their adjoints  $\mathcal L^*_1$, $\mathcal L_2^*$ from \eqref{op-L1star}--\eqref{op-L2star}. 

 Denote the following Banach space 
\begin{align}\label{space_E-alpha=1}
	{\mathcal E}_1 := \Big\{ (y, z, p, q, h_1, h_2) \ | \ & e^{s\mathfrak{S}^*_m}  \widehat{\mathfrak{Z}}^{-71/2}_m (y, z, p, q) \in [L^2(Q_T)]^4 ,  \\ \notag 
	& e^{s\mathfrak{S}^*_m}\Big( \widehat{\mathfrak{Z}}_m^{-79/2} h_1 \mathds{1}_\omega ,  \widehat{\mathfrak{Z}}_m^{-73/2} h_2 \mathds{1}_\omega\Big) \in [L^2((0,T)\times \omega)]^2 , \\ \notag 
	& e^{s\mathfrak{S}^*_m} \widehat{\mathfrak{Z}}_m^{-79/2} (y, z, p, q) \in \C^0([0,T]; [L^2(0,1)]^4)  \\ \notag 
	& \qquad \cap L^2(0,T; H^2_0(0,1)\times H^{1}_0(0,1)\times H^{2}_0(0,1)\times H^{1}_0(0,1)) , \\ \notag 
	& e^{s\widehat{\mathfrak{S}}_m} \left(\mathcal L_1 y -z_x - h_1 \mathds{1}_{\omega}\right) \in L^2(Q_T), \\ \notag 
	&  e^{s\widehat{\mathfrak{S}}_m} 
	\left(\mathcal L_2 z -y_x - h_2 \mathds{1}_{\omega} \right) \in L^2(Q_T), \\ \notag 
	& e^{s\widehat{\mathfrak{S}}_m} 
	\left(\mathcal L^*_1 p +q_x - y \mathds{1}_{\mathcal O}\right) \in L^2(Q_T), \\    \notag                          
	&  e^{s\widehat{\mathfrak{S}}_m} 
	\left(\mathcal L^*_2 q +p_x\right) \in L^2(Q_T), \\ \notag 
	& \qquad p(T, \cdot)=q(T, \cdot)=0 \text{ in } (0,1) \Big\}. 
\end{align}

\begin{proposition}[Null-controllability: the case $\alpha=1$]\label{prop-null-cont-alpha=1}
	Let $m$, $k$,  $s$ and $\lambda$ be fixed constants according to \Cref{theorem-carleman}. Let $f_1, f_2, f_3, f_4$  be the functions satisfying 
	\begin{align}\label{condition-on-f} 
		e^{s\widehat{\mathfrak{S}}_m}  (f_1, f_2, f_3,f_4) \in [L^2(Q_T)]^4  .
	\end{align} 
	Then, there exists controls $(h_1, h_2)$ and a solution $(y, z, p, q)$ to \eqref{sys_linear-1}--\eqref{sys_linear-2} (when $\alpha=1$) such that we have 
	$p(0)=q(0)=0$ in $(0,1)$. 
\end{proposition}
\begin{proof}
	We consider the following space 
	\begin{align*}
		{\mathcal Q}_1 : = \Big\{(u,w, \zeta, \theta) \in \C^\infty(\overline{Q_T}) \ | \ u=u_x =w=\zeta=\zeta_x = \theta=0 \text{ on } \Sigma_T \Big\},
	\end{align*}	
	and define the  bi-linear operator ${\mathcal K}_1: 	{\mathcal Q}_1 \times 	{\mathcal Q}_1 \to \mathbb R$ given by
	\begin{align}\label{bi-linear-form-2}
		&{\mathcal K}_1 ( (u, w, \zeta, \theta), (\underline u, \underline w, \underline \zeta, \underline \theta  )  )  \\ \notag 
		:&= \iint e^{-2s\mathfrak{S}^*_m}\widehat{\mathfrak{Z}}^{71}_m \Big[ (\mathcal L^*_1 u+w_x-\zeta \mathds{1}_{\mathcal O})(\mathcal L^*_1 \underline u + \underline{w}_x-\zeta \mathds{1}_{\mathcal O})    + 
		(\mathcal L^*_2 w +u_x)(\mathcal L^*_2 \underline w + \underline{u}_x) \\ \notag 
		&\qquad  \qquad \qquad
		+ (\mathcal L_1 \zeta -\theta_x)(\mathcal L_1 \underline \zeta-\underline{\theta}_x ) +  (\mathcal L_2\theta -\zeta_x)(\mathcal L_2 \underline \theta-\underline{\zeta}_x ) +\big(\mathcal L_2\theta -\zeta_{x}\big)_x\big(\mathcal L_2 \underline \theta-\underline{\zeta}_{x} \big)_x   \Big] \\ \notag 
		& \quad + \int_0^T \int_{\omega} e^{-2s \mathfrak{S}^*_m } \big( \widehat{\mathfrak{Z}}_m^{79} u \underline{u} + \widehat{\mathfrak{Z}}_m^{73} w \underline{w} \big) ,
	\end{align}
	as well as the  linear operator $l:\mathcal Q_1 \to \mathbb R$ given by
	\begin{align*}
		l_1((u,w,\zeta,\theta))
		: = \langle f_1, u\rangle_{L^2(Q_T)} + \langle f_2, w\rangle_{L^2(Q_T)}
		+ \langle f_3, \zeta\rangle_{L^2(Q_T)} + \langle f_4, \theta\rangle_{L^2(Q_T)}. 
	\end{align*}
	
	It is clear that the product \eqref{bi-linear-form-2} is defines an inner product since the observability inequality \eqref{eq:obs_final-alpha=1} holds. 
	We denote by $\mathcal Q_\flat$, the closure of $\mathcal Q_1$ w.r.t. the norm $\mathcal K_1(\cdot,\cdot)^{1/2}$ and indeed it is an Hilbert space endowed with the inner product \eqref{bi-linear-form-2}. The linear functional $l_1$ is also bounded due to \eqref{eq:obs_final-alpha=1} and the hypothesis \eqref{condition-on-f},  therefore the Lax-Milgram's theorem ensures the existence of unique $(\widehat u, \widehat w, \widehat \zeta, \widehat \theta)\in \mathcal Q_\flat\times \mathcal Q_\flat$ satisfying 
	\begin{align}\label{equa-bilinear}
		\mathcal K_1 ( (\widehat u, \widehat w, \widehat \zeta, \widehat \theta), (\underline u, \underline w, \underline \zeta, \underline \theta  )  ) = l_1(  (\underline u, \underline w, \underline \zeta, \underline \theta  )  ) , \quad \forall   (\underline u, \underline w, \underline \zeta, \underline \theta  ) \in \mathcal Q_\flat. 
	\end{align}
	
	Now, we set 
	\begin{align}\label{solution-controlled-linear-for-1}
		&\widehat y = e^{-2s\mathfrak{S}^*_m}  \widehat{\mathfrak{Z}}^{71}_m
		 (\mathcal L^*_1 \widehat u + \widehat{w}_x -\widehat \zeta \mathds{1}_{\mathcal O} ), \ \ \ 	\widehat z = e^{-2s\mathfrak{S}^*_m}   \widehat{\mathfrak{Z}}^{71}_m (\mathcal L^*_2 \widehat w + \widehat{u}_x  ) , \\
		\label{solution-controlled-linear-back-1}
		&	\widehat p = e^{-2s\mathfrak{S}^*_m} \widehat{\mathfrak{Z}}^{71}_m  (\mathcal L_1 \widehat \zeta - \widehat{\theta}_x ), \ \ \ \	\widehat q = e^{-2s\mathfrak{S}^*_m}  \widehat{\mathfrak{Z}}^{71}_m \left[(\mathcal L_2 \widehat \theta - \widehat{\zeta}_x) - (\mathcal L_2 \widehat \theta - \widehat{\zeta}_{x})_{xx}  \right] ,
	\end{align}
	and 
	\begin{align*}
		\widehat h_1 = e^{-2 s \mathfrak{S}^*_m } 
		 \widehat{\mathfrak{Z}}^{79}_m \widehat u \mathds{1}_\omega , \quad  	\widehat h_2 = e^{-2 s \mathfrak{S}^*_m } \widehat{\mathfrak{Z}}^{73}_m \widehat w \mathds{1}_\omega .
	\end{align*}
	Then, using the observability inequality \eqref{eq:obs_final-alpha=1}  we have (from the equation \eqref{equa-bilinear})
	\begin{multline}\label{bound_solution_control-1}
		\|e^{s\mathfrak{S}^*_m} \widehat{\mathfrak{Z}}^{-71/2}_m \widehat y\|_{L^2(Q_T)} + \|e^{s\mathfrak{S}^*_m} \widehat{\mathfrak{Z}}^{-71/2}_m \widehat z\|_{L^2(Q_T)}
		+ \|e^{s\mathfrak{S}^*_m} \widehat{\mathfrak{Z}}^{-71/2}_m \widehat p\|_{L^2(Q_T)} \\
			+ \|e^{s\mathfrak{S}^*_m} \widehat{\mathfrak{Z}}^{-71/2}_m \widehat q\|_{L^2(Q_T)} 
	+	\|e^{s\mathfrak{S}^*_m} \widehat{\mathfrak{Z}}^{-79/2}_m\widehat h_1\|_{L^2((0,T)\times \omega)}
	+ \|e^{s\mathfrak{S}^*_m} \widehat{\mathfrak{Z}}^{-73/2}_m\widehat h_2\|_{L^2((0,T)\times \omega)} <+\infty  ,
	\end{multline}
	and this $(\widehat y, \widehat z, \widehat p, \widehat q)$ is unique solution to the linearized system \eqref{sys_linear-1}--\eqref{sys_linear-2} (with $\alpha=1$) in the sense of transposition with the control functions $\widehat h_1$ and $\widehat h_2$.  Moreover, from \eqref{solution-controlled-linear-back-1} it is clear that 
	$$\widehat p(0)=0, \quad \widehat q(0)=0 \quad \text{in } (0,1).$$
	
	\smallskip 
	
At this stage,  we set 
	\begin{align*}
		(y^*, z^*, p^*, q^*) = e^{s \mathfrak{S}^*_m}
		\widehat{\mathfrak{Z}}^{-79/2}_m (\widehat y, \widehat z, \widehat p, \widehat q) ,
	\end{align*} 
and by following the similar argument as in \Cref{prop-null-cont-alpha=0}, one has 
	\begin{align*}
	(y^*, z^*, p^*, q^*) \in \C^0([0,T]; [L^2(0,1)]^4) \cap L^2(0,T; H^2_0(0,1) \times H^1_0(0,1) \times H^2_0(0,1) \times H^1_0(0,1)) . 
\end{align*}
This implies that the functions $(\widehat y, \widehat z, \widehat p, \widehat q, \widehat h_1, \widehat h_2)\in \mathcal E_1$ defined in \eqref{space_E-alpha=1}. The proof is finished.
\end{proof}

\subsubsection{Local null-controllability of the nonlinear system  ($\boldsymbol{\alpha=1}$)} 

The proof of the main result, that is \Cref{thm:main} for the case $\alpha=1$ is exactly similar as we describe for the case $\alpha=0$. We refer Subsection \ref{section-locall-null-alpha=0} for the details.

\smallskip

\section{Concluding remarks and comments}\label{sec:final}
In this work, we have proved the existence of insensitizing controls for a coupled system of fourth- and second-order parabolic PDEs. As usual, this problem is reformulated as a null-control problem for an extended system in cascade form (see \eqref{sys_equiv_1}--\eqref{sys_equiv_2}) but due to the presence of the parameter $\alpha$ in the sentinel \eqref{eq:sentinel}, this system may change its structure and different Carleman tools should be employed for studying the observability of the corresponding adjoint equations. Let us  present a concluding remark concerning the problems addressed in this work. 

\medskip

\textit{Less control than equations}. An important question related to the controllability of coupled systems is: what is (are) the minimum number of control(s) required to accomplish a given task. In the works \cite{Cerpa-Mercado-Pazoto} and \cite{Cerpa_Careno}, it has been shown that for proving the null-controllability of the system
\begin{align}\label{null-control-p}
	\begin{dcases}
		y_t + y_{xxxx} +\gamma y_{xx} + y y_x = z_x  & \text{in } Q_T, \\
		z_t - z_{xx} + \beta z_x =  y_x   & \text{in } Q_T,\\
		y=y_x=z=0        &\text{in } \Sigma_T, \\
		y(0)=y_0, \ \ z(0) =z_0 &\text{in } (0,1),
 	\end{dcases}
	\end{align}
it is needed only one control localized in either equation. 

\smallskip 

In our insensitizing problem we have used two controls, one for each component, but determining if we can reduce its number it is not so clear. As far as we know, there are very few papers devoted to the insensitizing control problems for coupled systems, see \cite{CGG15,CCC16,Cn17}. Similar to  our case, in those works, the original problem is transformed into a control problem for an extended system of four equations (two forward and two backward in time).  In particular, in \cite{CCC16}, the authors have used only one control to prove their result by using a sentinel depending (explicitly) only on one of the components of the system. This is comparable to choosing $\alpha=0$ in our case \eqref{eq:sentinel}. At a first glance, it seems that we can follow the ideas similar to \cite{CCC16} to eliminate the extra control, but in our case, the first order couplings make things difficult. Indeed, recalling our adjoint system \eqref{adj_sys_a_0} (i.e. the case $\alpha=0$), we see that the only way to remove an observation related to $w$ is to differentiate the second equation of \eqref{adj_sys_a_0} (changing $\mathds{1}_{\mathcal O}$ by some suitable smooth approximation) and use the coupling in the first one. By doing so we can obtain a Carleman estimate with localized terms depending on $u$, $\zeta$ and $\theta$. Then, the only way to estimate  $\theta$ locally is by using the second equation of \eqref{adj_sys_a_0} which reintroduces a local term of $w$ (see eq. \eqref{esti_theta_observation}), and we failed! Therefore,  dealing with cascade systems (forward-backward) of the original coupled systems can be tricky and its controllability and observability properties deserve further attention. This fact has been also pointed out in \cite{HSdeT18} in the context of hierarchic control problems.

\appendix

\section{An auxiliary result}

\begin{lemma}\label{Lemma-auxiliary}
	Recall the weight functions $\vphi_m$ and $\widehat \vphi_m$ defined by \eqref{weight_function} and \eqref{max_min} respectively for $\lambda>1$ and $k>m>0$. Then, for  any $s>0$ and $p\in \mathbb N^*$, there exists   $c_0>0$, such that we have 
	\begin{align}\label{good_sign}
		-p s\vphi_m + (p-1) s\widehat \vphi_m \leq \frac{-c_0 s}{t^m(T-t)^m} . 
	\end{align}
\end{lemma}
\begin{proof}
	The proof can be deduced from the explicit expressions of the weight functions. We see	
	\begin{align*}
		&-ps\vphi_m + (p-1)s \widehat \vphi_m \\
		 =& -ps \frac{e^{\lambda(1+\frac{1}{m})k \|\nu\|_{\infty} } - e^{\lambda\big(k\|\nu\|_{\infty} +\nu(x) \big)}   }{t^m(T-t)^m} + (p-1)s \frac{e^{\lambda(1+\frac{1}{m})k \|\nu\|_{\infty} } - e^{\lambda k\|\nu\|_{\infty} }   }{t^m(T-t)^m}\\
		=&  -\frac{s e^{\lambda(1+\frac{1}{m})k \|\nu\|_{\infty} }}{t^m(T-t)^m} + \frac{s e^{\lambda k \|\nu\|_{\infty}}}{t^m(T-t)^m} \left(p\big(e^{\lambda \nu(x)} -1\big) + 1 \right)\\
		= &-\frac{s e^{\lambda k \|\nu\|_{\infty}} }{t^m(T-t)^m} \left(e^{\frac{\lambda}{m} k \|\nu\|_{\infty}} -pe^{\lambda \nu(x)} +p-1 \right).
	\end{align*}
	Thus, for fixed $m>0$ and any $\lambda >0$, one may choose $k>m$ large enough such that there exists some $c_0>0$ verifying that the quantity $\left(e^{\frac{\lambda}{m} k \|\nu\|_{\infty}} -pe^{\lambda \nu(x)} +p-1 \right) \geq c_0 >0$,  $\forall x\in [0,1]$.  Hence, the result \eqref{good_sign} follows.
\end{proof}

\section{Well-posedness  of the linear stabilized KS system}
Let us consider the following coupled system 
\begin{align}\label{system-appendix}
	\begin{dcases}
		y_t + y_{xxxx} +\gamma y_{xx} = z_x +  f_1 & \text{in } Q_T, \\
		z_t - z_{xx} + \beta z_x =  y_x + f_2 & \text{in } Q_T,\\
		y=y_x=z=0        &\text{in } \Sigma_T, \\
		y(0)=y_0 \ \ z(0) =z_0   &\text{in } (0,1),
	\end{dcases}
\end{align}
where $(f_1,f_2)\in [L^2(Q_T)]^2$,  $\gamma>0$ and $\beta$ is any real number. 

Below, we write the standard well-posedness and some regularity result concerning the prototype of coupled system \eqref{system-appendix}. The proof will be omitted. 

\begin{proposition}[Well-posedness \& energy estimate]\label{Prop-appendix-well-posed}
	For any given $(y_0, z_0)\in [L^2(0,1)]^2$ and $(f_1,f_2)\in [L^2(Q_T)]^2$, there exists unique weak solution 
	$$(y,z) \in [\C^0([0,T]; L^2(0,1))]^2 \cap L^2(0,T; H^2_0(0,1)\times H^1_0(0,1))$$
	 to \eqref{system-appendix}, such that it satisfies 
	 \begin{align*}
	&\|(y,z)\|_{[\C^0([0,T]; L^2(0,1))]^2} + \|(y,z)\|_{L^2(0,T; H^2_0(0,1)\times H^1_0(0,1))  } \\
		&\quad \leq C \left( \|(y_0,z_0)\|_{[L^2(0,1)]^2} + \|(f_1,f_2)\|_{[L^2(Q_T)]^2}\right),
	 \end{align*}
 for some $C>0$. 	
\end{proposition}

\begin{proposition}[Regularity]\label{Prop-appendix-regularity}
	For any given $(y_0, z_0)\in H^2_0(0,1)\times H^1_0(0,1)$ and $(f_1,f_2)\in [L^2(Q_T)]^2$, the solution  
	$$(y,z) \in \C^0([0,T];  H^2_0(0,1)\times H^1_0(0,1)) \cap L^2(0,T; H^4(0,1)\times H^2(0,1))$$
	to \eqref{system-appendix}  satisfies 
	\begin{align*}
		&\|(y,z)\|_{\C^0([0,T];  H^2_0(0,1)\times H^1_0(0,1))} + \|(y,z)\|_{L^2(0,T; H^4(0,1)\times H^2_0(0,1))  } \\
		&\quad \leq C \left(\|(y_0,z_0)\|_{ H^2_0(0,1)\times H^1_0(0,1)}+ \|(f_1,f_2)\|_{[L^2(Q_T)]^2}\right),
	\end{align*}
	for some $C>0$. 	
\end{proposition}
A formal proof for the above regularity result can be found in \cite[Prop. 2.1]{Cerpa-Mercado-Pazoto}. 

\section*{Acknowledgements}


The work of the first author is partially supported by the French government research program
 ``Investissements d'Avenir" through the IDEX-ISITE initiative 16-IDEX-0001 (CAP 20-25). The work of the second author has been partially supported by the program ``Estancias Posdoctorales por México para la Formación y Consolidación de las y los Investigadores por México'' and by project A1-S-17475 of CONACyT, Mexico, and by the project IN109522 of DGAPA-UNAM, Mexico.

\bigskip

\bibliographystyle{siam}

\bibliography{ref_ins_ks}
	
\end{document}